\documentclass[10pt,reqno]{article}
\usepackage{geometry}
\usepackage{amsmath, amsfonts, dsfont, stmaryrd}
\usepackage{amsthm}
\usepackage{amsopn}
\usepackage{amssymb}
\usepackage{dsfont}
\usepackage[colorlinks=true]{hyperref}
\usepackage{upgreek}
\usepackage{calc}
\usepackage{accents}
\usepackage{verbatim}
\usepackage{textgreek}
\usepackage[OT1]{fontenc}
\usepackage[dvips]{graphicx}
\usepackage[dvips]{epsfig}
\usepackage{subfigure}
\usepackage{appendix}
\usepackage{enumerate}
\usepackage[square,numbers]{natbib}

\usepackage{stmaryrd}
\SetSymbolFont{stmry}{bold}{U}{stmry}{m}{n}

\newtheorem{theorem}{Theorem}[section]
\numberwithin{theorem}{section}
\newtheorem{lemma}[theorem]{Lemma}
\newtheorem{proposition}[theorem]{Proposition}
\newtheorem{corollary}[theorem]{Corollary}
\theoremstyle{definition}
\newtheorem{definition}[theorem]{Definition}
\newtheorem{assumption}[theorem]{Assumption}
\newtheorem{remark}[theorem]{Remark}
\numberwithin{equation}{section}

\renewcommand{\b}[1]{\boldsymbol{\mathrm{#1}}} 
\newcommand{\ol}[1]{\overline{#1} \!\,} 

\newcommand{\wt}{\widetilde}

\newcommand*\di{\mathop{}\!\mathrm{d}} 
\newcommand*\ii{\mathop{}\!\mathrm{i}} 
\renewcommand{\P}{\mathbb{P}}
\newcommand{\E}{\mathbb{E}}
\newcommand{\R}{\mathbb{R}}
\newcommand{\C}{\mathbb{C}}
\newcommand{\N}{\mathbb{N}}

\newcommand{\e}{\mathrm{e}}
\newcommand{\mf}[1]{\mathfrak{#1}}
\newcommand{\dt}[1]{\accentset{\approx}{#1}}

\newcommand{\p}[1]{(#1)} 
\newcommand{\pb}[1]{\bigl(#1\bigr)}
\newcommand{\pB}[1]{\Bigl(#1\Bigr)}

\newcommand{\pa}[1]{\left(#1\right)}

\newcommand{\qb}[1]{\bigl[#1\bigr]}
\newcommand{\qB}[1]{\Bigl[#1\Bigr]}
\newcommand{\qbb}[1]{\biggl[#1\biggr]}

\newcommand{\qa}[1]{\left[#1\right]}

\newcommand{\hbb}[1]{\biggl\{#1\biggr\}}

\newcommand{\abs}[1]{\lvert #1 \rvert} 

\newcommand{\absB}[1]{\Bigl\lvert #1 \Bigr\rvert}
\newcommand{\absbb}[1]{\biggl\lvert #1 \biggr\rvert}
\newcommand{\absBB}[1]{\Biggl\lvert #1 \Biggr\rvert}
\newcommand{\absa}[1]{\left\lvert #1 \right\rvert}

\newcommand{\norm}[1]{\lVert #1 \rVert} 

\DeclareMathOperator{\tr}{Tr}

\DeclareMathOperator{\re}{Re}
\DeclareMathOperator{\im}{Im}

\title{Local law and Tracy--Widom limit for sparse stochastic block models}

\author{Jong Yun Hwang \footnote{Department of Mathematical Sciences, KAIST, Daejeon, 34141, Korea
		\newline email: \texttt{jyh9006@kaist.ac.kr}} 
	\and Ji Oon Lee\footnote{Department of Mathematical Sciences, KAIST, Daejeon, 34141, Korea
		\newline email: \texttt{jioon.lee@kaist.edu}}
	\and Wooseok Yang \footnote{Department of Mathematical Sciences, KAIST, Daejeon, 34141, Korea
		\newline email: \texttt{ws.yang@kaist.ac.kr}}}
		
\date{\today}

\begin{document}

\maketitle

\begin{abstract}
We consider the spectral properties of sparse stochastic block models, where $N$ vertices are partitioned into $K$ balanced communities. Under an assumption that the intra-community probability and inter-community probability are of similar order, we prove a local semicircle law up to the spectral edges, with an explicit formula on the deterministic shift of the spectral edge. We also prove that the fluctuation of the extremal eigenvalues is given by the GOE Tracy--Widom law after rescaling and centering the entries of sparse stochastic block models. Applying the result to sparse stochastic block models, we rigorously prove that there is a large gap between the outliers and the spectral edge without centering.
\end{abstract}

\section{Introduction}\label{sec:intro}

Understanding an underlying network structure is one of the key problems in data science. Many real world data sets can be viewed as networks of interacting nodes, and a common approach to analyze the structure of the network is to find which nodes share similar properties so that they can be grouped into a community. Clustering, or community detection, to recognize such communities from given data sets is thus a natural and fundamental problem. 

Community detection problem is vital in understanding the real-world networks, such as biological networks and social networs. In biology and bioinformatics, community detection appears in finding functional modules in protein-protein interaction networks (\cite{CY06}), functional mapping of metabolic network (\cite{GA05a, GA05b}), analyzing gene expression data (\cite{CSC07, JTZ04}) and more. Community detection problems also naturally arise in social networks. The ``friendships'' networks of Facebook, the online social network, was studied, including anonymous Facebook users in one hundred American universities  (\cite{TKMP11, TMP12}). Communities of the network were identified, and it was found that the community structure depends strongly on their offline network, such as class year or House affiliation. There have been studies on community structures of other social networks, such as scientific collaboration networks (\cite{GN02, Newman01}).

The stochastic block model (SBM) is one of the simplest models of the network with communities. First appeared in the study of social networks (\cite{HLL83}), it consists of $N$ vertices partitioned into disjoint subsets $C_1, \dots, C_K$ and a $K \times K$ symmetric matrix $P$ of edge probabilities. The model appears in various fields of study, and numerous results have been obtained for community detection in SBM, including various algorithms (\cite{GN02,KMM13,HWX16,GV16,AS18}), phase transitions (\cite{ABH16}), and fundamental limits (\cite{RCY11}). We refer to \cite{Abbe17} for history and recent developments of the community detection problem and stochastic block models.

The spectral method is one of the most well-known approaches for the community detection of SBM. In this method, the adjacency matrix of a given graph is considered, whose extremal eigenvalues and corresponding eigenvectors contain the information on the ground truth of the model. In the simplest example of an SBM with two communities of equal size, if we denote the $N \times N$ adjacency matrix by $A$, the probability matrix $P$ is a $2 \times 2$ matrix, and the expected adjacency matrix $\E A$ has four blocks, i.e.,
\[
\E A = \left(
\begin{array}{c|c}
P_{11} & P_{12} \\
\hline
P_{21} & P_{22}
\end{array}
\right).
\]
If $P_{11}=P_{22} = p_s$ and $P_{12}=P_{21} = p_d$, the first two eigenvalues of $\E A$ are $N(p_s+p_d)/2$ and $N(p_s-p_d)/2$, and the eigenvalue $0$ has multiplicity $N-2$. If the difference $A - \E A$ is small, then the eigenstructure of $A$ is not much different from that of $\E A$, and one can recover the community structure from the second eigenvector of $A$. The spectral method is also useful in determining the number of communities $K$ when it is not known a priori (\cite{BS16,Lei16}).

In the spectral method, the perturbation $H := A - \E A$, which called centered SBM, is a random matrix, and its property, especially the behavior of its largest eigenvalue, can be precisely predicted by results from random matrix theory when $P$ does not depend on $N$. However, $H$ is different from Wigner matrices in two aspects: (1) the variances of entries are not identical, and (2) the matrix is sparse. (See Assumption \ref{assumption} for more detail on the sparsity.) The first aspect is due to that the intra-community probability $p_s$ and the inter-community probability $p_d$ are different from each other and hence the random variables have different variances. The second aspect is common in many real data, since the expected degree is much smaller than $N$ and the edge probability decays as $N$ grows. For sparse random matrices with identical off-diagonal entries, which correspond to Erd\H{o}s--R\'enyi graphs, the spectral properties were obtained in \cite{EKYY13,EKYY12,LS18}. One of the most notable aspects of sparse random matrices is that the deterministic shift of their largest eigenvalues are much larger than the size of the Tracy--Widom fluctuation. Thus, as discussed in Remark 2.14 of \cite{LS18}, in cases where the intra-community probability $p_s$ and the inter-community probability $p_d$ are both small and close to each other, we can predict that the algorithms for the community detection should reflect the shift of the largest eigenvalues if $p_s, p_d \ll N^{-1/3}$. However, to our best knowledge, it has not been proved for sparse SBM.

In this paper, we consider the spectral properties of sparse SBM with $K$ communities. We assume that the communities are of equal size, or balanced, with $P_{ii} = p_s$ and $P_{ij} = p_d$ for $i \neq j$. We further assume that the model is moderately sparse as in Assumption \ref{assumption}. 
Our main contributions are 
\begin{enumerate}[(1)]
	\item proof of local semicircle law for the centered sparse SBM that is believed to be optimal up to the edge of the spectrum (Theorems~\ref{thm:locallaw} and~\ref{thm:weak local law}), 
	\item proof of the Tracy--Widom limit for the shifted, rescaled largest eigenvalue of the centered sparse SBM (Theorem~\ref{thm:TWlimit}), and 
	\item application to the (non-centered) sparse SBM (Theorem~\ref{thm:perturbation}).
\end{enumerate}

The local semicircle law, the estimates on the resolvent of Wigner type matrices, has been the starting point in the local spectral analysis of Wigner matrices (\cite{EYY12a,EYY12}) and Erd\H{o}s--R\'enyi graphs (\cite{EKYY13,EKYY12}). We follow the classical strategy based on Schur complement formula and self-consistent equations as in \cite{EYY12a,EYY12}, which leads us to a weak local law for the resolvent entries (Theorem~\ref{thm:weak local law}). Since the weak local law is not sufficient for the proof of other properties such as the Tracy--Widom fluctuation of the extremal eigenvalues, we improve it to prove the strong local law for the normalized trace of the resolvent (Theorem~\ref{thm:locallaw}),  which is optimal up to the edge of the spectrum, by adapting the strategy of \cite{LS18,Hwang2018}.

The proof of the Tracy--Widom limit of the extremal eigenvalues is based on the Green function comparison method that utilizes a continuous interpolation as in \cite{LS18}. With the continuous flow, we can track the change of the normalized trace over time, which is offset by the (deterministic) shift of the spectral edge.

When applying the local spectral properties of Wigner matrices or Erd\H{o}s--R\'enyi graphs to the SBM, one of the main technical challenges stems from that the entries in the SBM are not identically distributed, especially the means of the entries are not equal, and thus the results from random matrix theory are not directly applicable. While the difficulty can be overcome by algorithms as in \cite{GN02}, it requires a priori knowledge on the number of clusters $K$. In this paper, we handle the issue by proving that there is a gap of order $1$ between $K$-th largest eigenvalue and $(K+1)$-st one, which is much larger than the gap between the $(K+1)$-st and the $(K+2)$-nd, when the number of clusters is $K$. This results justifies the use of the spectral method for community detection even when the SBM is sparse.

In the proof of the local law, as in \cite{LS18}, we choose a polynomial $P(m)$ of the normalized trace $m$ of the Green function, based on a recursive moment estimate. However, the fluctuation averaging mechanism, which was intrinsic in the analysis of Erd\H{o}s--R\'enyi graph, is much more complicated for the SBM due to the lack of the symmetry. Technically, it means that we need to separate the off-diagonal elements $H_{ij}$ into two cases in the cumulant expansion - one with when $i$ and $j$ are in the same community and the other when $i$ and $j$ are in different communities. With the separation, we need to consider a more complicated polynomial $P$ than the Erd\H{o}s--R\'enyi case in \cite{LS18}, and the analysis more involved also for the limiting distribution.

This paper is organized as follows: In Section \ref{sec:def}, we introduce our model and state the main results. In Section \ref{sec:strategy}, we describe the outline of our proof, with some important properties of the deterministic refinement of Wigner's semicircle law given in Appendix \ref{sec:rho wt m}. In Section \ref{sec:weak local law}, we prove a weak local semicircle law, which is used as an a priori estimate in the proof of our main results. Some technical lemmas in the proof are proved in Appendix \ref{app:weak}. In Section \ref{sec:local law}, we prove the strong local law by using the recursive moment estimates whose proofs are given in Appendix \ref{sec:recursive}, and the detailed proof of the bound on the operator norm $||H||$ provided in Appendix \ref{app:Hnorm}.   In Section \ref{sec:TW}, we prove the Tracy--Widom limit of the largest eigenvalue using the Green function comparison method, with some proofs of lemmas presented in Appendix \ref{app:TW}.

\begin{remark}[Notational remark]
	We use the symbols $O( \cdot )$ and $o( \cdot )$ for the standard big-O and little-o notation. The notations $O$, $o$, $\ll$, $\gg$ always refer to the limit $N \to \infty$ unless otherwise stated. Here, the notation $a \ll b$ means $a = o(b)$. We use $c$ and $C$ to denote positive constants that do not depend on $N$. Their values may change from line to line. For summation index, we use $\mf{i} \sim \mf{j}$ if $\mf{i}$ and $\mf{j}$ are within the same group. We write $a \sim b$ if there is $C \geq 1$ such that $C^{-1}|b| \leq |a| \leq C|b|$.  Throughout this paper we denote $z = E + \ii \eta \in \C^+$ where $E=\re z$ and $\eta=\im z$. 
\end{remark}

\section{Definition and main results}\label{sec:def}
\subsection{Models and notations} 

Let $H$ be an $N\times N$ symmetric matrix with $K^2$ blocks of same size. The blocks are based on the partition of the vertex set $[N] := \{1,2,\dots,N\}$,
\begin{equation}
[N] = V_1 \cup V_2 \cup \dots \cup V_K,
\end{equation}
where $|V_i|=N/K$. For $i,j \in \{1,2,\dots N\}$, we may consider two types of the edge probability $H_{ij}$, depending on whether $i$ and $j$ are within the same vertex set $V_\ell$ or not. More precisely, we consider the sparse block matrix model satisfying the following assumption:

\begin{assumption}[Balanced generalized sparse random matrix] \label{assumption} 
	Fix any small $\phi>0$. We assume that $H=(H_{ij})$ is a real $N \times N$ block random matrix with $K$ balanced communities, whose diagonal entries are almost surely zero and whose off-diagonal entries are independently distributed random variables, up to symmetry constraint $H_{ij}=H_{ji}$. We suppose that each $H_{ij}$ satisfies the moment conditions
	\begin{align}\label{eq:moments}
	\E H_{ij}=0, \qquad \E|H_{ij}|^2=\sigma_{ij}^2,\qquad\E|H_{ij}|^k \leq \frac{(Ck)^{ck}}{Nq^{k-2}},\qquad (k\geq 2),
	\end{align}
	with sparsity parameter $q$ satisfying
	\begin{align} \label{eq:sparsity}
	N^{\phi}\leq q \leq N^{1/2}.
	\end{align}
	Here, we further assume the normalization condition
	\begin{equation}
	\sum_{i=1}^{N}\sigma_{ij}^2 =1.
	\end{equation}
\end{assumption}
We denote by $\kappa_{ij}^{(k)}$ the $k$-th cumulant of the random variables $H_{ij}$. Under the moment condition \eqref{eq:moments},
\begin{align}
\kappa_{ij}^{(1)}=0, \qquad	|\kappa_{ij}^{(k)}|\leq \frac{(2Ck)^{2(c+1)k}}{Nq^{k-2}},\quad\qquad (k\geq 2).
\end{align}
For our model with the block structure, we abbreviate $\kappa_{ij}^{(k)}$ as
\begin{align} 
\kappa_{ij}^{(k)}=
\begin{cases}
\kappa_{s}^{(k)} & \text{ if } i \text{ and } j \text{ are within same community}, \\
\kappa_{d}^{(k)} & \text{ otherwise }.
\end{cases}
\end{align}
We will also use the normalized cumulants, $s^{(k)}$, by setting
\begin{align}
s_{(\b{\cdot})}^{(1)}:=0, \quad s_{(\b{\cdot})}^{(k)}:=Nq^{k-2}\kappa_{(\b{\cdot})}^{(k)}, \quad (k\geq 2).
\end{align}
We notice that we assume $H_{ii}=0$ a.s., although this condition can be easily removed. We also remark that we recover the sparse random matrix with i.i.d. entries (e.g., the adjacency matrix of a sparse Erd\H{o}s--R\'enyi graph) for the choice $K=1$.
For convenience, we define the parameters $\zeta$ and $\xi^{(4)}$ as  
\begin{equation}
\zeta : = \frac{s_s^{(2)} - s_d^{(2)}}{K}=\frac{N(\kappa_s^{(2)} - \kappa_d^{(2)})}{K} , \qquad \xi^{(4)} :=  \frac{s_s^{(4)}+(K-1)s_d^{(4)}}{K}.
\end{equation}

The prominent example of a balanced generalized sparse random matrix is the case where $H_{ij}$ is given by the Bernoulli random variable $A_{ij}$ with probability $p_s$ or $p_d$, depending on whether $i$ and $j$ are within same vertex set or not, respectively. For this reason, in this paper, we oftentimes use the term `centered generalized stochastic block model' or `cgSBM' as a representative of the balanced generalized sparse random matrix.	

In the rest of this subsection, we introduce some notations of basic definitions. 
\begin{definition}[High probability events]
	We say that an $N$-dependent event $\Omega \equiv \Omega^{(N)}$ holds with high probability if for any (large) $D>0$,
	\[\P(\Omega^c) \leq N^{-D}, \]
	for $N\geq N_0(D)$ sufficiently large.
\end{definition}

\begin{definition}[Stochastic domination] Let $X \equiv X^{(N)}, Y \equiv Y^{(N)}$ be $N$-dependent non-negative random variables. We say that $X$ stochastically dominates $Y$ if, for all small $\epsilon>0$ and large $D>0$,
	\begin{align}
	\mathbb{P}(X^{(N)}>N^\epsilon Y^{(N)})\leq N^{-D},
	\end{align}
	for sufficiently large $N\geq N_0(\epsilon,D)$, and we write $X\prec Y$. When $X^{(N)}$ and $Y^{(N)}$ depend on a parameter $u \in U$, then we say $X(u) \prec Y(u)$ uniformly in $u\in U$ if the threshold $N_0(\epsilon,D)$ can be chosen independently of $u$.  
\end{definition}
Throughout this paper, we choose $\epsilon>0$ sufficiently small. (More precisely, it is smaller than $\phi/10$, where $\phi >0$ is the fixed parameter in Assumption \ref{assumption} below.)

\begin{definition}[Stieltjes transform]
	For given a probability measure $\nu$, we define the Stieltjes transforms of $\nu$ as
	\[m_\nu(z):= \int \frac{\nu(\di x)}{x-z}, \qquad (z\in \C^+)  \] 
\end{definition}
For example, the  Stieltjes transform of the $\emph{semicircle measure}$,
\[ \varrho(\di x):=\frac{1}{2\pi}\sqrt{(4-x^2)_+} \di x, \]
is given by
\[ m_{sc}(z) = \int \frac{\varrho(\di x)}{x-z} = \frac{-z + \sqrt{z^2 -4}}{2}, \]
where the argument of $\sqrt{z^2-4}$ is chosen so that $m_{sc}(z) \in \C^+$ for $z \in \C^+$ and $\sqrt{z^2-4} \sim z$ as $z \to \infty $.
Clearly, we have 
\[ m_{sc}(z) + {m_{sc}(z)}^{-1}+z =0.  \] 
\begin{definition}[Green function(Resolvent)]
	Given a real symmetric matrix $H$ we define its $\emph{Green function}$ or $\emph{resolvent}$, $G(z)$, and the normalized trace of its Green function, $m^H$, by
	\begin{align}\label{greenfunction}
	G^H(z)\equiv G(z) := (H-zI)^{-1}, \qquad  m^H(z)\equiv m(z) := \frac{1}{N}\tr G^H(z), 
	\end{align}
	where $z = E + \ii \eta \in \C^+$ and $I$ is the $N\times N$ identity matrix. 
	
\end{definition}
Denoting by $\lambda_1 \geq \lambda_2 \geq \cdots \geq \lambda_N$ the ordered eigenvalues of $H$, we note that $m^H$ is the Stieltjes transform of the empirical eigenvalue measure of $H$, $\mu^H$, defined as 
\[
\mu^H := \frac{1}{N}\sum_{i=1}^{N}\delta_{\lambda_i}.
\]

Finally, we introduce the following domains in the upper-half plane
\begin{align}
\mathcal{E} &:=\{ z=E+\ii \eta \in \C^+ : |E|<3, 0 <\eta \leq 3\} ,\label{def:domain E} \\
\mathcal{D}_\ell &:= \{z=E+\ii \eta \in \C^+ : |E|<3, N^{-1+\ell} <\eta \leq 3\} \label{def:domain D}.
\end{align}

\subsection{Main results}

Our first main result is the local law for  $m^H$, the normalized trace of $G^H(z)$, up to the spectral edges.
\begin{theorem}[Strong local law] \label{thm:locallaw}
	Let $H$ satisfy Assumption \ref{assumption} with $\phi>0$. Then, there exist an algebraic function $\wt{m}:\mathbb{C}^+ \rightarrow \mathbb{C}^+$ and the deterministic number $2 \leq L < 3$
	such that the following hold:
	\begin{enumerate}
		\item[(1)] The function $\wt{m}$ is the Stieltjes transform of a deterministic probability measure $\wt{\rho}_t$, i.e., $\wt{m}(z) = m_{\wt{\rho}}(z).$ The measure $\rho$ is supported on $[-L, L]$ and $\wt{\rho}$ is absolutely continuous with respect to Lebesgue measure with a strictly positive density on $(-L, L)$.
		\item[(2)] The function $\wt{m}\equiv\wt{m}(z)$, $z\in\mathbb{C}^+$, is a solution to the polynomial equation
		\begin{align}
		P_{1,z}(\wt{m}) &:=1 + z\wt{m}+\wt{m}^2 +q^{-2}\Big(\frac{s_s^{(4)}+(K-1)s_d^{(4)}}{K}\Big)\wt{m}^4 \nonumber\\ 
		&=1 + z\wt{m}+\wt{m}^2 +q^{-2}\xi^{(4)}\wt{m}^4=0.
		\end{align}
		\item[(3)] The normalized trace $m(z)$ of the Green function $G(z)$ satisfies the local law
		\begin{align}\label{thm:stronglaw}
		|m(z) -\wt{m}(z)| \prec \frac{1}{q^2}+\frac{1}{N\eta},
		\end{align}
		uniformly on the domain $\mathcal{E}$.
	\end{enumerate}
\end{theorem}

The function $\wt{m}$ was first introduced in \cite{LS18} to consider a correction term to the semicircle measure in the sparse setting. Some properties of probability measure $\wt \rho$ and its Stieltjes transform $\wt m$ are collected in Lemma \ref{lemma:rho_t}.

From the local law in \eqref{thm:stronglaw}, we can easily prove the following estimates on the \emph{local density of states of $H$}. For $E_1 < E_2$ define 
\[
\mf{n}(E_1,E_2) := \frac{1}{N}|\{i : E_1 < \lambda_i < E_2 \}|, \qquad n_{\wt \rho}(E_1, E_2) := \int_{E_1}^{E_2} \wt\rho(x) \di x.
\]

\begin{corollary}[Integrated density of states]\label{coro:local density}
	Suppose that $H$ satisfies Assumption \ref{assumption} with $\phi >0$. Let $E_1, E_2 \in \R$, $E_1 < E_2 $. Then,
	\begin{equation}
	|\mf{n}(E_1, E_2)-n_{\wt \rho}(E_1, E_2)| \prec \frac{E_1-E_2}{q^2} + \frac{1}{N}.
	\end{equation}
\end{corollary}
Corollary \ref{coro:local density} easily follows from Theorem \ref{thm:locallaw} by applying the Helffer--Sj\"ostrand calculus. We refer to Section 7.1 of \cite{EKYY13_2} for more detail.

The proof of Theorem~\ref{thm:locallaw} is based on the following a priori estimates on entries of the resolvents, which we call the weak local semicircle law. While the weak law for $m$ is indeed weaker than the strong local law, Theorem \ref{thm:locallaw}, we have here an entrywise law, which is believed to be optimal. 

\begin{theorem}[Weak local semicircle law]\label{thm:weak local law}
	Suppose $H$ satisfies Assumption \ref{assumption}. Define the spectral parameter $\psi(z)$ by 
	\[
	\psi(z) := \frac{1}{q} + \frac{1}{\sqrt{N\eta}}, \qquad \qquad (z=E +\ii \eta).
	\]
	Then for any sufficiently small $\ell$,	the events 
	\begin{equation}\label{thm:weak law 1}
	\max_{i\neq j}|G_{ij}(z)|\prec  \psi(z)
	\end{equation}
	\begin{equation}\label{thm:weak law 2}
	\max_{i\neq j}|G_{ii}(z)-m|\prec \psi(z),
	\end{equation}
	and
	\begin{equation}\label{thm:weak law 3}
	|m(z)-m_{sc}(z)|\prec \frac{1}{\sqrt{q}} + \frac{1}{(N\eta)^{1/3}}
	\end{equation}
	hold uniformly on the domain $\mathcal{D}_\ell$.
\end{theorem}

\begin{remark}
	We can extend Theorem \ref{thm:weak local law} to domain $\mathcal{E}$ as in the proof of Theorem 2.8 in \cite{EKYY13}. However, we do not pursue the direction in this paper.
\end{remark}

An immediate consequence of \eqref{thm:weak law 2} of Theorem \ref{thm:weak local law} is the complete delocalization of the eigenvectors.
\begin{corollary}
	Suppose that $H$ satisfies Assumption \ref{assumption} with $\phi >0$. Denote by $(u_i^H)$ the $\ell^2$-normalized eigenvectors of $H$. Then,
	\begin{equation}
	\max_{1\leq i\leq N}\norm{u_i^H}_\infty \prec \frac{1}{\sqrt{N}}.
	\end{equation}
\end{corollary}

For the proof, we refer to the proof of Corollary 3.2 of \cite{EYY12a}.

Together with the weak local semicircle law, a standard application of the moment method yields the following weak bound on $\norm{H}$ ; see~\emph{e.g.} Lemma 4.3 of \cite{EKYY13} and Lemma 7.2 of \cite{EYY12a}.
\begin{lemma}\label{thm:weak matrix norm}
	Suppose that $H$ satisfies Assumption \ref{assumption} with $\phi >0$. Then,
	\begin{equation}\label{eq:weak matirx norm}
	\abs{\norm{H} -2} \prec \frac{1}{q^{1/2}}.
	\end{equation}
\end{lemma}

From the strong law, we can sharpen the estimate \eqref{eq:weak matirx norm} by containing the deterministic refinement to the semicircle law.

\begin{theorem}\label{thm:matrix norm}
	Suppose that $H$ satisfies Assumption \ref{assumption} with $\phi >0$. Then,
	\begin{equation}
	\abs{\norm{H} -L} \prec \frac{1}{q^4} + \frac{1}{N^{2/3}},
	\end{equation}
	where $\pm L$ are the endpoints of the support of the measure $\wt\rho$ given by
	\begin{equation}\label{eq:L}
	L= 2+\frac{\xi^{(4)}}{q^2} +O(q^{-4}).
	\end{equation}
\end{theorem}

\begin{figure}[t]
	\centering
	\subfigure[]{
		\includegraphics[width=0.3\textwidth]{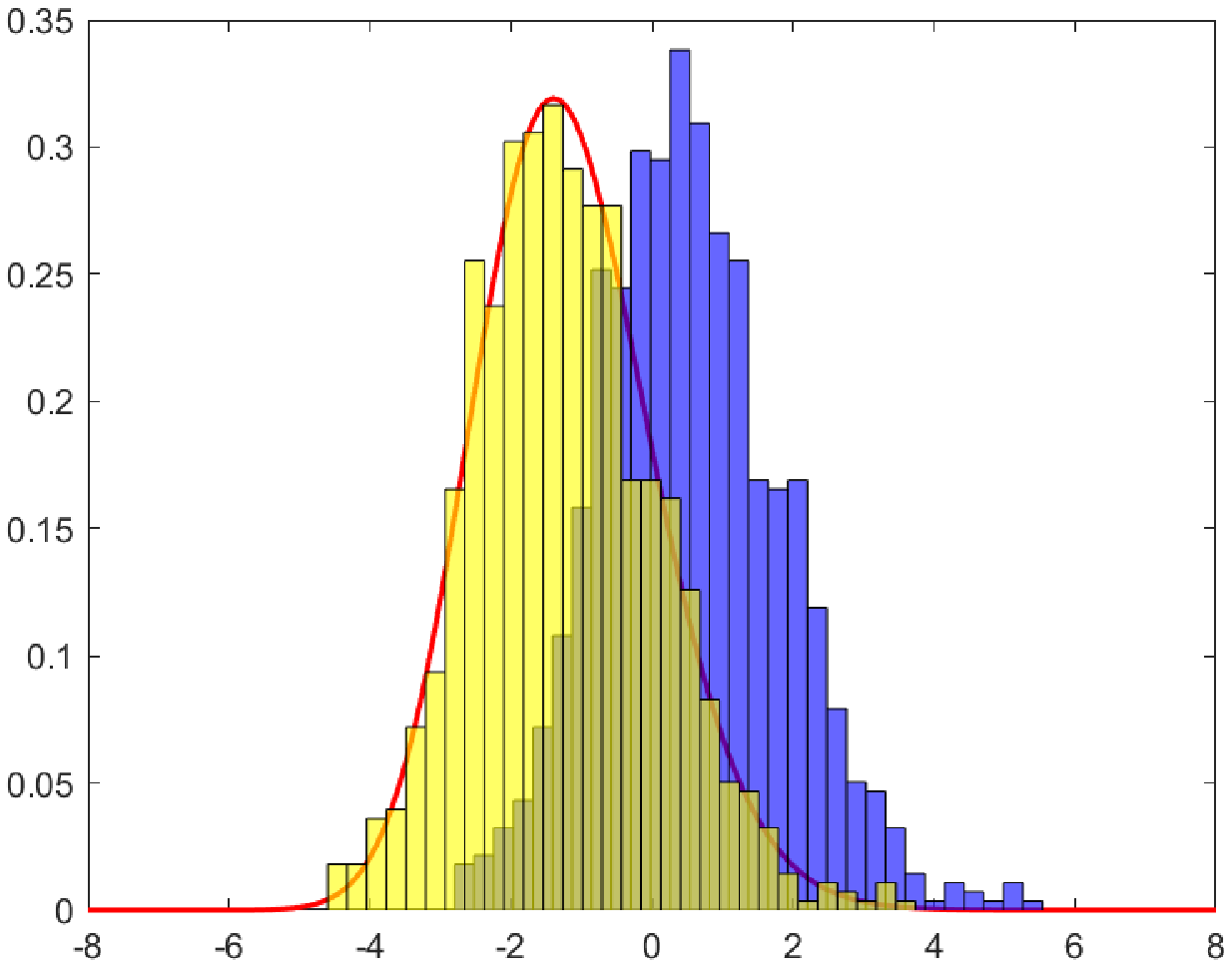}
		\label{fig:0.3}
	}
	\centering
	\subfigure[]{
		\includegraphics[width=0.3\textwidth]{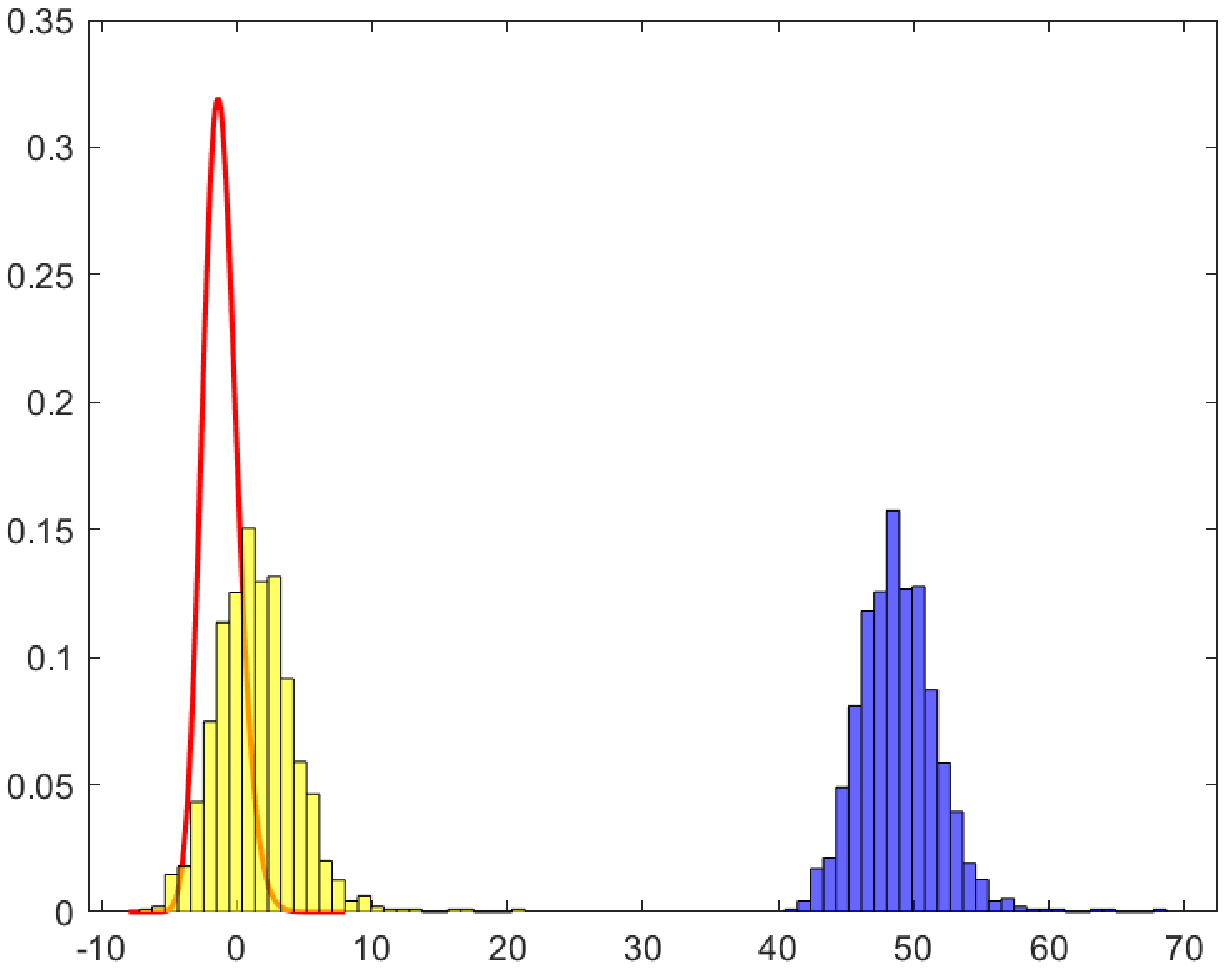}
		\label{fig:0.009}
	}
	\centering
	\subfigure[]{
		\includegraphics[width=0.3\textwidth]{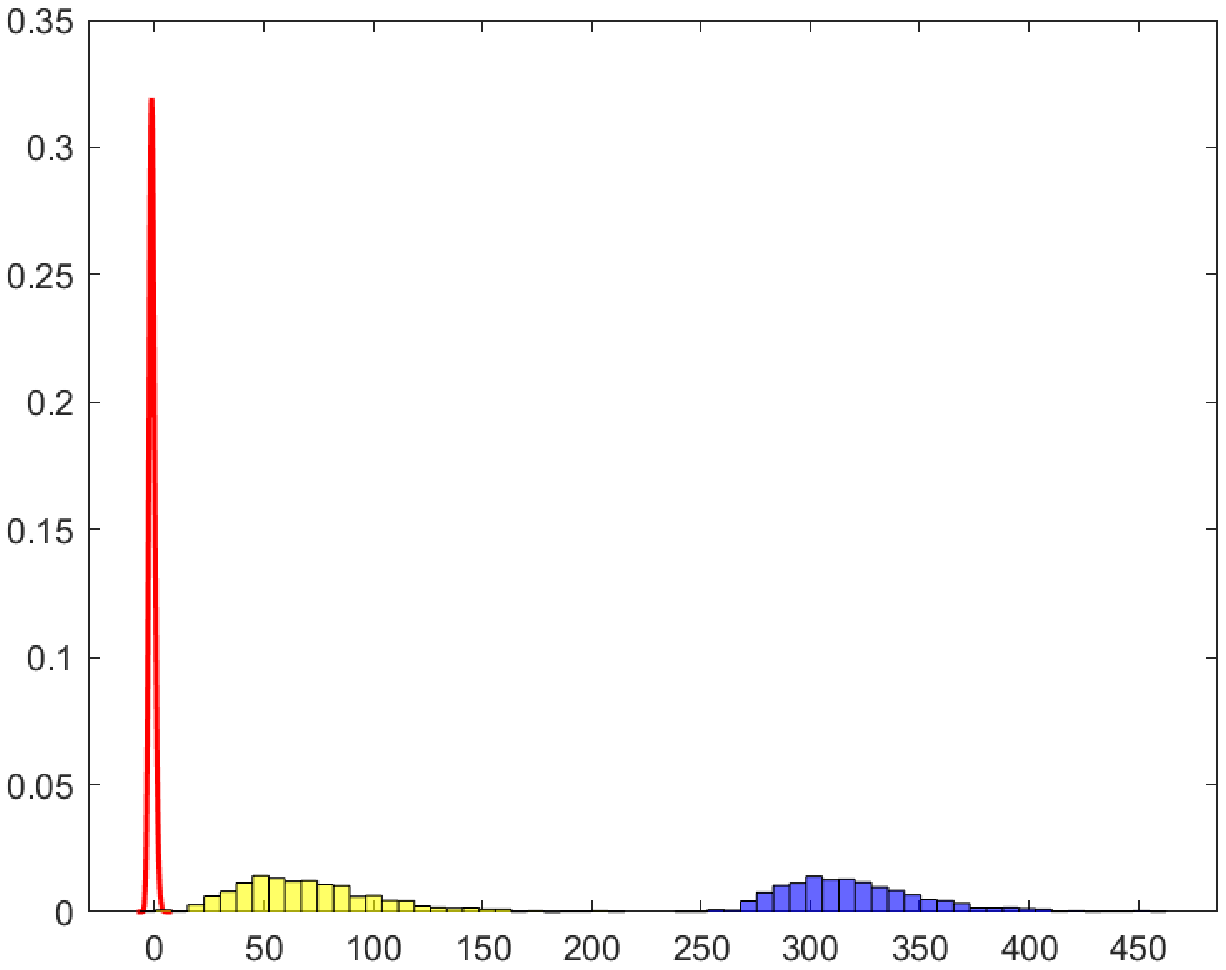}
		\label{fig:0.001}
	}
	\caption{ESD of the largest eigenvalues shifted by 2 (blue) and by $L$ (yellow) plotted against the Tracy--Widom law (red line): cgSBM with $N=27000$, $K=3$ and (a) $p_s=0.03, p_d=0.01 $ ; (b) $p_s=0.009, p_d=0.006$ ; (c) $p_s=0.002, p_d=0.001$ }
	\label{fig:p>0.05}
\end{figure}

Our last main result states that the fluctuations of the rescaled largest eigenvalue of the centered generalized stochastic block model are given by the Tracy--Widom law when the sparsity parameter $q$ satisfies $q \gg N^{1/6}$.
\begin{theorem}\label{thm:TWlimit}
	Suppose that $H$ satisfies Assumption \ref{assumption} with $\phi>1/6.$ Denote by $\lambda_1^{H}$ the largest eigenvalue of $H$. Then
	\begin{align}
	\lim_{N\rightarrow \infty}\mathbb{P}\Big( N^{2/3}\big(\lambda_1^{H} - L\big)\leq s\Big) = F_1(s)
	\end{align}
	where $L$ is given in \eqref{eq:L} and $F_1$ is the cumulative distribution function of the GOE Tracy--Widom law.
\end{theorem}

In Figure \ref{fig:p>0.05}, we plot histograms of shifted and rescaled largest eigenvalues of the $27000 \times 27000$ cgSBM against the Tracy--Widom distribution when $p_s$ and $p_d$ are (a) between $N^{-1/3}$ and $N^{-2/3}$, (b) between $N^{-2/3}$ and $N^{-7/9}$, and (c) less then $\log N /N$. The red line is the Tracy--Widom distribution and the blue histogram is the eigenvalue histogram shifted by 2 and the yellow one is histogram shifted by $L$. In the case (a), we see that the deterministic shift given by $L-2=O(q^{-2})$ is essential so that the empirical eigenvalue distribution of the sample matrices follows the Tracy--Widom distribution when we shift it by $L$ instead of 2.

In (b), Figure \ref{fig:0.009} shows that the empirical distribution of the sample matrices does not follow the Tracy--Widom law. In \cite{HLY17}, it was shown that for Erd\H{o}s--R\'enyi graph, there exists a transition from Tracy--Widom to Gaussian fluctuations when $p \sim N^{-2/3}$. We expect that the fluctuation of extreme eigenvalues of cgSBM will also follow the Gaussian distribution in this regime. Finally, in (c) where $p_s, p_d \ll \frac{\log N}{N}$, cgSBM contains an isolated vertex almost surely and thus is disconnected. Due to this disconnectedness, properties of a cgSBM will be entirely changed. In this case, as shown in Figure \ref{fig:0.001}, the empirical distribution follows neither the Tracy--Widom distribution nor the Gaussian distribution.

\subsection{Applications to the adjacency matrix of the sparse SBM}

We first introduce the Weyl's inequality.
\begin{theorem}[Weyl's inequality]
	Let $A$ and $B$ be symmetric matrices of size $N$, where $A$ has eigenvalues $\lambda_N^A \leq \dots \leq \lambda_1^A$, $B$ has eigenvalues $\lambda_N^B \leq \dots \leq \lambda_1^B$ and $A+B$ has eigenvalues  $\lambda_N^{A+B} \leq \dots \leq \lambda_1^{A+B}$. Then the following inequality holds.
	\begin{align}
	{\lambda _{j}^A+\lambda _{k}^B\leq \lambda _{i}^{A+B} \leq \lambda _{r}^A+\lambda _{s}^B}, \qquad \quad { ( r+s-1\,\leq \,i\,\leq \, j+k-N)}.
	\end{align}
\end{theorem}

Consider an adjacency matrix of the SBM with $N$ vertices and $K$ balanced communities. To make bulk eigenvalues lie in an order one interval, we may rescale this matrix ensemble and we are led to the following random matrix ensemble. Let $A$ be a real symmetric $N \times N$ matrix whose entries, $A_{ij}$,  are independent random variables satisfy
\begin{align}\label{def:A}
\P( A_{ij} =\frac{1}{\sigma} )=
\begin{cases}
p_{s} & (i \sim j) \\
p_{d} & (i\not\sim j)
\end{cases}, \qquad
\P( A_{ij} =0 )=
\begin{cases}
1-p_{s} & (i \sim j) \\
1-p_{d} & (i\not\sim j)
\end{cases},\qquad
\P( A_{ii} =0 )=1,
\end{align}
where $\sigma^2:=\frac{N}{K}p_s(1-p_s) + \frac{N(K-1)}{K}p_d(1-p_d)$.
Then we can get $\wt A:= A-\E A$ which is a centered matrix obtained from $A$. Here, $\wt A_{ij}$ have the distribution 
\begin{align}\label{def: wt A}
&\P(\wt A_{ij} =\frac{1-p_s}{\sigma} )=p_{s}, \qquad \P(\wt A_{ij} =-\frac{p_s}{\sigma} )=1-p_{s}, \qquad  (i \sim j) \notag \\
&\P(\wt A_{ij} =\frac{1-p_d}{\sigma} )=p_{d}, \qquad \P(\wt A_{ij} =-\frac{p_d}{\sigma} )=1-p_{d},  \qquad(i \not\sim j), \qquad\P(\wt A_{ii} =0 )=1.  
\end{align}
When $p_s$ and $p_d$ have order $N^{-1+2\phi}$, it can be easily shown that $\wt A$ satisfies Assumption \ref{assumption} with $q \sim N^{\phi}$. Now we apply Weyl's inequality to $A= \E A + \wt A$ and can get the following theorem.

\begin{theorem}\label{thm:perturbation}
	Fix $\phi>0$. Let $A$ and $\wt A$ satisfy \eqref{def:A} and \eqref{def: wt A} with $N^{-1+2\phi} \leq p_s, p_d \leq N^{-2\phi}$. 
	Then there is a constant $c$ such that 
	\begin{align}
	\lambda_N^{A} \leq \dots \leq \lambda_{K+1}^{A}\leq \lambda_1^{\wt A} \leq 2+c  \leq \lambda_{K}^A \leq \dots \leq \lambda_1^A.
	\end{align}
	with high probability. In other words, there is an order one gap between the $K$ largest eigenvalues and other eigenvalues.
\end{theorem}

\begin{figure}[h]
	\centering
	\subfigure[]{
		\includegraphics[width=0.43\textwidth]{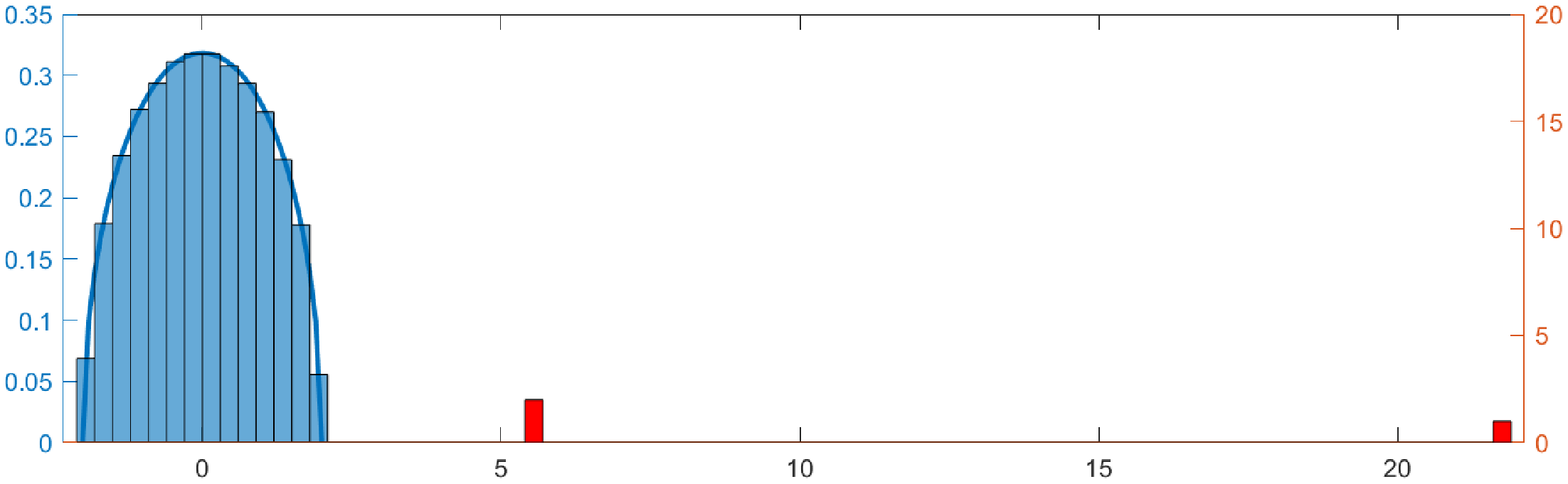}
		\label{fig:3group}
	}
	\centering
	\subfigure[]{
		\includegraphics[width=0.43\textwidth]{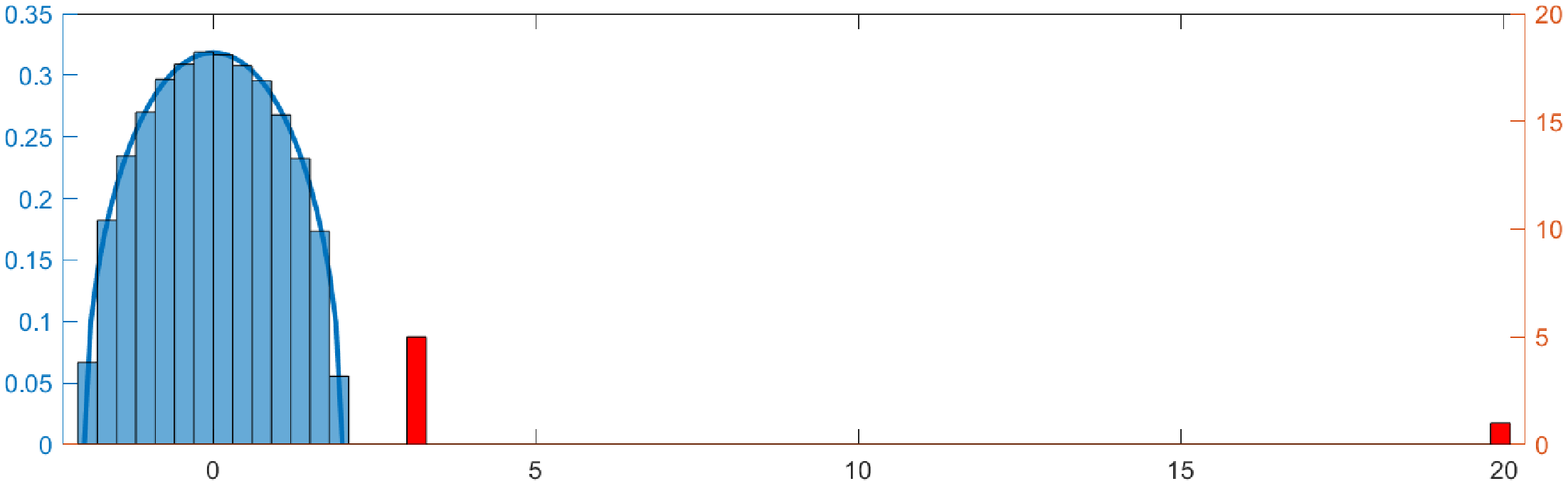}
		\label{fig:6group}
	}
	\caption{Empirical distributions of the eigenvalues of balanced SBM with (a) 3 clusters and (b) 6 clusters. Red bar shows the number of outliers.}
	\label{fig:outliers}
\end{figure}

In figure \ref{fig:outliers}, we plot the empirical distributions of the eigenvalues of the $3000 \times 3000$ balanced SBM with (a) $K=3$ and (b) $K=6$ clusters. Here, red bar implies the number of outliers from the bulk of the spectrum. As proved in Theorem \ref{thm:perturbation}, we can see a gap between the outliers and the bulk. For fixed $N$, this gap decreases with larger number of clusters since the nonzero eigenvalues of $\E A$ in \eqref{def:A} become smaller as the number of clusters increases.

\begin{remark}
	We can extend Theorem \ref{thm:perturbation} to more general matrix ensemble, which satisfies the condition that $A-\E A$ follows Assumption \ref{assumption} and entries of $\E A$ is order of $N^{-1+\epsilon}$ for any small $\epsilon$.
	We further remark that all our results also hold for complex Hermitian balanced generalized sparse random matrices without any change except that the limiting edge fluctuation is given by GUE Tracy--Widom law.
\end{remark}

\section{Strategy and outline of the proof}\label{sec:strategy}

In this section, we briefly outline the strategy of our proofs for the results in Section \ref{sec:def}.

\subsection{Main strategy for the proof} 
As illustrated in Section 3 of \cite{LS18}, a good estimate on the expectation of a sufficiently high power of the quadratic polynomial $1+zm+m^2$ is enough for the proof of the strong local law. To obtain such an estimate, we expand the term $1+zm$ by using a simple identity
\begin{align}\label{resolventId}
1+zG_{ii} = \sum_{k=1}^N H_{ik} G_{ki}.
\end{align}
In the expansion, which was called the resolvent expansion in \cite{LS18}, it is not easy to fully expand the terms with high powers, and the main idea in \cite{LS18} was to introduce the recursive moment estimate that estimates $\E| 1+zm+m^2|^D$ by the lower moments $\E|1+zm+m^2|^{D-\ell}$ for $\ell \geq 1$. When used together with the resolvent expansion, it makes the tracking of the higher order terms much simpler. 

In the actual estimate of the moments, we use a generalized version of Stein's lemma, which was introduced in \cite{Stein81}. It was used in the study of the linear eigenvalue statistics of random matrices (\cite{LP09,CL18}) and also the joint convergence of the largest eigenvalue and the linear statistics (\cite{BLW18}). 

\begin{lemma}[Cumulant expansion, generalized Stein's lemma]\label{lemma:Stein}
	Fix $\ell\in\mathbb{N}$ and let $F\in C^{\ell +1} (\mathbb{R}; \mathbb{C}^{+})$. Let $Y$ be a centered random variable with finite moments to order $\ell +2$. Then,
	\begin{equation}
	\mathbb{E}[YF(Y)] = \sum_{r=1}^{\ell} \frac{\kappa ^{(r+1)}(Y)}{r!} \mathbb{E}[F^{(r)}(Y)] + \mathbb{E}[\Omega_{\ell}(YF(Y))],
	\end{equation}
	where $\mathbb{E}$ denotes the expectation with respect to $Y$, $\kappa^{(r+1)}(Y)$ denotes the $(r+1)$-st cumulant of~$Y$ and $F^{(r)}$ denotes the $r$-th derivative of the function~$F$. The error term $ \Omega_{\ell}(YF(Y))$ satisfies
	\begin{align}
	\mathbb{E}[\Omega_{\ell}(YF(Y))]& \leq C_{\ell} \mathbb{E}[|Y|^{\ell +2}] \sup_{|t|\leq Q} |F^{(\ell +1)}(t)| + C_{\ell}\mathbb{E}[|Y|^{\ell +2} \mathds{1}(|Y|>Q) \sup_{t\in \mathbb{R}}|F^{(\ell +1)}(t)|,
	\end{align}
	where $Q>0$ is an arbitrary fixed cutoff and $C_{\ell}$ satisfies $C_{\ell}\leq(C\ell)^\ell / \ell !$ for some numerical constant $C$.
\end{lemma}

For more detail of the actual application of the methods explained in this subsection, we refer to \cite{LS18,Hwang2018}.

\subsection{Weak local semicircle law} 

The first obstacle we encounter in the proof of the strong local law is the lack of a priori estimates in the expansion. We follow the conventional strategy, developed in \cite{EYY12a,EYY12}, based on Schur complement formula and self-consistent equations. However, the bound in the weak local semicircle law is not as strong as the one obtained in \cite{EYY12}, due to the sparsity, but comparable with the weak local law in \cite{EKYY13}. The weak law, Theorem \ref{thm:weak local law}, is proved in Section \ref{sec:weak local law}.

\subsection{Strong local law}

The main technical difficulty of the proof the strong local law lies in that the entries of $H$ are not identically distributed. For example, if we use the cumulant expansion on the right side of \eqref{resolventId}, we get
\begin{align*}
\mathbb{E}[H_{ik}G_{ki}] &=  \sum_{r=1}^{\ell}\frac{\kappa^{(r+1)}_{ik}}{r!}\mathbb{E}[\partial^{(r+1)}_{ik}G_{ki}] + \mathbb{E}[\Omega_{\ell}(H_{ik}G_{ki})] \\
&= \mathbb{E}[\kappa_{ik}^{(2)}G_{kk}G_{ii}+ \kappa_{ik}^{(2)}G^2_{ik}] +\sum_{r=2}^{\ell}\frac{\kappa^{(r+1)}_{ik}}{r!}\mathbb{E}[\partial^{(r+1)}_{ik}G_{ki}] +\mathbb{E}[\Omega_{\ell}(H_{ik}G_{ki})],
\end{align*}
where $\Omega_{\ell}$ is the error term in the generalized Stein's lemma, Lemma \ref{lemma:Stein}. In the homogeneous case where $\kappa_{ik}^{(2)}$ are identical, the first term in the right hand side reduces to $m^2$ after averaging over indices $i$ and $k$. In our model, however, $\kappa_{ik}^{(2)}$ depends on the choice of $i$ and $k$ and hence the indices do not decouple even after the expansion. If the weak law were good enough so that the error from the substitution of $G_{kk}$ by $m$ is negligible, the analysis might work even with the absence of the decoupling mechanism, but Theorem \ref{thm:weak local law} is not enough in that purpose, and moreover, it is believed to be optimal.

In this paper, exploiting the community structure of the model, we decompose the sum into two parts, one for the case where $i$ and $k$ are in the same community and the other where $i$ and $k$ are not in the same community. For each sum, we replace the diagonal entries of the resolvent $G$ by $m$ if the error from the replacement is negligible. If the error is too large, we decompose it again by applying the community structure. The number of the diagonal entries of $G$ increases in each step, and the terms with enough diagonal entries can be handled by the substitution (by $m$) even with our local law. The detail can be found in Appendix \ref{sec:recursive}, especially in Subsection \ref{sec:subsec5}.

\subsection{Tracy--Widom limit and Green function comparison} 

The derivation of the Tracy--Widom fluctuation of the extremal eigenvalues from the strong local law is now a standard procedure in random matrix theory. In this paper, we follow the approach in \cite{LS18}, based on the Dyson matrix flow. For the sake of completeness, we briefly outline the main ideas of the proof.

To prove the Tracy--Widom fluctuation, we first notice that we can obtain the distribution of the largest eigenvalue of $H$ from the expectation of a function of $\im m(z)$. Then, we use the Green function comparison method to compare the edge statistics of the centered generalized stochastic block model and generalized Wigner matrix whose first and second moments match with our model. More precisely, for a given centered generalized stochastic block model $H_0$, we consider the \emph{Dyson matrix flow} with initial condition $H_0$ defined by
\begin{equation}
H_t := \e^{-t/2}H_0 + \sqrt{1-\e^{-t}}W^{G}, \qquad (t\geq 0),
\end{equation}
where $W^{G}$ is a generalized Gaussian Wigner matrix independent of $H_0$, with vanishing diagonal entries. The local edge statistics of $W^G$ with vanishing diagonal follows the GOE Tracy--Widom statistics; see Lemma 3.5 of \cite{LS18} and Theorem 2.7 of \cite{BEY14}.

Along the flow, we track the change of the expectation of a function of $\im m(z)$. Taking the deterministic shift of the edge into consideration, we find that fluctuation of the extremal eigenvalues of $H_t$ do not change over $t$, which establishes the Tracy--Widom fluctuation for $H_0$. The analysis along the proof requires the strong local law for the normalized trace of the Green function of $H_t$, defined as
\begin{align}\label{def:G_t}
G_t(z) = (H_t)^{-1}, \quad m_t(z) = \frac{1}{N} \sum_{i=1}^N (G_t)_{ii}(z),\quad (z\in\mathbb{C}^+).
\end{align}

We note that $H_t$ is also a balanced generalized sparse random matrix. To check this, let $\kappa_{t,ij}^{(k)}$ be the $k$-th cumulant of $(H_t)_{ij}$. Then, by the linearity of the cumulants under the addition of independent random variables, we have $\kappa_{t,(\b{\cdot})}^{(1)}=0$, $\kappa_{t,(\b{\cdot})}^{(2)}\leq \frac{C}{N}$ and $\kappa_{t,(\b{\cdot})}^{(k)}=\mathrm{e}^{-kt/2}\kappa_{(\b{\cdot})}^{(k)}$ for $k\geq 3$. In particular, we have the bound
\begin{align}
|\kappa_{t,(\b{\cdot})}^{(k)}| \leq \mathrm{e}^{-t}\frac{(Ck)^{ck}}{Nq_t^{k-2}} , \quad (k\geq 3),
\end{align}
where we introduced the time-dependent parameter
\begin{align}\label{def:q_t}
q_t:= q\mathrm{e}^{t/2}, \qquad 	\zeta_t : = \frac{N}{K} (\kappa_{t,s}^{(2)} - \kappa_{t,d}^{(2)} ).
\end{align}
We also define a polynomial $P_{z, t}$ of $m$ with parameters $z$ and $t$ by
\begin{align}\label{def:P}
P_{z,t}(m) :=& \Big(1 + zm+m^2 +\e^{-t} q_t^{-2}\xi^{(4)}m^4\Big)\Big((z+m+\zeta_tm)^2-\zeta_t(1+mz+m^2)\Big) \nonumber \\
=:&  P_{1,z,t}(m)P_{2,z,t}(m).
\end{align}
We generalize Theorem \ref{thm:locallaw} as follows:

\begin{proposition} \label{prop:locallaw}
	Let $H_0$ satisfy Assumption \ref{assumption} with $\phi>0$. Then, for any $t\geq 0$, 
	there exist deterministic number $2 \leq L_t < 3$ and an algebraic function $\wt{m}_t:\mathbb{C}^+ \rightarrow \mathbb{C}^+$ 
	such that the following hold:
	\begin{enumerate}
		\item[(1)] The function $\wt{m}_t$ is the Stieltjes transform of a deterministic probability measure $\wt{\rho}_t$, i.e., $\wt{m}_t(z) = m_{\wt{\rho}_t}(z).$ The measure $\rho_t$ is supported on $[-L_t, L_t]$ and $\wt{\rho}_t$ is absolutely continuous with respect to Lebesgue measure with a strictly positive density on $(-L_t, L_t)$.
		\item[(2)] The function $\wt{m}_t\equiv\wt{m}_t(z)$, $z\in\mathbb{C}^+$, is a solution to the polynomial equation
		\begin{align}\label{eq:P_1,t,z=0}
		P_{1,t,z}(\wt{m}_t) &:=1 + z\wt{m}_t+\wt{m}_t^2 +e^{-t}q_t^{-2}\Big(\frac{s_s^{(4)}+(K-1)s_d^{(4)}}{K}\Big)\wt{m}_t^4 \nonumber\\ 
		&=1 + z\wt{m}_t+\wt{m}_t^2 +\e^{-2t}q^{-2}\xi^{(4)}\wt{m}_t^4=0.
		\end{align}
		\item[(3)] The normalized trace of the Green function satisfies the local law
		\begin{align}\label{eq:stronglaw}
		|m_t(z) -\wt{m}_t(z)| \prec \frac{1}{q_t^2}+\frac{1}{N\eta},
		\end{align}
		uniformly on the domain $\mathcal{E}$ and uniformly in $t\in [0, 6\log N]$.
	\end{enumerate}
\end{proposition}
\begin{remark}
	Several properties of $\wt m_t$, defined in Proposition \ref{prop:locallaw}, are crucial in the proof of the Tracy--Widom fluctuation, especially the square-root decay at the edge of the spectrum and the deterministic shift of the edge, where the upper edge of the support of $\wt{\rho_t}$ given by
	\begin{align}\label{eq:L_t}
	L_t = 2 + e^{-t}q_t^{-2}\xi^{(4)} + O(e^{-2t}q_t^{-4}).
	\end{align}
	In Appendix \ref{sec:rho wt m}, we collect some important properties of $\wt m_t$ and some basic properties of $m_{sc}$, the Stieltjes transform of the semicircle measure.
\end{remark}

\section{Proof of weak local law}\label{sec:weak local law}
\subsection{Preliminaries}
In this section, we prove Theorem \ref{thm:weak local law}. Unlike strong local semicircle law for cgSBM, we can prove the weak local semicircle law under weaker condition. More precisely, we do not need any assumption about the community structure. Therefore, in this section, we consider generalized sparse random matrices defined as follows:
\begin{assumption}[Generalized sparse random matrix]\label{def:weak_H}
	Fix sufficiently small $\phi>0$. A generalized sparse random matrix, $H=(H)_{ij}$ is a symmetric $N \times N$ matrix whose diagonal entries are almost surely zero and whose off-diagonal entries are independent, up to symmetry constraint $H_{ij}=H_{ji}$, random variables. We further assume that each $H_{ij}$ satisfy the moment conditions
	\begin{align}\label{def:weak_momentcondition}
	\E H_{ij}=0, \qquad \E|H_{ij}|^2=\sigma_{ij}^2,\qquad\E|H_{ij}|^k \leq \frac{(Ck)^{ck}}{Nq^{k-2}},\qquad (k\geq 2),
	\end{align}
	with sparsity parameter $q$ satisfying
	\begin{align} \label{def:q}
	N^{\phi}\leq q \leq N^{1/2}.
	\end{align}
	We further assume the normalize condition that
	\begin{equation}
	\sum_{i=1}^{N}\sigma_{ij}^2 =1.
	\end{equation}
\end{assumption}

Recall that 
$$\mathcal{D}_\ell := \{z=E+\ii \eta \in \C^+ : |E|<3, N^{-1+\ell} <\eta \leq 3\}.$$
Throughout this section, we use the factor $N^\epsilon$ and allow $\epsilon$ to increase by a tiny amount from line to line to absorb numerical constants in the estimates. Moreover, we choose $\ell$ satisfying $ 4\phi \leq \ell \leq 1 $ and $\epsilon>0$ strictly smaller than the fixed parameter $\phi >0$ appearing in \eqref{def:q}. If we take $\phi$ sufficiently small, then Theorem \ref{thm:weak local law} states that \eqref{thm:weak law 1}, \eqref{thm:weak law 2} and \eqref{thm:weak law 3} holds on $\mathcal{D_\ell}$ for sufficiently small $\ell$. In other words we can claim that Theorem \ref{thm:weak local law} valid on $\mathcal{D_\ell}$ for any (small) $\ell$.

We define the $z$-dependent quantities
\[ v_k := G_{kk}-m_{sc} , \quad  [v] := \frac{1}{N}\sum_{k=1}^{N}v_k = m-m_{sc}. \]
Our goal is to estimate the following quantities,
\begin{equation}\label{def:Lambdas}
\Lambda_d := \max_{k}|v_k|=\max_{k}|G_{kk}-m_{sc}|, \qquad \Lambda_o:=\max_{k\neq l}|G_{kl}|, \qquad \Lambda:= |m-m_{sc}|.
\end{equation}

\begin{definition}[Minors]
	Consider general matrices whose indices lie in subsets of $\{1,\dots , N \}$. For $T \subset \{1,\dots , N \}$ we define $H^{(T)}$ as the $(N-|T|)\times (N-|T|)$ matrix 
	\[H^{(T)}=(H_{ij})_{i,j \in \{1,\dots , N \} \backslash T} .\] 
	It is important to keep the original values of the matrix indices in the minor $H^{(T)}$, not to identify $\{1,\dots , N \} \backslash T$ with $\{1,\dots , N -|T|\}$.
	We set \[ \sum_{i}^{(T)}:= \sum_{i:i\notin T} .\] If $T=\{a\}$, we abbreviate $(\{a\})$ by $(a)$ in the above definition; similarly, write $(ab)$ instead of $(\{a,b\}).$
	We also define the Green function of $H^{(T)}$ as
	\[  G^{(T)}_{ij}(z) := (H^{(T)}-z)^{-1}_{ij} .\]
\end{definition}

\begin{definition}[Partial expectation]
	Let $X \equiv X(H)$ be a random variable and $\textbf{h}_i  = (H_{ij})_{j=1}^N $. For $i \in \{1, \dots ,N\}$ we define the operations $\mathbb{E}_i $ and $\mathbb{IE}_i$ through
	\begin{equation}\label{eq:partial}
	\mathbb{E}_i  X := \E(X|\textbf{h}_i), \qquad \mathbb{IE}_i X := X- \mathbb{E}_i X.
	\end{equation}
\end{definition}

For $T \subset \{1,\dots , N \}$, we introduce the following notations:
\begin{equation}\label{def:Zij Kij}
Z^{(T)}_{ij} := \sum_{k,l}^{T}H_{ik}G^{(T)}_{kl}H_{lj} , \qquad K^{(T)}_{ij} := H_{ij}-z\delta_{ij}- Z^{(T)}_{ij}. 
\end{equation}		
We abbreviate 
\begin{equation}\label{def:Zi}
Z_i := \mathbb{IE}_i Z^{(i)}_{ii} = \mathbb{IE}_i \sum_{k,l}^{(i)}H_{ik}G^{(i)}_{kl}H_{li}.
\end{equation}

The following formulas with these notations were proved in Lemma 4.2 of \cite{EYY12a}.
\begin{lemma}[Self-consistent permutation formulas]\label{lem:self consistent for G} For any Hermitian matrix $H$ and $T \subset \{1, \dots , N\}$ the following identities hold. If $i,j,k \notin T$ and $i,j \neq k$ then
	\begin{equation}\label{eq:self cons 1}
	G^{(T)}_{ij}=G^{(Tk)}_{ij}+{G^{(T)}_{ik}G^{(T)}_{kj}}({G^{(T)}_{kk}})^{-1}.
	\end{equation}
	If $i,j \notin T$ satisfy $i \neq j$, then
	\begin{align}
	G^{(T)}_{ii}=&(K^{(iT)}_{ii})^{-1} = (H_{ii}-z-Z^{(iT)}_{ii})^{-1},\label{eq:self cons 2}\\
	G^{(T)}_{ij}=& -G^{(T)}_{jj}G^{(jT)}_{ii}K^{(ijT)}_{ij}=-G^{(T)}_{ii}G^{(iT)}_{jj}K^{(ijT)}_{ij}\label{eq:self cons 3}.
	\end{align}
	
\end{lemma}
We also have the $\emph{Wald identity}$
\begin{equation}\label{eq:Wald}
\sum_{j}|G_{ij}|^2 = \frac{\text{Im } G_{ii}}{\eta}.
\end{equation}

The following trivial estimate provide the bound for the matrix element of $H$.
\begin{lemma}\label{lem:bound Hij} We have 
	\[ |H_{ij}| \prec \frac{1}{q}.\]
\end{lemma}
\begin{proof}
	The proof follows from the Markov's inequality and the moment conditions for $H$.
\end{proof}

\subsection{Self-consistent perturbation equations}

Following \cite{EYY11}, we define the following quantities:
\begin{align}
&A_i := \sigma^2_{ii}G_{ii} + \sum_{j\neq i}\sigma^2_{ij}\frac{G_{ij}G_{ji}}{G_{ii}},\label{def:Ai}\\
&\Upsilon_i := A_i + h_{ii} - Z_i\label{def:Upsilon}
\end{align}
and recall the definition of $Z_i$ in \eqref{def:Zi}.
Using \eqref{eq:self cons 1} and \eqref{eq:self cons 2}, we can easily obtain the self-consistent equations for the deviation $m_{sc}$ of the diagonal matrix elements of the Green function;
\begin{equation}\label{eq:vi self consistent}
v_i = G_{ii}-m_{sc} = \frac{1}{-z-m_{sc}-\big(\sum_{j}\sigma^2_{ij}v_j -\Upsilon_i\big)}-m_{sc}.
\end{equation}
Now we define the exceptional (bad) event 
\begin{equation}\label{def:Bad set}
\textbf{B}=\textbf{B}(z):=\{\Lambda_d(z)+\Lambda_o(z) \geq (\log N)^{-2}\},
\end{equation}
and the control parameter 
\begin{equation}\label{def:Psi}
\Psi(z) := \sqrt{\frac{\Lambda(z)+\text{Im }m_{sc}(z)}{N\eta}}.
\end{equation}
On $\textbf{B}^c$, we have $\Psi(z) \leq CN^{-2\phi}$ by definition of $\mathcal{D_\ell}$.
We collect some basic properties of the Green function in the following elementary lemma which were proved in Lemma 3.5 of \cite{EYY12} and Lemma 3.12 of \cite{EYY11}.
\begin{lemma}\label{lem:estimate basic quantities}
	Let $\mathbb{T}$ be a subset of $\{1,\dots,N\}$ and $i \notin \mathbb{T}$. Then there exists a constant $C=C_\mathbb{T}$ depending on $|\mathbb{T}|$, such that the following hold in $\textbf{B}^c$
	\begin{align}
	&|G^{(\mathbb{T})}_{kk} - m_{sc}| \leq \Lambda_d+C\Lambda_o^2   &\text{for all } k \notin \mathbb{T}\label{eq:Greenlemma1}\\
	&\frac{1}{C} \leq |G^{(\mathbb{T})}_{kk}| \leq C   &\text{for all } k \notin \mathbb{T}\label{eq:Greenlemma2}\\
	&\max_{k\neq l} |G^{(\mathbb{T})}_{kl}| \leq C\Lambda_o \label{eq:Greenlemma3}\\
	&\max_{i} |A_i| \leq \frac{C}{N} + C\Lambda_o^2 \label{eq:Greenlemma4}
	\end{align}
	for any fixed $|\mathbb{T}|$ and for sufficiently large $N$.	
\end{lemma}
We note that all quantities depend on the spectral parameter $z$ and the estimates are uniform in $z = E + \ii\eta$.

\subsection{Estimate of the exceptional events and analysis of the self-consistent equation}
We introduce three lemmas which estimate some exceptional events. Their proofs are given in Appendix \ref{app:3 lemmas}.
\begin{lemma}\label{lem:esteimate Lambda o}
	For fixed $z \in D_\ell$ and any small $\epsilon>0 $, we have on $\textbf{B}^c$ with high probability
	\begin{equation}\label{eq:estimate Lambda o}
	\Lambda_o \leq C\left(\frac{N^\epsilon}{q}+N^{\epsilon}\Psi\right).
	\end{equation}
\end{lemma}

\begin{lemma}\label{lem:bound Zi Zij}
	For any $z \in D_\ell$, we have on $\textbf{B}^c$ with high probability
	\begin{equation}\label{eq:bound Zi}
	|Z_i| \leq C \left(\frac{N^{\epsilon}}{q} + N^{\epsilon}\Psi\right),
	\end{equation}
	\begin{equation}\label{eq:bound Zij}
	|Z^{ij}_{(ij)}| \leq C \left(\frac{N^{\epsilon}}{q} + N^{\epsilon}\Psi\right) \quad 	 (i\neq j).
	\end{equation}
\end{lemma}

\begin{lemma}\label{lem:estimate Upsilon}For any $z \in D_\ell$, we have on $\textbf{B}^c$ with high probability
	\begin{equation}
	|\Upsilon_i| \leq C \left(\frac{N^{\epsilon}}{q} + N^{\epsilon}\Psi\right).
	\end{equation}
\end{lemma}

We define the events
\begin{align}
\Omega_h &:= \left\{ \max_{1\leq i,j \leq N} |H_{ij}| \geq \frac{N^\epsilon}{q}  \right\} \cup \left\{ 
\big| \sum_{i=1}^{N} H_{ii} \big| \geq N^{\epsilon} \big(\frac{1}{q}+1\big)\right\}, \nonumber\\
\Omega_d &:=\left\{ \max_{i}|Z_i| \geq C \left(\frac{N^{\epsilon}}{q} + N^{\epsilon}\Psi\right)  \right\}  ,\\
\Omega_o &:=\left\{ \max_{i\neq j}|Z^{(ij)}_{ij}| \geq C \left(\frac{N^{\epsilon}}{q} + N^{\epsilon}\Psi\right) \right\},  \nonumber
\end{align}
and let

\begin{equation}
\Omega(z) := \Omega_h \cup [(\Omega_d\cup \Omega_o) \cap \textbf{B}^c].
\end{equation}

Then by $\eqref{eq:bound Zi}$, $\eqref{eq:bound Zij}$ and the large deviation estimate $\eqref{lem:LDE 1}$ we can show that $\Omega^c$ holds with high probability.

From ($\ref{eq:vi self consistent}$), we obtain the following equation for $v_i$
\begin{equation}
v_i = m_{sc}^2\bigg(\sum_{j} \sigma_{ij}^2 v_j -\Upsilon_i\bigg)+m_{sc}^3\bigg(\sum_{j} \sigma_{ij}^2 v_j -\Upsilon_i\bigg)^2+O\bigg(\sum_{j} \sigma_{ij}^2 v_j -\Upsilon_i\bigg)^3.
\end{equation}
By assumption, $\sum_{j}\sigma_{ij}^2 =1$,  $\textbf{e}=(1,1,\dots,1)$ is the unique eigenvector of $B=({\sigma_{ij}})$ with simple eigenvalue $1$. Define the parameter 
\begin{equation}
g=g(z) := \max\{\delta_+, |1-\text{Re } m_{sc}^2(z)| \}. 
\end{equation}
and we recall the following basic lemma that was proven in Lemma 4.8 of \cite{EYY11}.
\begin{lemma}\label{lem:bound norm}
	The matrix $I-m_{sc}^2(z)B$ is invertible on the subspace orthogonal to $\textbf{e}$. let $\textbf{u}$ be a vector which is orthogonal to $\textbf{e}$ and let 
	\[\textbf{w}=(I-m_{sc}^2(z)B)\textbf{u}, \]
	then 
	\[ ||\textbf{u}||_\infty  \leq \frac{C\log N}{g(z)} ||\textbf{w}||_\infty \]
	for some constant $C$ that only depends on $\delta_-$.
\end{lemma}
We introduce the following lemma which estimates the deviation of $v_i$ from its average $[v]$.
\begin{lemma}\label{lem:lem analysi of }
	Fix $z \in D_\ell$. If in some set $\Xi$ it holds that 
	\begin{equation}\label{eq:Lambda d bound assume}
	\Lambda_d \leq \frac{q}{(\log N)^{3/2}}, 
	\end{equation}
	then in the set $\Xi \cap \textbf{B}^c$, we have 
	\begin{align}\label{eq:bound max vi-[v]}
	\max_{i}|v_i-[v]|&\leq \frac{C\log N }{g}\left( \Lambda^2 + \frac{N^{\epsilon}}{q}+N^{\epsilon} \Psi + \frac{(\log N)^2}{g^2}\bigg(\frac{N^{\epsilon}}{q}+N^{\epsilon} \Psi \bigg)^2 \right)\nonumber\\
	&\leq \frac{C \log N}{g^3}\left(\Lambda^2 + \frac{N^{\epsilon}}{q}+N^{\epsilon} \Psi  \right),
	\end{align}
	with high probability.
\end{lemma}
With Lemma \ref{lem:lem analysi of }, we can show the following.
\begin{lemma}\label{lem:analysis of self consis}
	Fix $z \in D_\ell$. Define $[Z]:= N^{-1}\sum_{i=1}^{N}Z_i$. Then in the set $\textbf{B}^c$ we have
	\begin{equation}
	(1-m_{sc}^{2})[v]=m_{sc}^3[v]^2+m_{sc}^2[Z]+O\left(\frac{\Lambda^2}{\log N} \right) +O\left((\log N)^3\left(\frac{N^{\epsilon}}{q}+N^{\epsilon} \Psi \right)^2 \right),
	\end{equation}
	with high probability.
\end{lemma}
Proofs of Lemmas \ref{lem:lem analysi of } and \ref{lem:analysis of self consis} are given in Appendices \ref{app:lem analysi of } and \ref{app:analysis of self}, respectively.

\subsection{Dichotomy estimate for $\Lambda$ and  continuity argument}
In $\textbf{B}^c$, $\max_{i}|Z_i| = O(\frac{N^{\epsilon}}{q}+N^{\epsilon} \Psi)$ holds with high probability. Therefore using $\frac{N^{\epsilon}}{q}+N^{\epsilon} \Psi \leq N^{-\epsilon}$, with high probability, we have 
\begin{equation}\label{eq:weak self consistent equation}
(1-m_{sc}^{2})[v]=m_{sc}^3[v]^2+O\left(\frac{\Lambda^2}{\log N}\right)+O\bigg(\frac{N^{\epsilon}}{q}+N^{\epsilon} \Psi\bigg).
\end{equation}
\begin{lemma}\label{lem:inital estimate}
	Let $\eta \geq 2$. Then for $z \in D_\ell$ we have
	\begin{equation}\label{eq:estimate Lambda_d + Lambda_o}
	\Lambda_d(z)+\Lambda_o(z) \leq C\left(\frac{N^\epsilon}{q} + \frac{N^{\epsilon}}{\sqrt{N}} \right) \leq (\log N)^{-2}.
	\end{equation}
\end{lemma}

Lemma \ref{lem:inital estimate} is an initial estimates on $\Lambda_d$ and $\Lambda_o$ for large $\eta \sim 1$ to get the continuity argument started.

We further introduce the following notations
\begin{equation}
\alpha := \left|\frac{1-m_{sc}^2}{m_{sc}^3} \right|, \qquad \qquad \beta:= \frac{N^\epsilon}{\sqrt{q}} + \frac{N^{\epsilon}}{(N\eta)^{1/3}},
\end{equation}
where $\alpha$ and $\beta$ depend on the parameter $z$. For any $z \in D_\ell$ we have the bound $\beta \leq N^{-\frac{1}{2}\phi}$.

From Lemma $\ref{lem:basic properties of m}$, it follows that for any $z \in D_\ell$ there is a constant $K \geq 1 $ such that 
\begin{equation}
\frac{1}{K} \sqrt{\kappa + \eta} \leq \alpha(z) \leq K\sqrt{\kappa + \eta}.
\end{equation}

Since $\sqrt{\kappa + \eta}$ is increasing and $\beta(E+i\eta)$ is decreasing in $\eta$, we know that, for any fixed $E$ and $U>1$, $\sqrt{\kappa + \eta}=2U^2K\beta(E+i\eta)$ has a unique solution $\tilde{\eta}=\tilde{\eta}(U,E)$ which satisfies $\tilde{\eta} \ll 1$.

\begin{lemma}[Dichotomy]\label{lem:dichtomy_weak}
	There exist a constant $U_0$ such that for any fixed $U\geq U_0$, there exists constant $C_1(U)$ such that the following hold for any $z \in D_\ell$. 
	\begin{align}
	\Lambda(z) &\leq U\beta(z) \qquad   or \qquad   \Lambda(z) \geq \frac{\alpha(z)}{U} & \text{if  }  \eta \geq
	\tilde{\eta}(U,E)\label{eq:dicho1}\\
	\Lambda(z) &\leq C_1(U)\beta(z)    & \text{if  } \eta < \tilde{\eta}(U,E)\label{eq:dicho2}
	\end{align}
	on $\textbf{B}^c(z)$ with high probability.
\end{lemma}
Proofs of Lemma \ref{lem:inital estimate} and \ref{lem:dichtomy_weak} are given in  \ref{subsec:proof dichtomy}.

Now choose a decreasing finite sequence $\eta_k \in D_\ell$, $k=1,2,\dots,k_0$, with $k_0 \leq CN^8$, $|\eta_k-\eta_{k+1}|\leq N^{-8}$, $\eta_1=2$, and $\eta_{k_0}=N^{-1+l}$. Fix $E$ with $|E|\leq 3$ and set $z_k = E + i\eta_k$. We fix $U\geq U_0$ and recall the definition of $\tilde{\eta}$ from Lemma $\ref{lem:dichtomy_weak}$.

Consider the first case of $z_1$. For large $N$, it is easy to show that $\eta_1 \geq \tilde{\eta}$ for any $|E|\leq 3$. Therefore Lemma $\ref{lem:inital estimate}$ and Lemma $\ref{lem:dichtomy_weak}$ imply that 
$\textbf{B}^c(z_1)$ and $\Lambda(z_1) \leq U\beta(z_1) $
hold with high probability. For general $k$ we have the following : 
\begin{lemma}\label{lem:bound Omega_k}
	Define $\Omega_k := \textbf{B}^c(z_k) \bigcap \{\Lambda(z_k)\leq C^{(k)}(U)\beta(z_k) \} $ where 
	\begin{equation}
	C^{(k)}(U) =  \begin{cases}
	U &  \text{if  $ \eta_k \geq \tilde{\eta}(U,E) $}\\
	C_1(U) &  \text{if  $ \eta_k \leq \tilde{\eta}(U,E) $}.\\
	\end{cases}
	\end{equation}
	Then
	\begin{equation}\label{eq:Omega_k high probability}
	\P(\Omega_k^c) \leq 2kN^{-D}.
	\end{equation}
\end{lemma}

Proof of the Lemma \ref{lem:bound Omega_k} is given in Appendix \ref{app:proof of Lemma lem bound Omega_k}. Now we are ready to prove our main theorem. 
\begin{proof}[Proof of Theorem \ref{thm:weak local law}]
	Take a lattice $L \subset D_\ell$ such that $|L| \leq CN^6$ and for any $z \in D_\ell$ there exist $\tilde{z} \in L$ satisfying $|z-\tilde{z}|\leq N^{-3}$. From the Lipschitz continuity of the map $z\mapsto G_{ij}(z)$ and $z\mapsto m_{sc}(z)$ with a Lipschitz constant bounded by $\eta^{-2}\leq N^2$, we have
	\begin{equation}\label{eq:Lipshitz1}
	|G_{ij}(z)-G_{ij}(\tilde{z})|\leq \frac{|z-\tilde{z}|}{\eta^2} \leq \frac{1}{N} .
	\end{equation}
	We also have 
	\begin{equation}\label{eq:Lipshitz2}
	|m(z)-m(\tilde{z})|\leq \frac{|z-\tilde{z}|}{\eta^2} \leq \frac{1}{N} .
	\end{equation}
	By Lemma $\ref{lem:bound Omega_k}$, we have for some large constant $C$ 
	\begin{equation}
	\P \left[ \bigcap_{\tilde{z} \in L} \left\{ \Lambda(\tilde{z}) \leq C\beta(\tilde{z})  \right\}\right] \geq 1-N^{-D}.
	\end{equation}
	Hence with $\eqref{eq:Lipshitz1}$ and $\beta \gg N^{-1}$ we find 
	\begin{equation}
	\P \left[ \bigcup_{z \in D_\ell} \left\{ \Lambda(z) > C\beta(z)  \right\}\right] \leq N^{-D},
	\end{equation}
	for some constant $C$.
	Using similar argument, we can also get
	\begin{equation}
	\P \left[ \bigcap_{z \in D_\ell} \textbf{B}^c(z)  \right] \leq N^{-D}.
	\end{equation}
	In other words, we proved ($\ref{thm:weak law 3}$). Using $\eqref{eq:estimate Lambda o}$, $\eqref{eq:bound max vi-[v]}$ and $\eqref{eq:Omega_k high probability}$ with similar lattice arguments, we can conclude the proof of Theorem \ref{thm:weak local law}.
\end{proof}

\section{Proof of Proposition \ref{prop:locallaw} and Theorem \ref{thm:matrix norm}} \label{sec:local law}
\subsection{Proof of Proposition \ref{prop:locallaw}}

In this section we prove Proposition \ref{prop:locallaw}. Recall the subdomain $\mathcal{D_\ell}$ of $\mathcal{E}$ and the definition of the polynomial $P\equiv P_{t,z}, P_1\equiv P_{1,t,z}$ and $P_2\equiv P_{2,t,z}$ in \eqref{def:P}. We have the following lemma called recursive moment estimate.
\begin{lemma}[Recursive moment estimate] \label{lemma:recursive} Fix $\phi>0$ and fix any $t\geq 0$. Let $H_0$ satisfies Assumption \ref{assumption}. Then, for any $D>10$ and small $\epsilon>0$, the normalized trace of the Green function, $m_t\equiv m_t(z)$, of the matrix $H_t$ satisfies
	\begin{align}\label{eq:recursive}
	&\mathbb{E}|P(m_t)|^{2D}\leq N^{\epsilon}\mathbb{E}\Big[\Big(\frac{1}{q_t^4} +\frac{\im m_t}{N\eta} )|P(m_t)|^{2D-1}\Big] + N^{-\epsilon/8}q_t^{-1}\mathbb{E}\Big[|m_t -\wt{m}_t|^2|P(m_t)|^{2D-1}\Big]  \notag \\
	&\qquad +N^{\epsilon} q_t^{-1}\sum_{s=2}^{2D}\sum_{s'=0}^{s-2}\mathbb{E}\Big[\Big(\frac{\im m_t}{N\eta}\Big)^{2s-s'-2}|P'(m_t)|^{s'}|P(m_t)|^{2D-s}\Big]  \\
	&\qquad +N^{\epsilon}\sum_{s=2}^{2D}\mathbb{E}\Big[\Big(\frac{1}{N\eta}+\frac{1}{q_t}\Big(\frac{\im m_t}{N\eta} \Big)^{1/2}+\frac{1}{q_t^2}\Big) \Big(\frac{\im m_t}{N\eta} \Big)^{s-1}|P'(m_t)|^{s-1}|P(m_t)|^{2D-s}\Big] +N^{\epsilon}q_t^{-8D}, \notag 
	\end{align}
	uniformly on the domain $\mathcal{D_\ell}$, for sufficiently large $N$.
\end{lemma}

We give the detailed proof of Lemma \ref{lemma:recursive} in Appendix \ref{sec:recursive}. In this section, we only sketch the idea of the proof.
We estimate the expectation of $|P(m_t)|^{2D}$ using Lemma \ref{lemma:Stein}. For example, consider the cumulant expansion of $1+zm$, a part of $P_{1,t,z}(m)$, computed by
\begin{align*}
\mathbb{E}[1+zm] = \mathbb{E}\Big[\frac{1}{N}\sum_{i=1}^{N}(1+zG_{ii})\Big] = \mathbb{E}\Big[\frac{1}{N}\sum_{i=1}^{N}(HG)_{ii}\Big] = \mathbb{E}\Big[\frac{1}{N}\sum_{i,j}(H_{j i}G_{j i})\Big],
\end{align*}
where we used the definition of Green function to get the second equation. Then by Lemma \ref{lemma:Stein}, we get
\begin{align*}
\mathbb{E}\Big[\frac{1}{N}\sum_{i,j}(H_{j i}G_{j i})\Big]&=\frac{1}{N} \sum_{i\neq j} \kappa_{ij}^{(2)}\E \qbb{ \pB{\partial_{ij} G_{ij}} }\nonumber \\
&=- \frac{1}{N} \E \qbb{ \sum_{i,j} \kappa_{ij}^{(2)} G_{ii}G_{jj}}+\frac{1}{N} \E \qbb{ \sum_{i} \kappa_{ii}^{(2)} G_{ii}^2}-\frac{1}{N} \E \qbb{ \sum_{i\neq j} \kappa_{ij}^{(2)} G_{ij}^2},
\end{align*}
and it can be easily shown that the terms containing at least one off-diagonal Green function entries are sufficiently small.
Thus, the main order term we need to estimate is 
\begin{align*}
\frac{1}{N} \E \qbb{ \sum_{i,j} \kappa_{ij}^{(2)} G_{ii}G_{jj}}&=\frac{1}{N}\E \qbb{ \sum_{i,j} \kappa_d^{(2)}G_{ii}G_{jj}}+\frac{1}{N}\E\qbb{\sum_{i\sim j} (\kappa_s^{(2)}-\kappa_d^{(2)})G_{ii}G_{jj}} \nonumber \\
&= \E \qbb{(1-\zeta )m^2  } + \frac{\zeta K}{N^2}\E\qbb{\sum_{i\sim j} G_{ii}G_{jj}}. \notag 
\end{align*}
To compute with sufficiently small error, using Lemma \ref{lemma:Stein} again, we multiply $z$ to the second term to obtain
\begin{align*}
\frac{\zeta K}{N^2}\E\qbb{\sum_{i\sim j} zG_{ii}G_{jj}}&= \frac{\zeta K}{N^2}\E\qbb{\sum_{i\sim j} \pB{\sum_{k}H_{ik}G_{ki} -1 } G_{jj}}=  \sum_{r=1}^{l}  \frac{1}{r!} \E J_{r} - \zeta \E[m] + O(\Phi_\epsilon),
\end{align*}
where 
\begin{align*}
J_{r}=\frac{\zeta K}{N^2}\sum_{i\sim j}\sum_{k\neq i}\kappa_{ik}^{(r+1)}\E \qbb{ \pB{\partial_{ik}^{r} G_{ik}G_{jj}}},
\end{align*}
and $O(\Phi_\epsilon)$ is a sufficiently small error term defined by the right side of \eqref{lemma:recursive}.
One of the main order term which only consists of the diagonal entries of the Green function is
\begin{align}\label{eq:J1}
\E J_{1}=\frac{\zeta K}{N^2}\sum_{i\sim j}\sum_{k}\kappa_{d}^{(2)}\E \qbb{ \pB{G_{ii} G_{jj}G_{kk}} {}}-\frac{\zeta K}{N^2}\sum_{i\sim j\sim k}(\kappa_{s}^{(2)}-\kappa_{d}^{(2)})\E \qbb{ \pB{G_{ii} G_{jj}G_{kk}} {}} .
\end{align}
The first term of the right hand side of \eqref{eq:J1} can be estimated by
\begin{align*}
-\frac{\zeta  K}{N^2}\sum_{i\sim j}\sum_{k\neq i}\kappa_{d}^{(2)}\E \qbb{ \pB{G_{ii} G_{jj}G_{kk}} {}}=-\frac{\zeta  (1-\zeta  )K}{N^2}\sum_{i\sim j}\E \qbb{ mG_{ii} G_{jj} }.
\end{align*}
Thus, if we can estimate the second term of the right hand side of \eqref{eq:J1} with sufficiently small error, then we get the good estimation for
\begin{align*}
\frac{\zeta  K}{N^2}\E\qbb{\sum_{i\sim j} zG_{ii}G_{jj}}+\frac{\zeta  (1-\zeta  )K}{N^2}\sum_{i\sim j}\E \qbb{ mG_{ii} G_{jj} } = \frac{\zeta K}{N^2}\E \qbb{\big(z+(1-\zeta )m\big)\sum_{i\sim j} G_{ii} G_{jj}  }.
\end{align*}
Still it is not easy to handle the second term of the right hand side of \eqref{eq:J1} due to its community structure. We abbreviate 
\begin{align*}
\hat{J}:=\frac{\zeta  K}{N^2}\sum_{i\sim j\sim k}(\kappa_{s}^{(2)}-\kappa_{d}^{(2)})\E \qbb{ \pB{zG_{ii} G_{jj}G_{kk}} {}},
\end{align*}
and use Lemma \ref{lemma:Stein} once more. With some expansions and calculations, one can get the estimation for $\hat{J}$ and with small error terms. We apply a similar strategy to the expectation of $|P(m_t)|^{2D}$ and we can obtain the recursive moment estimate stated in \eqref{eq:recursive}.


With Lemma \ref{lemma:recursive}, we can prove Proposition \ref{prop:locallaw}.

\begin{proof}[Proof of Proposition \ref{prop:locallaw} and Theorem \ref{thm:locallaw}]
	Fix $t\in [0,6\log N]$. Let $\wt m_t$ be the solution $w_t$ of the equation $P_{1,t,z}(w_t)=0$. One can show the first two parts directly from the properties of $\wt m_t$ and its Stieltjes inversion $\rho_t$; see Appendix \ref{sec:rho wt m}. It remains to prove the third part of the proposition. 
	Since $|m_t| \sim 1$,  $|\wt m_t| \sim 1$, there exist positive constant $c_1, c_2$ such that
	\begin{align}
	&\frac{1}{c_1} \leq P_2(\wt{m_t}) \leq c_1 \\
	&\frac{1}{c_2} + O(\frac{1}{q^2})\leq P''(\wt{m_t}) \leq {c_2} + O(\frac{1}{q^2}).
	\end{align}
	We introduce the following $z$- and $t$-dependent deterministic parameters
	\begin{align}
	\alpha_1(z) := \im \widetilde{m}_t(z), \quad \alpha_2(z):=P'(\widetilde{m}_t(z)), \quad \beta:= \frac{1}{N\eta} + \frac{1}{q_t^2},
	\end{align}
	with $z=E+\ii\eta$. 
	We note that 
	\[
	\abs{\alpha_2} \geq \abs{P_2(\wt m_t)}\abs{P'_1(\wt m_t)} \geq \frac{1}{c_1}\im P'_1(\wt m_t) \geq \frac{1}{c_1} \im \wt m_t =  \frac{1}{c_1}\alpha_1
	\]
	Further let 
	\begin{align}\label{def:Lambda_t}
	\Lambda_t(z):=|m_t(z) - \widetilde{m}_t(z) |,\quad (z\in\mathbb{C}^+).
	\end{align}
	Note that from weak local law for the cgSBM \eqref{thm:weak law 3}, we have that $\Lambda_t(z)\prec 1$ uniformly on $\mathcal{D}_\ell$. Since $P_1(\wt m_t) =0$, we have $P'(\wt m_t) = P_1'(\wt m_t) P_2 (\wt m_t)$.
	Similar as in the proof of Lemma 5.1 of \cite{LS18}, we have  
	\[
	\abs{\alpha_2} = \abs{P' (\wt m_t)} = |P_1'(\wt m_t)|| P_2 (\wt m_t)| \sim \sqrt{\kappa_t + \eta}.
	\]
	
	Recall that Young's inequality states that for any $a,b >0$ and $x,y>1$ with $x^{-1}+y^{-1}=1$,
	\begin{equation}\label{eq:young}
	ab \leq \frac{a^x}{x}+\frac{b^y}{y}.
	\end{equation}

	Let $D\geq 10$. Choose any small $\epsilon >0$. The strategy is now as follows. We apply Young's inequality \eqref{eq:young} to split up all the terms on the right hand side of~\eqref{eq:recursive} and absorb resulting factors of $\mathbb{E}|P(m_t)|^{2D}$ into the left hand side. For the first term on the right of~\eqref{eq:recursive}, we get, upon using applying \eqref{eq:young} with $x=2D$ and $y=2D/(2D-1)$, that
	\begin{align}\label{eq:recursive line1}
	N^{\epsilon}\Big(&\frac{\im m_t}{N\eta} + q_t^{-4}\Big)|P(m_t)|^{2D-1}\\ &\leq  N^{\epsilon} \frac{\alpha_1 + \Lambda_t}{N\eta}|P(m_t)|^{2D-1}+ N^{\epsilon}  q_t^{-4}|P(m_t)|^{2D-1} \notag\\
	&\leq \frac{N^{(2D+1)\epsilon}}{2D}\beta^{2D}(\alpha_1+\Lambda_t )^{2D} + \frac{N^{(2D+1)\epsilon}}{2D}q_t^{-8D}+\frac{2(2D-1)}{2D}N^{-\frac{\epsilon}{2D-1}}|P(m_t)|^{2D}, \notag 
	\end{align}
	since $(N\eta)^{-1} \leq \beta $ and note that the last term can be absorbed into the left side of~\eqref{eq:recursive}. The same idea can be applied to the second term on the right side of~\eqref{eq:recursive}. Hence we have 
	\begin{equation}\label{eq:recursive line1.5}
	N^{-\epsilon/8}q_t^{-1}\Lambda^2_t|P(m_t)|^{2D-1} \leq \frac{N^{-(D/4-1)\epsilon}}{2D}q^{-2D}_t\Lambda_t^{4D}+\frac{2D-1}{2D}N^{-\frac{\epsilon}{2D-1}}\abs{P(m_t)}^{2D}.
	\end{equation}
	To handle the other terms, we Taylor expand $P'(m_t)$ around $\widetilde{m}_t$ as
	\begin{align}
	|P'(m_t) - \alpha_2 -P''(\widetilde{m}_t)(m_t-\widetilde{m}_t)| \leq C\Lambda_t^2.
	\end{align}
	Therefore, for some constant $C_1$, we get
	\begin{equation}
	\abs{P'(m)} \leq \abs{\alpha_2} + C_1\Lambda_t,
	\end{equation}
	for all $z \in \mathcal{D_\ell}$, with high probability.
	Note that for any fixed $s \geq 2$, 
	\begin{align*}
	(\alpha_1+\Lambda_t)^{2s-s'-2} (\abs{\alpha_2} + C_1\Lambda_t)^{s'} &\leq N^{\epsilon/2}(\alpha_1+\Lambda_t)^{s-1} (\abs{\alpha_2} + C_1\Lambda_t)^{s-1}\\
	&\leq N^{\epsilon}(\alpha_1+\Lambda_t)^{s/2} (\abs{\alpha_2} + C_1\Lambda_t)^{s/2}
	\end{align*}
	with high probability, uniformly in $\mathcal{D}_\ell$, since $\alpha_1 \leq c_1 \abs{\alpha_2} \leq C$ and $\Lambda_t \prec 1$.
	Also note that $2s-s'-2 \geq s$ since $s' \leq s-2$. Therefore for the second line in \eqref{eq:recursive}, for $2\leq s\leq 2D$, 
	\begin{align}\label{eq:recursive line2}
	N^{\epsilon} q_t^{-1}&\Big(\frac{\im m_t}{N\eta}\Big)^{2s-s'-2}|P'(m_t)|^{s'}|P(m_t)|^{2D-s} \notag \\
	&\leq N^{\epsilon}q_t^{-1}\beta^s(\alpha_1+\Lambda_t)^{2s-s'-2} (\abs{\alpha_2} + C_1\Lambda_t)^{s'}|P(m_t)|^{2D-s}  \\
	&\leq N^{2\epsilon}q_t^{-1}\beta^s(\alpha_1+\Lambda_t)^{s/2} (\abs{\alpha_2} + C_1\Lambda_t)^{s/2}|P(m_t)|^{2D-s}  \notag\\
	&\leq  N^{2\epsilon}q_t^{-1}\frac{s}{2D}\beta^{2D}(\alpha_1+\Lambda_t)^{D} (\abs{\alpha_2} + C_1\Lambda_t)^{D} + N^{2\epsilon}q_t^{-1}\frac{2D-s}{2D}|P(m_t)|^{2D} \notag
	\end{align}
	uniformly on $\mathcal{D}_\ell$ with high probability.
	Similarly, for the last term in \eqref{eq:recursive}, for $2\leq s\leq 2D$, we obtain
	\begin{align}\label{eq:recursive line3}
	N^{\epsilon}&\Big(\frac{1}{N\eta}+\frac{1}{q_t}\Big(\frac{\im m_t}{N\eta} \Big)^{1/2}+\frac{1}{q_t^2}\Big) \Big(\frac{\im m_t}{N\eta} \Big)^{s-1}|P'(m_t)|^{s-1}|P(m_t)|^{2D-s} \notag \\
	& \leq  N^{2\epsilon}\beta\cdot \beta^{s-1}(\alpha_1+\Lambda_t)^{s/2} (\abs{\alpha_2} + C_1\Lambda_t)^{s/2}|P(m_t)|^{2D-s} \\ 
	&\leq \frac{s}{2D}\pB{N^{2\epsilon}N^{\frac{(2D-s)\epsilon}{4D^2}}}^\frac{2D}{s}\beta^{2D}(\alpha_1+\Lambda_t)^{D} (\abs{\alpha_2} + C_1\Lambda_t)^{D} + \frac{2D-s}{2D}\pB{N^{-\frac{(2D-s)\epsilon}{4D^2}}}^{\frac{2D}{2D-s}}|P(m_t)|^{2D} \notag \\
	& \leq N^{(2D+1)\epsilon}\beta^{2D}(\alpha_1+\Lambda_t)^{D} (\abs{\alpha_2} + C_1\Lambda_t)^{D} + N^{-\frac{\epsilon}{2D}}|P(m_t)|^{2D}, \notag
	\end{align}
	uniformly on $\mathcal{D_\ell}$ with high probability where we used 
	\begin{equation}\label{eq:beta bound}
	\frac{1}{N\eta}+\frac{1}{q_t}\Big(\frac{\im m_t}{N\eta} \Big)^{1/2}+\frac{1}{q_t^2} \prec \beta.
	\end{equation}
	From \eqref{eq:recursive}, \eqref{eq:recursive line1}, \eqref{eq:recursive line1.5}, \eqref{eq:recursive line2} and \eqref{eq:recursive line3}
	\begin{align}
	\mathbb{E}[|P(m_t)|^{2D}] \leq &N^{(2D+1)\epsilon}\mathbb{E}[\beta^{2D}(\alpha_1 +\Lambda_t)^D(|\alpha_2|+C_1\Lambda_t)^D]+\frac{N^{(2D+1)\epsilon}}{2D}q_t^{-8D} \\
	& +\frac{N^{-(D/4-1)\epsilon}}{2D}q_t^{-2D}\mathbb{E}[\Lambda_t^{4D}] +CN^{-\frac{\epsilon}{2D}}\E [|P(m_t)|^{2D}],\notag 
	\end{align}
	for all $z\in\mathcal{D}_\ell$.
	Since the last term can be absorbred into the left side, we eventually find
	\begin{align}\label{eq:recursive1}
	&\mathbb{E}[|P(m_t)|^{2D}] \notag\\
	&\leq CN^{(2D+1)\epsilon}\mathbb{E}[\beta^{2D}(\alpha_1 +\Lambda_t)^D(|\alpha_2|+C_1\Lambda_t)^D]+C\frac{N^{(2D+1)\epsilon}}{2D}q_t^{-8D} +C\frac{N^{-(D/4-1)\epsilon}}{2D}q_t^{-2D}\mathbb{E}[\Lambda_t^{4D}]  \\
	&\leq N^{3D\epsilon}\beta^{2D}|\alpha_2|^{2D} +N^{3D\epsilon}\beta^{2D}\mathbb{E}[\Lambda_t^{2D}]+N^{3D\epsilon}q_t^{-8D}+N^{-D\epsilon/8}q_t^{-2D}\mathbb{E}[\Lambda_t^{4D}], \notag
	\end{align}
	uniformly on $\mathcal{D}_\ell$, where we used $\alpha_1\le c_1 |\alpha_2|$ and the inequality
	\begin{equation}\label{eq:schwarz}
	(a+b)^x \leq 2^{x-1}(a^x+b^x)
	\end{equation}
	for any $a,b \geq 0$ and $x\geq 1$ with $D>10$  to get the last line.

	Now, we aim to control $\Lambda_t$ in terms of $|P(m_t)|$. For that, from the third order Taylor expansion of $P(m_t)$ around $\widetilde m_t$ to get
	\begin{align} \label{eq:3rdTaylor}
	\Big|P(m_t)-P' (\wt m_t)(m_t-\widetilde{m}_t)-\frac{1}{2}P''(\widetilde{m}_t)(m_t-\widetilde{m}_t)^2\Big|\leq C\Lambda_t^3,
	\end{align}
	since $P(\widetilde{m}_t)=0$ and $P'''(\widetilde{m}_t) \sim 1$. Then using $\Lambda_t\prec 1$ and $P''(\widetilde{m}_t) \geq C+O(q_t^{-2})$ we obtain
	\begin{align} \label{eq:Lambdasq}
	\Lambda_t^2\prec |\alpha_2|\Lambda_t +|P(m_t)|, \quad (z\in\mathcal{D}_\ell).
	\end{align}	
	Taking the $2D$-power of \eqref{eq:Lambdasq}, using \eqref{eq:schwarz} again, and taking the expectation, we get 
	\begin{equation}
	\mathbb{E}[\Lambda_{t}^{2D}]\leq 4^{2D}N^{\epsilon/2}|\alpha_2|^{2D}\E[\Lambda_{t}^{2D}]+4^{2D}N^{\epsilon/2}\E[|P(m_t)|^{2D}]
	\end{equation}	
	Replacing form \eqref{eq:recursive1} for $\E[|P(m_t)|^{2D}]$, for sufficiently large $N$, we obtain
	\begin{align}
	\E [\Lambda_t ^{4D}]\leq &N^{\epsilon}|\alpha_2|^{2D}\E[\Lambda_{t}^{2D}] + N^{(3D+1)\epsilon}\beta^{2D}|\alpha_2|^{2D} +N^{(3D+1)\epsilon}\beta^{2D}\mathbb{E}[\Lambda_t^{2D}]+N^{(3D+1)\epsilon}q_t^{-8D}\notag\\
	&+N^{-D\epsilon/8+\epsilon}q_t^{-2D}\mathbb{E}[\Lambda_t^{4D}], 
	\end{align}
	uniformly on $\mathcal{D}_\ell$. Using Schwarz inequality for the first term and the third term on the right, absorbing the terms $o(1)\E[\Lambda_t^{4D}]$ into the left side and using \eqref{eq:beta bound} we get 
	\begin{equation}
	\mathbb{E}[\Lambda_{t}^{4D}]\leq N^{2\epsilon}|\alpha_2|^{4D} N^{(3D+2)\epsilon}\beta^{2D}|\alpha_2^{2D}|+N^{(3D+2)\epsilon}\beta^{4D},
	\end{equation}
	uniformly on $\mathcal{D}_\ell$.
	This estimate can be fed back into \eqref{eq:recursive1}, to get the bound
	\begin{align}
	\mathbb{E}[|P(m_t)|^{2D}]& \leq N^{3D\epsilon}\beta^{2D}|\alpha_2|^{2D}+N^{3D\epsilon}\beta^{4D}\E[\Lambda^{2D}]+N^{(3D+1)\epsilon}\beta^{4D}+ q_t^{-2D}|\alpha_2|^{4D} \notag \\
	&\leq N^{5D\epsilon}\beta^{2D}|\alpha_2|^{2D}+N^{5D\epsilon}\beta^{4D}+q_t^{-2D}|\alpha_2|^{4D},
	\end{align}
	uniformly on $\mathcal{D}_\ell$ for sufficiently large $N$. 
	
	For any fixed $z\in\mathcal{D}_\ell$, Markov's inequality then yields $|P(m_t)|\prec|\alpha_2|\beta+\beta^2+q_t^{-1}|\alpha_2|^2$. Then we can obtain from the Taylor expansion of $P(m_t)$ around $\wt m_t$ in \eqref{eq:3rdTaylor} that 
	\begin{equation}
	|\alpha_2(m_t-\widetilde{m}_t)+\frac{P''(\wt m)}{2}(m_t-\widetilde{m}_t)^2|\prec\psi\Lambda_t^2 +|\alpha_2|\beta+\beta^2+q_t^{-1}|\alpha_2|^2,
	\end{equation}
	for each fixed $z\in\mathcal{D}_\ell$, where $\psi$, defined in Theorem \ref{thm:weak local law}, satisfies $\psi\geq q_t^{-2}$.
	Uniformity in $z$ is easily achieved using a lattice argument and the Lipschitz continuity of $m_t(z)$ and $\widetilde{m}_t(z)$ on $\mathcal{D}_\ell$. Furthermore, for any (small) $\epsilon >0$ and (large) $D$ there is an event $\wt \Xi$ with $\P(\wt \Xi)\geq 1-N^D$ such that for all $z \in \mathcal{D}_\ell$,
	\begin{equation}\label{eq:2ndTaylor}
	|\alpha_2(m_t-\widetilde{m}_t)+\frac{P''(\wt m)}{2}(m_t-\widetilde{m}_t)^2|\leq N^\epsilon\psi\Lambda_t^2 +N^\epsilon|\alpha_2|\beta+N^\epsilon\beta^2+N^\epsilon q_t^{-1}|\alpha_2|^2,
	\end{equation}
	on $\wt \Xi$, for $N$ sufficiently large.

	Recall that there exists a constant $C_0>1$ which satisfies $C_0^{-1}\sqrt{\kappa_t(E)+\eta}\leq|\alpha_2|\leq C_0\sqrt{\kappa_t(E)+\eta}$, where we can choose $C_0$ uniform in $z\in\mathcal{D}_\ell$. Note that, for a fixed $E$, $\beta=\beta(E+\ii\eta)$ is a decreasing function of $\eta$ whereas $\sqrt{\kappa_t(E)+\eta}$ is increasing. Hence there is $\widetilde{\eta_0}\equiv\widetilde{\eta_0}(E)$ such that $\sqrt{\kappa(E)+\widetilde{\eta_0}}=C_0q_t\beta(E+\ii\widetilde{\eta}_0)$. We consider the subdomain $\widetilde{\mathcal{D}}\subset\mathcal{D}_\ell$ defined by
	\begin{align}
	\widetilde{\mathcal{D}}:=\{z=E+\ii\eta\in\mathcal{D}_\ell:\eta>\widetilde{\eta}_0(E)\}.
	\end{align}
	On this subdomain $\widetilde{\mathcal{D}}$, $\beta\leq q_t^{-1}|\alpha_2|$, hence we get from \eqref{eq:2ndTaylor} that there is a high probability event $\widetilde{\Upxi}$ such that
	\begin{align*}
	|\alpha_2(m_t-\widetilde{m}_t)+\frac{P''(\wt m)}{2}(m_t-\widetilde{m}_t)^2|\leq o(1)\Lambda_t^2+3N^\epsilon q_t^{-1}|\alpha_2|^2
	\end{align*}
	and thus
	\begin{align*}
	|\alpha_2|\Lambda_t\leq (\frac{c_2}{2}+o(1))\Lambda_t^2 +3N^\epsilon q_t^{-1}|\alpha_2|^2
	\end{align*}
	uniformly on $\widetilde{\mathcal{D}}$ on $\widetilde{\Upxi}$. Hence, on $\widetilde{\Upxi}$, we have either
	\begin{align}
	|\alpha_2|\leq 2c_2\Lambda_t\quad \textrm{or}\quad \Lambda_t\leq 6N^\epsilon q_t^{-1}|\alpha_2|,\quad(z\in\widetilde{\mathcal{D}}).
	\end{align}
	When $\eta=3$, it is easy to see that
	\begin{align}
	|\alpha_2|\geq |P_2(\wt m)||P'_1(\wt m)| \geq \frac{1}{c_1} \pB{|z+2\wt m|-C\frac{1}{q^2}} \geq \frac{\eta}{c_1}=\frac{3}{c_1} \gg 6N^\epsilon q_t^{-1}|\alpha_2|,
	\end{align}
	for sufficiently large $N$. From the a priory estimate, we know that $|\Lambda_t|\prec \psi$, we hence find that
	\begin{align}\label{eq:dicho}
	\Lambda_t\leq 6N^\epsilon q_t^{-1}|\alpha_2|,
	\end{align}
	holds for $z \in \wt{\mathcal{D}}$ on the event $\widetilde{\Upxi}$. 	
	Putting~\eqref{eq:dicho} back into~\eqref{eq:recursive1}, we obtain that
	\begin{align}
	\mathbb{E}[|P(m_t)|^{2D}]&\leq N^{4D\epsilon}\beta^{2D}|\alpha_2|^{2D} +N^{3D\epsilon}q_t^{-8D}+q_t^{-6D}|\alpha_2|^{4D} \notag \\
	&\leq N^{6D\epsilon}\beta^{2D}|\alpha_2|^{2D}+N^{6D\epsilon}\beta^{4D},
	\end{align}
	for any small $\epsilon>0$, and large $D$, uniformly on $\widetilde{\mathcal{D}}$. For $z\in\mathcal{D}_\ell\backslash\widetilde{\mathcal{D}}$, it is direct to check the estimate $\mathbb{E}[|P(m_t)|^{2D}]\leq N^{6D\epsilon}\beta^{2D}|\alpha_2|^{2D}+N^{6D\epsilon}\beta^{4D}$. Using a lattice argument and the Lipschitz continuity, we find from a union bound that for any small $\epsilon>0$ and large $D$ there exists an event $\Upxi$ with $\mathbb{P}(\Upxi)\geq 1-N^{-D}$ such that
	\begin{align}\label{eq:2ndTaylor2}
	|\alpha_2(m_t-\widetilde{m}_t)+\frac{P''(\wt m)}{2}(m_t-\widetilde{m}_t)^2|\leq N^\epsilon\psi\Lambda_t^2 +N^\epsilon|\alpha_2|\beta+N^\epsilon\beta^2
	\end{align}
	on $\Upxi$, uniformly on $\mathcal{D}_\ell$ for any sufficiently large $N$.

	Recall that for fixed $E$, $\beta=\beta(E+\ii\eta)$ is a decreasing function of $\eta$, $\sqrt{\kappa_{t}(E) +\eta}$ is an increasing function of $\eta$, and $\eta_0\equiv \eta_0(E)$ satisfies that $\sqrt{\kappa(E)+\eta_0} = 10C_0 N^{\epsilon}\beta(E+\ii\eta_0)$. Further notice that $\eta_0(E)$ is a continuous function. We consider the following subdomains of $\mathcal{E}$:
	\begin{align*}
	\mathcal{E}_1&:=\{z=E+\ii\eta\in\mathcal{E} : \eta\leq \eta_0(E), 10N^{\epsilon}\leq N\eta\}, \\
	\mathcal{E}_2&:=\{z=E+\ii\eta\in\mathcal{E} : \eta>\eta_0(E), 10N^{\epsilon}\leq N\eta\}, \\
	\mathcal{E}_3&:=\{z=E+\ii\eta\in\mathcal{E} :  10N^{\epsilon}\geq N\eta\}.
	\end{align*}
	We consider the cases $z\in\mathcal{E}_1$, $z\in\mathcal{E}_2$ and $z\in\mathcal{E}_3$, and split the stability analysis accordingly. Let $\Upxi$ be a high probability event such that~\eqref{eq:2ndTaylor2} holds. Note that we can choose $\ell$ sufficiently small that satisfies $\mathcal{D_\ell} \supset \mathcal{E}_1 \cup \mathcal{E}_2$.

	{\it Case 1:} If $z\in\mathcal{E}_1$, we note that $|\alpha_2|\leq C_0\sqrt{\kappa(E)+\eta}\leq 10C_0^2 N^{\epsilon}\beta(E+\ii\eta)$. Then, we find that
	\begin{align*}
	\absB{\frac{P''(\wt m)}{2}}\Lambda_t^2 &\leq|\alpha_2|\Lambda_t + N^{\epsilon}\psi\Lambda_t^2 + N^{\epsilon}|\alpha_2|\beta + N^{\epsilon}\beta^2 \\ 
	&\leq 10C_0^2 N^{\epsilon}\beta\Lambda_t + N^{\epsilon}\psi\Lambda_t^2 + (10C_0^2 N^{\epsilon}+1) N^{\epsilon}\beta^2,
	\end{align*}
	on $\Upxi$. Hence, there is some finite constant $C$ such that on $\Upxi$, we have $\Lambda_t\leq CN^{\epsilon}\beta$, $z\in\mathcal{E}_1$.
	
	{\it Case 2:} If $z\in\mathcal{E}_2$, we obtain that
	\begin{align}
	|\alpha_2|\Lambda_t\leq (\absB{\frac{P''(\wt m)}{2}}+N^{\epsilon}\psi)\Lambda_t^2 + |\alpha_2|N^{\epsilon}\beta + N^{\epsilon}\beta^2,
	\end{align}
	on $\Upxi$. We then notice that $C_0|\alpha_2|\geq \sqrt{\kappa_t(E)+\eta}\geq 10C_0 N^{\epsilon}\beta$, \emph{i.e.} $N^{\epsilon}\beta\leq |\alpha_2|/10$, so that
	\begin{align}
	|\alpha_2|\Lambda_t\leq (\absB{\frac{P''(\wt m)}{2}}+N^{\epsilon}\psi)\Lambda_t^2 + (1+N^{-\epsilon})|\alpha_2|\beta 
	\leq  c_2\Lambda_t^2 + (1+N^{-\epsilon})|\alpha_2|\beta
	\end{align}
	on $\Upxi$, where we used that $N^{\epsilon}\psi\leq 1$. Hence, on $\Upxi$, we have either
	\begin{align}
	|\alpha_2|\leq 2c_2\Lambda_t \quad\textrm{or}\quad \Lambda_t\leq 3N^{\epsilon}\beta.
	\end{align}
	We use the dichotomy argument and the continuity argument similarly to the strategy to get~\eqref{eq:dicho}. Since $3N^{\epsilon}\beta\leq |\alpha_2|/8$ on $\mathcal{E}_2$, by continuity, we find  that on the event $\Upxi$, $\Lambda_t\leq 3N^{\epsilon}\beta$ for $z\in\mathcal{E}_2$.
	
	{\it Case 3:} For $z\in  \mathcal{E}_3$ we use that $|m'_t(z)|\leq \frac{\im m_t(z)}{\im z}, z\in\mathbb{C}^{+}$. Note that $m_t$ is a Stieltjes transform of a probability measure. Set $\widetilde{\eta}:=10N^{-1+\epsilon}$ and observe that
	\begin{align}
	|m_t(E+\ii\eta)|
	&\leq \int_{\eta}^{\widetilde{\eta}}\frac{s\im m_t(E+\ii s)}{s^2} \di s+ \Lambda_t(E+\ii\widetilde{\eta})+|\widetilde{m}_t(E+\ii \widetilde{\eta})|.
	\end{align}
	It is easy to check that $s\rightarrow s\im m_t(E+\ii s)$ is monotone increasing. Thus, we find that
	\begin{align}
	|m_t(E+\ii\eta)| &\leq\frac{2\widetilde{\eta}}{\eta}\im m_t(E+\ii\widetilde{\eta})+\Lambda_t(E+\ii\widetilde{\eta}) + |\widetilde{m}_t(E+\ii\widetilde{\eta})| \notag \\
	&\leq C\frac{N^{\epsilon}}{N\eta}\big(\im \widetilde{m}_t(E+\ii\widetilde{\eta})+\Lambda_t(E+\ii\widetilde{\eta})\big) + |\widetilde{m}_t(E+\ii\widetilde{\eta})|,
	\end{align}
	for some $C$ where we used $\widetilde{\eta}=10N^{-1+\epsilon}$ to obtain the second inequality. Since $z=E+\ii\widetilde{\eta}\in\mathcal{E}_1 \cup \mathcal{E}_2$, we have $\Lambda_t(E+\ii\widetilde{\eta})\leq CN^{\epsilon}\beta(E+\ii\widetilde{\eta})\leq C$ on $\Upxi$. Since $\widetilde{m}_t$ is uniformly bounded on $\mathcal{E}$, we obtain that $\Lambda_t\leq CN^{\epsilon}\beta$ on $\Upxi$, for all $z\in \mathcal{E}_3$.

	To sum up, we get $\Lambda_t\prec\beta$ uniformly on $\mathcal{E}$ for fixed $t\in[0,6\log N]$. Choosing $t=0$, we have proved Theorem \ref{thm:locallaw}. Now we use the continuity of the Dyson matrix flow to prove that this result holds for all $t\in[0,6\log N]$. Consider a lattice $\mathcal{L}\subset[0,6\log N]$ with spacings of order $N^{-3}$. Then we obtain that $\Lambda_t\prec\beta$, uniformly on $\mathcal{E}$ and on $\mathcal{L}$, by a union bound. Thus, by continuity, we can extend the conclusion to all $t\in[0,6\log N]$ and conclude the proof of Proposition~\ref{prop:locallaw}.
\end{proof}

\subsection{Proof of Theorem \ref{thm:matrix norm}} 
Theorem \ref{thm:matrix norm} follows directly from the following result.
\begin{lemma}\label{lem:Hnorm}
	Suppose that $H_0$ satisfy Assumption \ref{assumption} with $\phi >0$. Then,
	\begin{equation}\label{eq:norm H_t}
	\abs{\norm{H_t}-L_t} \prec \frac{1}{q_t^4}+\frac{1}{N^{2/3}},
	\end{equation}
	uniformly in $t \in [0, 6\log N]$.
\end{lemma}

For the proof of Lemma \ref{lem:Hnorm}, the similar strategy to the ones in \cite{LS18} and \cite{Hwang2018} can be applied. We establish the upper bound on the largest eigenvalue of $H_t$, using a stability analysis starting from \eqref{eq:recursive1} and the fact that $\alpha_1(z) = \im\wt{m}_t$ behaves as $\eta/\sqrt{\kappa_t (E)+\eta}$, for $E\geq L_+$. The detailed proof is given in Appendix \ref{app:Hnorm}.

\section{Proof of Tracy--Widom limit}\label{sec:TW}
In this section, we prove the Theorem \ref{thm:TWlimit}, the Tracy--Widom limiting distribution of the largest eigenvalue. Following the idea from \cite{EYY12}, we consider the imaginary part of the normalized trace of the Green function $m\equiv m^{H}$ of $H$. For $\eta>0$, define
\begin{align}
\theta_\eta(y)=\frac{\eta}{\pi(y^2+\eta^2)},\quad(y\in\mathbb{R}).
\end{align}
From the definition of the Green function, one can easily check that	
\begin{align}
\im m(E+\ii\eta)=\frac{\pi}{N}\tr \theta_\eta(H-E).
\end{align}

The first proposition in this section shows how we can approximate the distribution of the largest eigenvalue by using the Green function. Recall that $L_+$ is the right endpoint of the deterministic probability measure in Theorem \ref{thm:locallaw}. 
\begin{proposition} \label{prop:Greenftn}
	Let $H$ satisfy Assumption \ref{assumption}, with $\phi>1/6$. Denote by $\lambda_1^{H}$ the largest eigenvalue of $H$. Fix $\epsilon>0$ and let $E\in\mathbb{R}$ be such that $|E-L|\leq N^{-2/3+\epsilon}$. Set $E_{+} :=L+2N^{-2/3+\epsilon}$ and define $\chi_E:=\mathds{1}_{[E,E_{+}]}.$ Let $\eta_1:=N^{-2/3-3\epsilon}$ and $\eta_2:=N^{-2/3-9\epsilon}$. Let $K:\mathbb{R}\rightarrow[0,\infty)$ be a smooth function satisfying
	\begin{align}
	K(x)=\left\{\begin{array}{ll}
	1&\textrm{if $|x|<1/3$} \\
	0&\textrm{if $|x|>2/3$},
	\end{array}\right.
	\end{align}
	which is a monotone decreasing on $[0,\infty)$. Then, for any $D>0$,
	\begin{align}
	\mathbb{P}(\lambda_1^{H}\leq E-\eta_1 )-N^{-D}<\mathbb{E}[K(\tr(\chi_E *\theta_{\eta_2})(H))]<\mathbb{P}(\lambda_1^{H}\leq E+\eta_1 )+N^{-D}
	\end{align}
	for $N$ sufficiently large, with $\theta_{\eta_2}$.
\end{proposition}
We refer to Proposition 7.1 of \cite{LS18} for the proof. We remark that the lack of the improved local law near the lower edge does not alter the proof of Proposition \ref{prop:Greenftn}. 

Define $W^{G}$ be a $N\times N$ generalized Wigner matrix independent of $H$ with Gaussian entries $W^G_{ij}$ satisfying 
\begin{align*}
\mathbb{E}W^G_{ij} = 0, \qquad \mathbb{E}|W^G_{ij}|^2 = \mathbb{E}|H_{ij}|^2, 
\end{align*}
and denote by $m^{G}\equiv m^{W^{G}}$ the normalized trace of its Green function. The following is the Green function comparison for our model.
\begin{proposition} \label{prop:Greencomparison}
	Under the assumptions of Proposition~\ref{prop:Greenftn} the following holds. Let $\epsilon>0$ and set $\eta_0=N^{-2/3-\epsilon}$. Let $E_1, E_2\in\mathbb{R}$ satisfy $|E_1|,|E_2|\leq N^{-2/3+\epsilon}$. Consider a smooth function $F:\mathbb{R}\rightarrow\mathbb{R}$ such that
	\begin{align}
	\max_{x\in\mathbb{R}}|F^{(l)}(x)|(|x|+1)^{-C}\leq C,\quad(l\in\mathbf{[}1, 11\mathbf{]}).
	\end{align}
	Then, for any sufficiently small $\epsilon>0$, there exists $\delta>0$ such that
	\begin{align}
	\left| \mathbb{E}F\Big(N\int^{E_2}_{E_1}\im m(x+L_++\ii\eta_0)dx\Big) -\mathbb{E}F\Big(N\int^{E_2}_{E_1}\im m^{G}(x+\lambda_{+}+\ii\eta_0)dx\Big) \right|\leq N^{-\delta},
	\end{align}
	for large enough $N$.
\end{proposition}

From Theorem 2.7 of \cite{BEY14}, we know that the largest eigenvalue of the generalized Wigner matrix follows the Tracy--Widom distribution.
Thus, Proposition \ref{prop:Greencomparison} directly implies Theorem \ref{thm:TWlimit}, the Tracy--Widom limit for the largest eigenvalue.  A detailed proof is found, {\it e.g.}, with the same notation in~\cite{LS18}, Section 7.

In the remainder of the section, we prove Proposition \ref{prop:Greencomparison}. We begin by the following application of the generalized Stein lemma.
\begin{lemma}\label{lemma:appStein}
	Fix $\ell \in \mathbb{N}$ and let $F\in C^{\ell+1} (\mathbb{R}; \mathbb{C}^+).$ Let $Y\equiv Y_0$ be a random variable with finite moments to order $\ell+2$ and let $W$ be a Gaussian random variable independent of $Y$. Assume that $\mathbb{E}[Y]=\mathbb{E}[W]=0$ and $\mathbb{E}[Y^2] =\mathbb{E}[W^2]$. Introduce
	\begin{align}
	Y_t := \mathrm{e}^{-t/2} Y_0 +\sqrt{1-\mathrm{e}^{-t}}W,
	\end{align}
	and let $\dot Y_t\equiv \di Y_t /\di t$.  Then,
	\begin{align}\label{eq:genStein}
	\mathbb{E}\Big[ \dot Y_t F(Y_t)\Big] = -\frac{1}{2} \sum_{r=2}^{\ell} \frac{\kappa^{(r+1)}(Y_0)}{r!} \mathrm{e}^{-\frac{(r+1)t}{2}}\mathbb{E} \big[F^{(r)}(Y_t)\big] + \mathbb{E}\big[\Omega_\ell (\dot Y_t F(Y_t))\big],
	\end{align}
	where $\mathbb{E}$ denotes the expectation with respect to $Y$ and $W$, $\kappa^{(r+1)}(Y)$ denotes the $(r+1)$-th cumulant of $Y$ and $F^{(r)}$ denotes the $r$-th derivative of the function $F$. The error term $\Omega_\ell$ in \eqref{eq:genStein} satisfies
	\begin{align}
	\big|\mathbb{E}\big[\Omega_\ell (\dot Y_t F(Y_t))\big]\big|&
	\leq C_\ell \mathbb{E}[|Y_t|^{|\ell+2}] \sup_{|x|\leq Q} |F^{(\ell+1)} (x)|\notag\\ &\qquad+ C_\ell \mathbb{E}[|Y_t|^{\ell+2}\mathbf{1}(|Y_t|>Q)]\sup_{x\in \mathbb{R}}|F^{(\ell+1)} (x)|, 
	\end{align}
	where $Q> 0$ is an arbitrary fixed cutoff and $C_\ell$ satisfies $C_\ell \leq \frac{(C\ell)^{\ell}}{\ell !}$ for some numerical constant $C$.
\end{lemma}
\begin{proof}[Proof of Proposition \ref{prop:Greencomparison}] Fix a (small) $\epsilon>0$. Consider $x\in [E_1, E_2]$. For simplicity, let
	\begin{align}
	G\equiv G_t(x+L_t +\ii\eta_0), \quad m\equiv m_t(x+L_t+\ii\eta_0),
	\end{align}
	with $\eta_0=N^{-2/3-\epsilon}$, and define
	\begin{align}
	X\equiv X_t:=N\int^{E_2}_{E_1}\im m(x+L_t +\ii\eta_0) \di x.
	\end{align}
	Note that $X\prec N^{\epsilon}$ and $|F^{(l)}(X)|\prec N^{C\epsilon}$ for $l\in[1,11]$. From \eqref{eq:L_t} we can obtain that
	\begin{align*}
	L_t=2+\mathrm{e}^{-t}\xi^{(4)}q_t^{-2}+O(\mathrm{e}^{-2t}q_t^{-4}),\qquad
	\dot{L_t}=-2\mathrm{e}^{-t}\xi^{(4)}q_t^{-2}+O(\mathrm{e}^{-2t}q_t^{-4}),
	\end{align*}
	with $q_t=\mathrm{e}^{t/2}q_0$, where $\dot{L_t}$ denotes the derivative with respect to $t$ of $L_t$. Let $z=x+L_t+\ii\eta_0$ and $G\equiv G(z)$. Differentiating $F(X)$ with respect to $t$, we obtain
	\begin{align}
	\frac{\textrm{d}}{\textrm{d}t}\mathbb{E}F(X) &= \mathbb{E}\Big[F'(X)\frac{\textrm{d}X}{\textrm{d}t}\Big]=\mathbb{E}\Big[F'(X)\im\int^{E_2}_{E_1}\sum^{N}_{i=1}\frac{\textrm{d}G_{ii}}{\textrm{d}t}\textrm{d}x\Big] \notag \\
	&=\mathbb{E}\Big[F'(X)\im\int^{E_2}_{E_1}\Big(\sum_{i,j,k}\dot{H}_{jk}\frac{\partial G_{ii}}{\partial H_{jk}}+\dot{L_t}\sum_{1\leq i,j\leq N} G_{ij}G_{ji}\Big)\textrm{d}x\Big], \label{eq:F_derivative}
	\end{align}
	where by definition
	\begin{align}
	\dot{H}_{jk}\equiv(\dot{H}_{t})_{jk}=-\frac{1}{2}\mathrm{e}^{-t/2}(H_0)_{jk}+\frac{\mathrm{e}^{-t}}{2\sqrt{1-\mathrm{e}^{-t}}}W^{G}_{jk}.
	\end{align}
	Thus, we find that
	\begin{align}\label{eq:F'term1}
	&\sum_{i,j,k}\mathbb{E}\Big[\dot{H}_{jk}F'(X)\frac{\partial G_{ii}}{\partial H_{jk}}\Big] = -2 \sum_{i,j,k}\mathbb{E}\Big[\dot{H}_{jk}F'(X)G_{ij}G_{ki}\Big] \notag \\
	&=\frac{\mathrm{e}^{-t}}{N}\sum_{r=2}^{\ell}\frac{q_t^{-(r-1)}}{r!}\sum_{i}\sum_{j\neq k}\mathbb{E}[s_{(jk)}^{(r+1)}\partial_{jk}^r (F'(X)G_{ij}G_{ki})] + O(N^{1/3 +C\epsilon}),
	\end{align}
	for $\ell=10$, by Lemma \ref{lemma:appStein}, where we use the short hand $\partial_{jk}=\partial/\partial H_{jk}$. Here, the error term $O(N^{1/3 +C\epsilon})$ in \eqref{eq:F'term1} corresponds to $\Omega_\ell$ in \eqref{eq:genStein}, which is $O(N^{C\epsilon} N^2 q_t^{-10})$ for $X=H_{ji}$. To estimate the right-hand side of \eqref{thm:locallaw}, we use the following lemma whose proof is in Appendix \ref{app:TW}.
	
	\begin{lemma}\label{lemma:A_r}
		For an integer $r\geq 2$, let
		\begin{align}
		A_r:=\frac{\mathrm{e}^{-t}}{N}\frac{q_t^{-(r-1)}}{r!}\sum_{i}\sum_{j\neq k}\mathbb{E}[s_{(jk)}^{(r+1)}\partial_{jk}^r (F'(X)G_{ij}G_{k i})].
		\end{align}
		Then, for any $r\neq 3$,
		\begin{align}	
		A_r=O(N^{2/3-\epsilon'}),
		\end{align}
		and
		\begin{align}
		A_3=2\mathrm{e}^{-t}\xi^{(4)}q_t^{-2}\sum_{i,j}\mathbb{E}[F'(X)G_{ij}G_{ji}]+O(N^{2/3-\epsilon'}).
		\end{align}
	\end{lemma}
	
	Assuming Lemma \ref{lemma:A_r}, we find that there exists $\epsilon'>2\epsilon$ such that, for all $t\in [0,6\log N]$,
	\begin{align}
	\sum_{i,j,k}\mathbb{E}\Big[\dot{H}_{jk} F'(X)\frac{\partial G_{ii}}{\partial H_{jk}}\Big]=-\dot{L_t}\sum_{i,j}\mathbb{E}[G_{ij}G_{ji}F'(X)]+O(N^{2/3-\epsilon'}),
	\end{align}
	which implies that the right-hand side of \eqref{eq:F_derivative} is $O(N^{-\epsilon'/2})$. Integrating Equation \eqref{eq:F_derivative} from $t=0$ to $t=6 \log N$, we get
	\begin{align*}
	\left|\mathbb{E}F \Big( N\int_{E_1}^{E_2} \im m(x+L_t+\ii\eta_0) \textrm{d}x\Big)_{t=0} -\mathbb{E}F\Big(  N\int_{E_1}^{E_2} \im m(x+L_t+\ii\eta_0) \textrm{d}x\Big)_{t=6 \log N}\right| \leq N^{-\epsilon'/4}.
	\end{align*}
	By comparing largest eigenvalues of $H$ and~$\lambda_i^{G}$, we can get desired result. Let $\lambda_i(6 \log N)$ be the $i$-th largest eigenvalue of $H_{6 \log N}$ and $\lambda_i^{G}$ the $i$-th largest eigenvalue of $W^G$, then $|\lambda_i(6\log N)-\lambda_i^G|\prec N^{-3}$. Then we find that
	\begin{align}
	\left| \im m|_{t=6 \log N} - \im m^{G} \right| \prec N^{-5/3}.
	\end{align}
	This completes the proof of Proposition \ref{prop:Greencomparison}.
\end{proof}

\appendix

\section{Properties of $\rho$ and $\wt m_t$}\label{sec:rho wt m}
We first introduce some basic properties of $m_{sc}$.
\begin{lemma}[Basic properties of $m_{sc}$]\label{lem:basic properties of m} Define the distance to the spectral edge
	\begin{equation}
	\kappa \equiv \kappa(E) := \big| |E| -2 \big|.
	\end{equation}
	Then for $z \in D_\ell$ we have
	\begin{equation}\label{eq:m_z}
	|m_{sc}(z)| \sim 1, \qquad \qquad |1-m_{sc}^2| \sim \sqrt{\kappa+\eta}
	\end{equation}
	and 
	\begin{equation}\label{eq:imm_z}
	\emph{Im }  m_{sc}(z) \sim \begin{cases}
	\sqrt{\kappa+\eta} &\text{if $E \leq 2$}\\
	\frac{\eta}{\sqrt{\kappa+\eta}} &\text{if $E \geq 2$}.\\
	\end{cases}
	\end{equation}
	Moreover, for all $z$ with $\emph{Im }z >0$, 
	\begin{equation}
	|m_{sc}(z)|=|m_{sc}(z)+z|^{-1} \leq 1.
	\end{equation}
\end{lemma}
\begin{proof}
	The proof is an elementary calculation; see Lemma 4.2 in \cite{EYY11}.
\end{proof}
Now we introduce some properties of $\wt m_t$. Recall that $\wt m_t$ is a solution to the polynomial equation $P_{1,t,z}(\wt m_t)=0$ in \eqref{eq:P_1,t,z=0}. The proofs are similar to that of Lemma 4.1 of \cite{LS18}.
\begin{lemma}\label{lemma:rho_t}
	For any fixed $z = E + \ii \eta \in \mathcal{E}$ and any $t \geq 0$, the polynomial equation $P_{1,t,z}(w_t)=0$ has a unique solution $w_t \equiv w_t(z)$ satisfying $\im w_t >0$ and $|w_t|\leq 5$. Moreover, $w_t$ has the following properties : 
	\begin{enumerate}
		\item[(1)] There exists a probability measure $\wt \rho_t$ such that the analytic continuation of $w_t$ coincides with the Stieltjes transform of $\wt \rho_t$.
		\item[(2)] The probability measure $\wt \rho_t$ is supported on $[-L_t,L_t]$, for some $L_t \geq 2$, has a strictly positive density inside its support and vanishes as a square-root at the edges, i.e. letting
		\begin{equation}
		\kappa_t \equiv \kappa_t(E) := \min \{|E+L_t|,|E-L_t| \},
		\end{equation}
		we have 
		\begin{equation}
		\wt \rho_t(E) \sim \kappa_t^{1/2}(E), \qquad \qquad (E\in (-L_t, L_t))
		\end{equation}
		Moreover, $L_t = 2 + e^{-t}q_t^{-2}\xi^{(4)} + O(e^{-2t}q_t^{-4}) $ and $\dot{L}_t = -2e^{-t}q_t^{-2}\xi^{(4)} + O(e^{-2t}q_t^{-4}),$ where $\dot{L_t}$ denotes the derivative with respect to $t$ of $L_t$.
		\item[(3)] The solution $w_t$ satisfies that 
		\begin{align}
		&\im w_t(E+\ii \eta) \sim \sqrt{\kappa_t+\eta} \qquad \text{if  } E \in [-L_t,L_t], \nonumber \\
		&\im w_t(E+\ii \eta) \sim \frac{\eta}{\sqrt{\kappa_t+\eta}} \qquad \text{if  } E \notin [-L_t,L_t].
		\end{align}
	\end{enumerate}	
\end{lemma}
\begin{remark}
	For $z \in \mathcal{E}$, we can check the stability condition $|z+w_t|>\frac{1}{6}$  in the proof of the lemma since
	\begin{equation}
	|z+w_t|=\frac{1+e^{-2t}q^{-2}\xi^{(4)}|w_t|^4}{|w_t|}
	\end{equation}
	and $|w_t|<5$.
\end{remark}

\section{Proof of Lemmas in Section \ref{sec:weak local law}}\label{app:weak}
In this appendix, we provide the proofs of the series of lemmas stated in Section \ref{sec:weak local law}, including Lemmas \ref{lem:esteimate Lambda o}, \ref{lem:bound Zi Zij}, \ref{lem:estimate Upsilon}, \ref{lem:lem analysi of }, \ref{lem:analysis of self consis}, \ref{lem:inital estimate} and \ref{lem:bound Omega_k}.

We introduce the large deviation estimates concerning independent random variables whose moments decay slowly. Their proofs are given in Appendix of \cite{EKYY13}.

\begin{lemma}[Large deviation estimates]\label{lem:LDE} Let $(a_{i})$ and $(b_{i})$ be independent families of centered random variables satisfying 
	\begin{equation}
	\E |a_i|^p \leq \frac{C^p}{Nq^{p-2}} , \qquad \quad \E |b_i|^p \leq \frac{C^p}{Nq^{p-2}}
	\end{equation}
	for $2\leq p \leq (\log N)^{A_0\log\log N} $. Then for all $\epsilon > 0$, any $A_i \in \C$ and $B_{ij} \in \C$  we have with high probability
	\begin{align}
	\left| \sum_{i=1}^{N} A_i a_i \right| &\leq N^{\epsilon}\left[\frac{\max_{i}|A_i|}{q} + \left(\frac{1}{N}\sum_{i=1}^{N}|A_i|^2 \right)^{1/2} \right],\label{lem:LDE 1}\\
	\left| \sum_{i=1}^{N} \ol{a_i}B_{ii} a_i -\sum_{i=1}^{N} \sigma^2_i B_{ii} \right| &\leq N^{\epsilon}\frac{B_d}{q},\label{lem:LDE 2}\\
	\left| \sum_{1\leq i\neq j \leq N} \ol{a_i}B_{ij} a_j \right| &\leq N^{\epsilon}\left[\frac{B_o}{q} + \bigg(\frac{1}{N^2}\sum_{i\neq j}|B_{ij}|^2 \bigg)^{1/2} \right],\label{lem:LDE 3}\\
	\left| \sum_{i,j=1}^{N} a_i B_{ij} b_j \right| &\leq N^{\epsilon}\left[\frac{B_d}{q^2}+ \frac{B_o}{q} + \bigg(\frac{1}{N^2}\sum_{i\neq j}|B_{ij}|^2 \bigg)^{1/2} \right],\label{lem:LDE 4}
	\end{align}
	where $\sigma^2_{i}$ denotes the variance of $a_i$ and \[B_d := \max_{i}|B_{ii}|, \qquad B_o=\max_{i\neq j}|B_{ij}|. \]
\end{lemma}

\subsection{Proof of Lemma \ref{lem:esteimate Lambda o}, Lemma \ref{lem:bound Zi Zij} and Lemma \ref{lem:estimate Upsilon}}\label{app:3 lemmas}
\begin{proof}[Proof of Lemma \ref{lem:esteimate Lambda o}]
	Using $\eqref{eq:self cons 3}, \eqref{eq:Greenlemma2}, \eqref{eq:Greenlemma3} ,\eqref{eq:self cons 1}$ and Lemma $\ref{lem:bound Hij}$ we get on $\textbf{B}^c$, with high probability  
	\begin{align}\label{eq:estimate G_ij}
	|G_{ij}|=|G_{ii}G^{(i)}_{jj}(H_{ij}-Z_{ij})|&\leq C(|H_{ij}|+|Z_{ij}|)\nonumber\\
	&\leq C\frac{N^\epsilon}{q} + C\left| \sum_{k,l}^{(ij)}H_{ik}G^{ij}_{kl}H_{lj}\right|\nonumber\\
	&\leq C\frac{N^\epsilon}{q} + N^{\epsilon}\left[\frac{\Lambda_o}{q}+\frac{C}{q^2}+\bigg(\frac{1}{N^2}\sum_{k,l}^{(ij)}|G^{(ij)}_{kl}|^2 \bigg)^{1/2} \right]\nonumber\\
	&\leq C\frac{N^\epsilon}{q} + C \frac{N^{\epsilon}\Lambda_o}{q}+CN^{\epsilon}\bigg(\frac{1}{N^2}\sum_{k,l}^{(ij)}|G^{(ij)}_{kl}|^2 \bigg)^{1/2}
	\end{align}	
	where the last step follows using $\eqref{def:weak_H}$ and $\eqref{lem:LDE 4}$.
	From $\eqref{eq:self cons 1}$ and $\eqref{eq:Greenlemma2}$, on $\textbf{B}^c$ we have
	\begin{equation}
	G^{(ij)}_{kk}=G_{kk}+O(\Lambda^2_o).
	\end{equation} 
	Thus on $\textbf{B}^c$, 
	\begin{equation}\label{eq:estimate sum of G^2}
	\frac{1}{N^2}\sum_{k,l}^{(ij)}|G^{(ij)}_{kl}|^2 =\frac{1}{N^2\eta}\sum_{k}^{(ij)}\text{Im }G^{(ij)}_{kk} \leq \frac{\text{Im }m}{N\eta}+\frac{C\Lambda^2_o}{N\eta}
	\end{equation}
	by $\eqref{eq:Wald}$.
	Therefore, taking the maximum over $i\neq j$, we get 
	\begin{equation}
	\max_{i\neq j}|G_{ij}|=\Lambda_o\leq C\frac{N^\epsilon}{q}+o(1)\Lambda_o+CN^{\epsilon}\sqrt{\frac{\text{Im }m}{N\eta}},
	\end{equation}
	since $N\eta \geq N^{\ell}$. Hence we obtain on $\textbf{B}^c$ with high probability
	\begin{equation}
	\Lambda_o \leq C\left(\frac{N^\epsilon}{q}+N^{\epsilon}\Psi\right).
	\end{equation}\\
\end{proof}
\begin{proof}[Proof of Lemma \ref{lem:bound Zi Zij}]
	By definition of $Z_i$ we have 
	\begin{equation*}
	Z_i = \sum_{k}^{(i)}(|H_{kk}|^2-\E|H_{ik}|^2)G^{(i)}_{kk}+\sum_{k\neq l}^{(i)}H_{ik}G^{(i)}_{kl}H_{li}.
	\end{equation*}
	By the large deviation estimate $\eqref{lem:LDE 2}$ and  $\eqref{lem:LDE 3}$, we can get
	\begin{equation*}
	\begin{split}
	|Z_i| &\leq \frac{CN^\epsilon}{q} + CN^{\epsilon} \left[\frac{\Lambda_o}{q} +  \bigg(\frac{1}{N^2}\sum_{k\neq l}^{(i)}\left|G^{(i)}_{kl}\right|^2 \bigg)^{1/2}\right]\\
	&\leq \frac{CN^{\epsilon}}{q} + \frac{CN^{\epsilon} \Lambda_o}{q} + \frac{CN^{\epsilon}\Lambda_o}{\sqrt{N\eta}} + CN^{\epsilon}\sqrt{\frac{\text{Im }m}{N\eta}},
	\end{split}
	\end{equation*}
	with high probability where we used ($\ref{eq:estimate sum of G^2}$).
	Using Lemma $\ref{lem:esteimate Lambda o}$ we can finish the proof of ($\ref{eq:bound Zi}$).
	Similarly using ($\ref{lem:LDE 4}$) we have
	\begin{align*}
	Z^{(ij)}_{ij} &=\sum_{k,l}^{(ij)}H_{ik}G^{(ij)}_{kl}H_{li}\\
	&\leq N^{\epsilon}\left[\frac{\max|G^{(ij)}_{kk}|}{q^2}+ \frac{\max|G^{(ij)}_{kl}|}{q} + \left(\frac{1}{N^2}\sum_{i\neq j}|G^{(ij)}_{kl}|^2 \right)^{1/2} \right]\\
	&\leq N^{\epsilon} \left( \frac{C}{q^2}+ \frac{\Lambda_o}{q} +\sqrt{\frac{\text{Im }m + C\Lambda_o^2}{N\eta}} \right)\\
	&\leq N^{\epsilon} \left( \frac{C}{q^2}+ \frac{\Lambda_o}{q} + \Psi + \frac{ C\Lambda_o}{\sqrt{N\eta}} \right)\\
	&\leq CN^{\epsilon} \left( \frac{1}{q}+\Psi \right),
	\end{align*}
	with high probability where we used Lemma $\ref{lem:esteimate Lambda o}$ at the last step.\\
\end{proof}
\begin{proof}[Proof of Lemma \ref{lem:estimate Upsilon}]
	On $\textbf{B}^c$, we have $|A_i| \leq \frac{C}{N} + C\Lambda_o^2$ and $|H_{ij}| \prec \frac{1}{q}$.
	Thus, with high probability, we have
	\begin{align}
	|\Upsilon_i| &\leq |A_i|+|H_{ii}|+|Z_i| \nonumber \\
	&\leq  \frac{C}{N} + C\Lambda_o^2 +\frac{CN^\epsilon}{q} + CN^{\epsilon} \left( \frac{1}{q}+\Psi \right). 
	\end{align}
	Invoking ($\ref{eq:estimate Lambda o}$) we can get 
	\begin{align}
	|\Upsilon_i| &\leq C\left(\frac{1}{N} + \frac{N^\epsilon}{q} +\frac{N^{\epsilon}}{q} +N^{\epsilon}\Psi \right)\nonumber\\
	&\leq C \left(\frac{N^{\epsilon}}{q} + N^{\epsilon}\Psi\right) \ll 1
	\end{align}
	with high probability.
\end{proof}

\subsection{Proof of Lemma \ref{lem:lem analysi of } }\label{app:lem analysi of }
Recall that 
\[v_i = m_{sc}^2 (\sum_{j}\sigma_{ij}^2v_j -\Upsilon_i ) +m_{sc}^3 (\sum_{j}\sigma_{ij}^2v_j -\Upsilon_i )^2+O(\sum_{j}\sigma_{ij}^2v_j -\Upsilon_i )^3. \]
Using previous bound for $\Upsilon_i$, we can write
\begin{align}\label{eq:vi expand}
v_i &= m_{sc}^2\sum_{j}\sigma_{ij}^2v_j + O(\frac{N^{\epsilon}}{q}+N^{\epsilon} \Psi)+O(\frac{N^{\epsilon}}{q}+N^{\epsilon} \Psi +\Lambda_d^2)\nonumber\\
&=m_{sc}^2\sum_{j}\sigma_{ij}^2v_j+ O(\frac{N^{\epsilon}}{q}+N^{\epsilon}\Psi)+O(\Lambda_d^2),
\end{align}
on some set $\Xi$ such that $\Lambda_d \leq \frac{q}{(\log N)^{3/2}}$.
Taking the average over $i$, we have 
\begin{equation}
(1-m_{sc}^2)[v]=O(\frac{N^{\epsilon}}{q}+N^{\epsilon} \Psi)+O(\Lambda_d^2),
\end{equation}
and it follows from $\eqref{eq:vi expand}$ that 
\begin{equation}
v_i-[v]=m_{sc}^2\sum_{j}\sigma_{ij}^2(v_j-[v])+O(\frac{N^{\epsilon}}{q}+N^{\epsilon} \Psi)+O(\Lambda_d^2).
\end{equation}
Applying Lemma \ref{lem:bound norm} for $u_i = v_i - [v]$, we get
\begin{equation}\label{eq:bound max vi-[v] 2}
\max_{i}|v_i-[v]|\leq \frac{C\log N}{g}\left(\frac{N^{\epsilon}}{q}+N^{\epsilon} \Psi+\Lambda_d^2\right),
\end{equation}
therefore
\begin{equation*}
\Lambda_d \leq \Lambda +\frac{C\log N}{g}\left(\frac{N^{\epsilon}}{q}+N^{\epsilon} \Psi+\Lambda_d^2\right).
\end{equation*}
With \eqref{eq:Lambda d bound assume}, this implies 
\begin{equation}\label{eq:bound Lambda_d}
\Lambda_d \leq \Lambda + \frac{C\log N}{g}\left(\frac{N^{\epsilon}}{q}+N^{\epsilon} \Psi+\Lambda^2\right).
\end{equation}
Using ($\ref{eq:bound Lambda_d}$) to bound ($\ref{eq:bound max vi-[v] 2}$), we get the first inequality of  $\eqref{eq:bound max vi-[v]}$.
For the second inequality, since we have 
\[\frac{N^{\epsilon}}{q}+N^{\epsilon} \Psi\leq N^{-\epsilon}, \]
the second one follows. 
\subsection{Proof of Lemma \ref{lem:analysis of self consis}}\label{app:analysis of self}
From Lemma \ref{lem:lem analysi of }, we can observe that 
\begin{equation}
\max_{i}|v_i-[v]| \leq C \log N \left(\Lambda^2 + \frac{N^{\epsilon}}{q}+N^{\epsilon}\Psi\right),
\end{equation}
and 
\begin{equation}\label{eq:bound Lambda_d 2}
\Lambda_d \leq \Lambda + C \log N \left(\Lambda^2 + \frac{N^{\epsilon}}{q}+N^{\epsilon}\Psi\right),
\end{equation}
holds with high probability.

Define $\wt\Psi := \frac{N^{\epsilon}}{q}+N^{\epsilon} \Psi$. 
From $\Lambda \leq (\log N)^{-2}$ and Lemma $\ref{lem:basic properties of m}$, we have $\text{Im }  m_{sc}(z) \geq c\eta$ with some positive constant $c$. Hence 
\begin{equation}\label{eq:bound phi}
\frac{N^{\epsilon}}{\sqrt{N}} \leq \wt \Psi=\frac{N^{\epsilon}}{q}+N^{\epsilon} \Psi \leq N^{-\epsilon}.	
\end{equation}
By definition of $\Upsilon_i$, Lemma $\ref{lem:estimate basic quantities}$ and Lemma $\ref{lem:esteimate Lambda o}$, on  $\textbf{B}^c$ we have 
\begin{align}
\Upsilon_i & = A_i + h_{ii} - Z_i \nonumber\\ 
&=h_{ii}-Z_i+O\left(\frac{1}{N}+\Lambda_o^2 \right) \nonumber\\
&=h_{ii}-Z_i+O(\wt \Psi^2).
\end{align}
with high probability.
Recall that 
\[v_i = m_{sc}^2 (\sum_{j}\sigma_{ij}^2v_j -\Upsilon_i ) +m_{sc}^3 (\sum_{j}\sigma_{ij}^2v_j -\Upsilon_i )^2+O(\sum_{j}\sigma_{ij}^2v_j -\Upsilon_i )^3. \]
From previous observation, we can bound the last term of this equation by $O(\wt \Psi^3+\Lambda_d^3)$ which is bounded by $O(\wt \Psi^2+\Lambda^3)$, using $\eqref{eq:bound Lambda_d 2}$ and $\eqref{eq:bound phi}$. Hence in $\textbf{B}^c$ 
\begin{align}
v_i &= m_{sc}^2 (\sum_{j}\sigma_{ij}^2v_j -\Upsilon_i ) +m_{sc}^3 (\sum_{j}\sigma_{ij}^2v_j -\Upsilon_i )^2+O(\sum_{j}\sigma_{ij}^2v_j -\Upsilon_i )^3\nonumber \\
&=m_{sc}^2 \left(\sum_{j}\sigma_{ij}^2v_j+Z_i-h_{ii}+O(\wt \Psi^2)\right) + m_{sc}^3 \left(\sum_{j}\sigma_{ij}^2v_j +O(\wt \Psi)\right)^2+O(\wt \Psi^2+\Lambda^3),
\end{align}
holds with high probability.
Summing up $i$ and dividing by $N$, we have
\begin{equation}\label{eq:expand [v]}
[v]=m_{sc}^2[v] +m_{sc}^2[Z]+O(\wt \Psi^2+\Lambda^3)+\frac{1}{N}m_{sc}^3\sum_{i}\left(\sum_{j}\sigma_{ij}^2v_j+O(\wt \Psi) \right)^2,
\end{equation}
where we used $\frac{1}{N}\big| \sum_{i=1}^{N} h_{ii} \big| \leq \frac{2N^{\epsilon}}{N}\leq (N^{\epsilon}\Psi)^2 \leq \wt \Psi^2$ with high probability by ($\ref{lem:LDE 1}$).
Writing $v_j=(v_j-[v])+[v]$, with AM-GM inequality, the last term of ($\ref{eq:expand [v]}$) can be estimated by 
\begin{equation}
\begin{split}
\frac{1}{N}&m_{sc}^3\sum_{i}\left(\sum_{j}\sigma_{ij}^2v_j+O(\wt \Psi)  \right)^2 \\ 
&=  m_{sc}^3 \bigg( [v]^2 +O\big(( (\log N)(\Lambda^2 + \wt \Psi))^2\big) +O(\wt \Psi^2) 
+ [v](\log N)(\Lambda^2+\wt \Psi) + [v]\wt \Psi+(\log N)(\Lambda^2+\wt \Psi)\wt \Psi   \bigg)\\
&= m_{sc}^3[v]^2+O\big((\log N)\Lambda(\Lambda^2 +\wt \Psi) \big)+O\big( (\log N)^2\wt \Psi^2\big),
\end{split}
\end{equation}
where we used ($\ref{eq:bound max vi-[v]}$) and $|[v]|=\Lambda$.
Thus 
\begin{align}
(1-m_{sc}^{2})[v] & =m_{sc}^3[v]^2+m_{sc}^2[Z]+O\big((\log N)\Lambda(\Lambda^2 +\wt \Psi) \big)+O\big( (\log N)^2\wt \Psi^2\big) +O(\Lambda^3 + \wt \Psi^2) \nonumber \\
&=m_{sc}^3[v]^2+m_{sc}^2[Z]+O\left(\frac{\Lambda^2}{\log N} \right) +O\left((\log N)^3\wt \Psi^2 \right)\nonumber \\
&=m_{sc}^3[v]^2+m_{sc}^2[Z]+O\left(\frac{\Lambda^2}{\log N} \right) +O\left((\log N)^3\left(\frac{N^{\epsilon}}{q}+N^{\epsilon} \Psi \right)^2 \right),
\end{align}
holds with high probability, using  $\Lambda \leq (\log N)^{-2}$ in $\textbf{B}^c$.

\subsection{Proof of Lemma \ref{lem:inital estimate} and Lemma \ref{lem:dichtomy_weak}}\label{subsec:proof dichtomy}
\begin{proof}[Proof of Lemma \ref{lem:inital estimate}]
	Fix $z \in D_\ell$. For arbitrary $T \subset \{1,\dots,N\}$ we have the trivial bound
	\begin{equation}\label{eq:trivial bound}
	|G_{ij}^{(T)}| \leq \frac{1}{\eta}, \qquad\qquad |m^{(T)}| \leq \frac{1}{\eta}, \qquad\qquad |m_{sc}|\leq \frac{1}{\eta}.
	\end{equation}
	First we estimate $\Lambda_o$. Recall $\eqref{eq:estimate G_ij}$. We have, with high probability, 
	\begin{equation}
	|G_{ij}|\leq C\frac{N^\epsilon}{q}+o(1)\Lambda_o+CN^{\epsilon}\sqrt{\frac{\text{Im }m^{(ij)}}{N\eta}} \leq C\frac{N^\epsilon}{q}+o(1)\Lambda_o+\frac{CN^{\epsilon}}{\sqrt{N}}.
	\end{equation}
	Taking maximum over $i\neq j$ then we get
	\begin{equation}\label{eq:estimate Lambda_o2}
	\Lambda_o \leq C\frac{N^\epsilon}{q}+\frac{CN^{\epsilon}}{\sqrt{N}}
	\end{equation}
	with high probability.
	Now we estimate $\Lambda_d$. From the definition of $\Upsilon_i$, 
	\begin{equation}
	|\Upsilon_i| \leq \frac{CN^\epsilon}{q} +|Z_i|+|A_i|.
	\end{equation}
	Using ($\ref{eq:self cons 3}$), we can estimate $A_i$ with high probability by
	\begin{align}
	|A_i| &\leq \frac{C}{N}|G_{ii}|+\left| \sum_{j\neq i}\sigma_{ij}^2 \frac{G_{ij}}{G_{ii}}G_{ji}\right| \nonumber \\
	&\leq \frac{C}{N} + \sum_{j\neq i}\sigma_{ij}^2 |G_{ji}||G_{jj}^{(i)}||K_{ij}^{(ij)}| \nonumber \\
	&\leq \frac{C}{N} + \sum_{j\neq i}\sigma_{ij}^2 \Lambda_o \frac{1}{\eta} \left(|H_{ij}|+|Z_{ij}^{(ij)}| \right) \nonumber \\
	&\leq O\left(\frac{1}{N}+\Lambda_o \frac{N^\epsilon}{q} +\Lambda_o N^{\epsilon} \left(\frac{1}{q^2} +\frac{\Lambda_o}{q} +\frac{1}{\sqrt{N}}\right) \right) \leq C \frac{N^\epsilon}{q},
	\end{align}
	where the last step we used that with high probability 
	\begin{equation}
	|Z_{ij}^{(ij)}|\leq N^{\epsilon}\left(\frac{C}{q^2}+\frac{\Lambda_o}{q} + \frac{C}{\sqrt{N}} \right),
	\end{equation}
	as follows from the large deviation estimate ($\ref{lem:LDE 4}$). Similar as ($\ref{eq:bound Zi}$), using ($\ref{lem:LDE 2}$) and ($\ref{lem:LDE 3}$), we can get
	\begin{equation}
	|Z_i|\leq C\left(\frac{N^\epsilon}{q} + \frac{N^{\epsilon}}{\sqrt{N}} \right)
	\end{equation}
	with high probability. Thus we find with high probability
	\begin{equation}
	|\Upsilon_i| \leq C\left( \frac{N^\epsilon}{q} + \frac{N^{\epsilon}}{\sqrt{N}}\right).
	\end{equation}
	Let $\Upsilon = \max_{i} |\Upsilon_i|$. Then for all $n$, we can write the self consistent equation ($\ref{eq:vi self consistent}$) as 
	\begin{equation}
	v_n = \frac{\sum_{i}\sigma_{ni}^2v_i + O(\Upsilon)}{(z+m_{sc}+\sum_{i}\sigma_{ni}^2v_i+O(\Upsilon))(z+m_{sc})}.
	\end{equation}
	We have $\Lambda_d \leq 1$ since 
	\[ \Lambda_d = \max |v_i| = \max |G_{ii}-m_{sc}|\leq \max(|G_{ii}|+|m_{sc}|) \leq \frac{2}{\eta} \leq 1. \]
	Therefore the denominator of $v_n$ is larger than
	\[ \big((2-1)+O(\Upsilon)\big)2 \geq 3/2 \]
	for sufficiently large $N$ since $|z+m_{sc}|=|m_{sc}|^{-1}\geq 2.$
	Hence 
	\begin{equation}
	v_n \leq \frac{\sum_{i}\sigma_{ni}^2v_i + O(\Upsilon)}{2/3} \leq \frac{\Lambda_d + O(\Upsilon)}{3/2}.
	\end{equation}
	Taking maximum over $n$, we get
	\begin{equation}\label{eq:estimate Lambda_d2}
	\Lambda_d \leq C\left(\frac{N^\epsilon}{q} + \frac{N^{\epsilon}}{\sqrt{N}} \right)
	\end{equation}
	with high probability. 
	Now the estimate ($\ref{eq:estimate Lambda_d + Lambda_o}$) follows from  ($\ref{eq:estimate Lambda_o2}$) and  ($\ref{eq:estimate Lambda_d2}$).
\end{proof}

\begin{proof}[Proof of Lemma \ref{lem:dichtomy_weak}]
	Fix $z= E+i\eta \in D_\ell$. From $\eqref{eq:weak self consistent equation}$, with high probability, we have 
	\begin{align}
	\frac{(1-m_{sc}^{2})}{m_{sc}^3}[v]&=[v]^2+O\left(\frac{\Lambda^2}{\log N}\right)+O\bigg(\frac{N^{\epsilon}}{q}+N^{\epsilon} \Psi\bigg) \nonumber\\
	&=[v]^2+O\left(\frac{\Lambda^2}{\log N}\right)+O\bigg(\frac{N^{\epsilon}}{q}+ \sqrt{\beta^3\alpha + \beta^3\Lambda}\bigg) \nonumber \\
	&=[v]^2+O\left(\frac{\Lambda^2}{\log N}\right)+C^*(\beta\Lambda+\alpha\beta+\beta^2),
	\end{align}
	for some constant $C^*\geq 1$. Set $U_0=9(C^*+1)\geq 18$.
	In other words, on $\textbf{B}^c$ we have with high probability 
	\begin{equation}\label{eq:dichotomy_weak}
	\left|\frac{(1-m_{sc}^{2})}{m_{sc}^3}[v]-[v]^2\right| \leq O\left(\frac{\Lambda^2}{\log N}\right)+C^*(\beta\Lambda+\alpha\beta+\beta^2).
	\end{equation}
	\textit{Case 1 : $\eta \geq \tilde{\eta}$}. From the definition of $\tilde{\eta}$ and $C^*$, 
	\begin{equation}\label{eq:compare beta and alpha}
	\beta \leq \frac{\alpha}{2U^2} \leq \frac{\alpha}{2c^*}\leq \alpha.
	\end{equation}
	Therefore from ($\ref{eq:dichotomy_weak}$) 
	\begin{align}
	\alpha\Lambda \leq 2\Lambda^2 + C^*(\beta\Lambda+\alpha\beta+\beta^2) 
	\leq 2\Lambda^2 + \frac{\alpha\Lambda}{2}+2C^*\alpha\beta.
	\end{align}
	This implies 
	\begin{equation}
	\alpha \Lambda \leq 4\Lambda^2 + 4C^*\alpha\beta.
	\end{equation}
	Hence we have either 
	\begin{equation}
	\frac{\alpha\Lambda}{2} \leq 4\Lambda^2, \qquad \text{or}\qquad \frac{\alpha\Lambda}{2} \leq 4c^*\alpha\beta,
	\end{equation}
	which implies \eqref{eq:dicho1}.\\
	\textit{Case 2 : $\eta < \tilde{\eta}$}. From the definition of $\tilde{\eta}$, 
	\[ \alpha \leq 2U^2K^2\beta. \]
	Therefore 
	\[\Lambda^2 \leq 2\alpha\Lambda+ 2+C^*(\beta\Lambda+\alpha\beta+\beta^2)\leq C'\beta \Lambda + C' \beta^2, \]
	for some constant $C'=C'(U)$. This quadratic inequality implies \eqref{eq:dicho2},
	\[\Lambda \leq C_1(U)\beta , \]
	for some $U$-dependent constant $C_1$. This finishes the proof.
\end{proof}

\subsection{Proof of Lemma \ref{lem:bound Omega_k}}\label{app:proof of Lemma lem bound Omega_k}
We will use induction on $k$, the case $k=1$ has been checked. Assume that ($\ref{eq:Omega_k high probability}$) holds for $k$. Now we need to estimate
\begin{equation}
\P(\Omega_{k+1}^c) \leq A+B+\P(\Omega_k^c),
\end{equation}
where we define
\[ A:= \P(\Omega_k \cap \textbf{B}(z_{k+1})) = \P\bigg[ \Omega_k \cap \big\{ \Lambda_d(z_{k+1}) + \Lambda_o(z_{k+1}) > (\log N)^{-2} \big\} \bigg], \]
\[ B:= \P \big(\Omega_k \cap \textbf{B}^c(z_{k+1}) \cap \Omega_{k+1}^c  \big) = \P\bigg[ \Omega_k \cap \textbf{B}^c(z_{k+1})  \cap \big\{ \Lambda(z_{k+1}) > C^{(k+1)}(U)\beta(z_{k+1}) \big\} \bigg].     \]
First we estimate $A$. For all $i,j$
\begin{equation}
|G_{ij}(z_{k+1}) -G_{ij}(z_{k})| \leq |z_{k+1}-z_k|\sup_{z\in D_\ell}\left| \frac{\partial G_{ij}(z)}{\partial z}\right| \leq N^{-8}\sup_{z\in D_\ell}\frac{1}{(\text{Im }z)^2} \leq N^{-6}.
\end{equation}
Hence on $\Omega_k$ with high probability
\begin{align}
\Lambda_d(z_{k+1}) + \Lambda_o(z_{k+1}) &\leq \Lambda_d(z_{k}) + \Lambda_o(z_{k})+2N^{-6}\nonumber \\
&\leq \Lambda(z_k) + C \log N \left(\Lambda(z_k)^2 + \frac{N^{\epsilon}}{q}+N^{\epsilon}\Psi(z_k)\right)+2N^{-6} \nonumber \\
&\leq  C^{(k)}(U)\beta(z_k)+ C \log N \left(\Lambda(z_k)^2 + \frac{N^{\epsilon}}{q}+N^{\epsilon}\Psi(z_k)\right) \nonumber \\
&\leq (\log N)^{-2},
\end{align}
where the second step we used ($\ref{eq:estimate Lambda o}$) and ($\ref{eq:bound Lambda_d 2}$). Hence we can conclude that $A$ holds with high probability.
Now we consider $B$. First, assume that $\eta_{k+1} \geq \tilde{\eta}$. Then 
\begin{equation}
|\Lambda(z_{k+1})-\Lambda(z_{k})| \leq N^{-6}\leq \frac{1}{2}U\beta(z_{k+1}).
\end{equation}
Thus on $\Omega_k$ we get 
\begin{equation} \label{eq:Lambda<Ubeta }
\Lambda(z_{k+1})\leq \Lambda(z_k)+N^{-6}\leq U\beta(z_{k})+N^{-6}\leq \frac{3}{2}U\beta(z_{k+1}).
\end{equation}
If $\eta_{k+1} \geq \tilde{\eta}$, then by ($\ref{eq:compare beta and alpha}$) and $\eqref{eq:Lambda<Ubeta }$,
\begin{equation}
\Lambda(z_{k+1}) < \frac{\alpha(z_{k+1})}{U},
\end{equation}
and from dichotomy lemma, we find that
\begin{equation}
\Lambda(z_{k+1}) \leq U\beta(z_{k+1})
\end{equation}
holds with high probability on $\Omega_k \cap \textbf{B}^c(z_{k+1})$. If $\eta_{k+1} < \tilde{\eta}$, then ($\ref{eq:Lambda<Ubeta }$) immediately implies
\begin{equation}
\Lambda(z_{k+1}) \leq C^{(k+1)}(U)\beta(z_{k+1}).
\end{equation} 
Hence when $\eta_{k+1} \geq \tilde{\eta}$, $B$ holds with high probability. For $\eta_{k} < \tilde{\eta}$ case, this implies $\eta_{k+1} < \tilde{\eta}$, so by dichotomy lemma 
\begin{equation}
\Lambda(z_{k+1}) \leq C^{(k+1)}(U)\beta(z_{k+1}).
\end{equation}
We have therefore proved that 	
\begin{equation}
\P(\Omega_{k+1}^c) \leq 2N^{-D}+\P(\Omega_k^c),
\end{equation}
and this implies \eqref{eq:Omega_k high probability}.

\section{Proof of recursive moment estimate}\label{sec:recursive}
In this section, we prove Lemma \ref{lemma:recursive}. Fix $t\in [0,6\log N]$. Recall definitions in \eqref{def:G_t}, \eqref{def:q_t}. For brevity, we drop $t$ from the notation in $H_t$, its matrix elements and its Green function through this section.  We also omit the $t$-dependence of $\kappa^{(k)}_{t,ij}$ and $\xi_t$.
We define the $z$-dependent control parameter $\Phi_\epsilon \equiv \Phi_\epsilon(z) $ by 
\begin{align}
&\Phi_\epsilon= N^{\epsilon}\mathbb{E}\Big[\Big(\frac{1}{q_t^4} +\frac{\im m_t}{N\eta} )|P(m_t)|^{2D-1}\Big] + N^{-\epsilon/8}q_t^{-1}\mathbb{E}\Big[|m_t -\wt{m}_t|^2|P(m_t)|^{2D-1}\Big] +N^{\epsilon}q_t^{-8D}  \notag \\
&\qquad +N^{\epsilon} q_t^{-1}\sum_{s=2}^{2D}\sum_{s'=0}^{s-2}\mathbb{E}\Big[\Big(\frac{\im m_t}{N\eta}\Big)^{2s-s'-2}|P'(m_t)|^{s'}|P(m_t)|^{2D-s}\Big] \notag \\
&\qquad +N^{\epsilon}\sum_{s=2}^{2D}\mathbb{E}\Big[\Big(\frac{1}{N\eta}+\frac{1}{q_t}\Big(\frac{\im m_t}{N\eta} \Big)^{1/2}+\frac{1}{q_t^2}\Big) \Big(\frac{\im m_t}{N\eta}  \Big)^{s-1}|P'(m_t)|^{s-1}|P(m_t)|^{2D-s}\Big].
\end{align}
Then Lemma \ref{lemma:recursive} states that, for any small $\epsilon >0$,
\[
\E |P|^{2D} \leq \Phi_\epsilon(z), \qquad\qquad (z\in\mathcal{D_\ell}),	
\]
for $N$ sufficiently large, where the domain $\mathcal{D_\ell}$ is defined in \eqref{def:domain D}. We say that a \emph{random variable $Z$ is negligible} if $|\E [Z]|\leq C\Phi_\epsilon$ for some $N$-dependent constant $C$.

To prove the recursive moment estimate, we return to Lemma \ref{lemma:Stein}  which reads
\begin{align}
\E \qB{\pb{&(z+m+\zeta m)^2+2\zeta^2m^2-\zeta}(1+zm)P^{D-1}\ol{P^D} } \notag \\
=& \frac{1}{N}\sum_{i\neq k}\sum_{r=1}^{l}\frac{\kappa_{ik}^{(r+1)}}{r!}\E \qB{ (\partial_{ik})^r \pB{G_{ik}\pb{(z+m+\zeta m)^2+2\zeta^2m^2-\zeta}P^{D-1}\ol{P^D}} } \notag \\
& + \E \qB{ \Omega_l \pB{\pb{(z+m+\zeta m)^2+2\zeta^2m^2-\zeta}(1+zm)P^{D-1}\ol{P^D}}  },
\end{align}
where $\partial_{ij} = \partial/(\partial H_{ik}), \kappa_{ik}^{(\b{\cdot})} \equiv \kappa_{t,ik}^{(\b{\cdot})}$ and $\zeta\equiv \zeta_t$.

Abbreviate 
\begin{align}
I\equiv I(z,m,D) &= \pb{(z+m+\zeta m)^2+2\zeta^2m^2-\zeta}(1+zm)P^{D-1}\ol{P^D} \notag \\
&= Q (1+zm)P^{D-1}\ol{P^D}
\end{align}
where $Q:=(z+m+\zeta m)^2+2\zeta^2m^2-\zeta$.
Then we can rewrite the cumulant expansion as 
\begin{align}\label{eq:cumulantexp1}
&\E I = \sum_{r=1}^{l} \sum_{s=0}^{r} w_{I_{r,s}} \E I_{r,s} + \E \Omega_l(I),
\end{align}
where we set
\begin{align}
I_{r,s}&=\frac{1}{N} \sum_{i\neq k} \kappa_{ik}^{(r+1)}{ \pB{\partial_{ik}^{r-s} G_{ik}} \pB{\partial_{ik}^{s}\pB{\pb{(z+m+\zeta m)^2+2\zeta^2m^2-\zeta}P^{D-1}\ol{P^D}}}} \notag \\
&=\frac{1}{N} \sum_{i\neq k} \kappa_{ik}^{(r+1)}{ \pB{\partial_{ik}^{r-s} G_{ik}} \pB{\partial_{ik}^{s}\pB{QP^{D-1}\ol{P^D}}}},\\
w_{I_{r,s}}& = \frac{1}{(r-s)!s!}.
\end{align}

\subsection{Truncation of the cumulant expanstion}
In this subsection, we will bound the error term $\E\Omega_{l}(I)$ defined in  \eqref{eq:cumulantexp1} for large $l$.
Let $E^{[\mf{ik}]}$ denote the $N \times N$ matrix determined by
\begin{equation}
(E^{[\mathfrak{ik}]})_\mathfrak{ab}=
\begin{cases}
\delta_{\mathfrak{ia}}\delta_\mathfrak{kb}+\delta_\mathfrak{ib}\delta_\mathfrak{ka}   & \qquad\text{ if } \mathfrak{i} \neq \mathfrak{k}, \\
\delta_\mathfrak{ia}\delta_\mathfrak{ib} & \qquad\text{ if } \mathfrak{i} =\mathfrak{k},
\end{cases}
\qquad	(\mathfrak{i,k,a,b} \in  \llbracket  1,N \rrbracket) .
\end{equation}
For each pair of indices $(\mf{i,k})$, we define the matrix $H^{(\mf{ik})}$ from $H$ through the decomposition
\begin{equation}
H=H^{(\mf{ik})}+H_\mf{ik}E^{[\mf{ik}]}.
\end{equation}
Here and after we use Fraktur fonts to denote the indices that can be in  $\llbracket  1,N \rrbracket$. With this notation we have the following estimate.
\begin{lemma}
	Suppose that $H$ satisfies Definition $\ref{assumption}$ with $\phi >0$. Let $\mf{i,k} \in \llbracket  1,N \rrbracket $, $D\in \N$ and $z \in \mathcal{D}_\ell$. Define the function $F_\mf{ki}$ by 
	\begin{align}
	F_\mf{ki}(H):=&G_\mf{ki}((z+m+\zeta m)^2+2\zeta^2m^2-\zeta)P^{D-1}\ol{P^D} \notag \\
	=&G_\mf{ki}QP^{D-1}\ol{P^D},
	\end{align}
	where $G \equiv G^H(z)$ and $P\equiv P(m(z))$. Choose an arbitrary $l \in \N$. Then for any $\epsilon >0$, 
	\begin{equation}\label{eq:truncation}
	\E \qbb{\sup_{x\in\R,|x|\leq q_t^{-1/2}} |\partial_\mf{ik}^l F_\mf{ki}(H^{(\mf{ik})}+xE^{[\mf{ik}]})|} \leq N^\epsilon,
	\end{equation}
	uniformly $z\in \mathcal{D}_\ell$, for sufficiently large $N$. 
\end{lemma}

\begin{proof}
	Fix the pair of indices $\mf{(a,b)}$ and $\mf{(i,k)}$. By the definition of the Green function and the decomposition of the $H$ we have
	\begin{equation}
	G_\mf{ab}^{H^{\mf{(ik)}}} = 	G_\mf{ab}^{H} + H_\mf{ik} (G^{H^\mf{(ik)}}E^{\mf{[ik]}}G^H)_\mf{ab} = G^H_\mf{ab} +H_\mf{ik}G_\mf{ai}^{H^{\mf{(ik)}}}G^H_\mf{kb}+H_\mf{ik}G_\mf{ak}^{H^{\mf{(ik)}}}G^H_\mf{ib}.
	\end{equation}
	Letting $\Lambda_o^{H^\mf{(ik)}} := \max_{\mf{a,b}}|G_\mf{ab}^{H^\mf{(ik)}}|$ and
	$\Lambda_o^{H} := \max_{\mf{a,b}}|G_\mf{ab}^{H}|$, we have
	\[
	\Lambda_o^{H^\mf{(ik)}} \prec \Lambda_o^{H} + \frac{1}{q_t} \Lambda_o^{H} \Lambda_o^{H^\mf{(ik)}}.
	\]
	From Theorem \ref{thm:weak local law} and \eqref{eq:moments}, we get $ \Lambda_o^{H^\mf{(ik)}} \prec \Lambda_o^{H}\prec 1$, uniformly in $z \in \mathcal{D}_\ell$. Similarly for $x \in \R$, we get 
	\[
	G_\mf{ab}^{H^{\mf{(ik)}}+xE^\mf{[ik]}} = G_\mf{ab}^{H^{\mf{(ik)}}} - x (G^{H^{\mf{(ik)}}}E^\mf{[ik]}G^{H^{\mf{(ik)}}+xE^\mf{[ik]}})_\mf{ab},
	\]
	and therefore we have 
	\begin{equation}
	\sup_{|x|\leq q_t^{-1/2}}\max_{\mf{a,b}}|G_\mf{ab}^{H^{\mf{(ik)}}+xE^\mf{[ik]}}| \prec \Lambda_o^{H^\mf{(ik)}} \prec 1,
	\end{equation}
	uniformly in $z \in \mathcal{D}_\ell.$
	Recall that $P$ is a polynomial of degree 6 in $m$. Hence $F_\mf{ki}$ is a multivariate polynomial of degree $6(2D-1) +3 $ in the Green function entries and the normalized trace $m$ whose number of member terms is bounded by $3\cdot 7^{2D-1}$. Therefore $\partial_{\mf{ik}}F_\mf{ki}$ is a multivariate polynomial of degree $6(2D-1)+3 + l$ whose number of member terms is roughly bounded by $3\cdot 7^{2D-1}\cdot(6(2D-1)+3+2l)^l$. Hence we get 
	\begin{equation}
	\E \qbb{\sup_{|x|\leq q_t^{-1/2}} |\partial_\mf{ik}^l F_\mf{ki}(H^{(\mf{ik})}+xE^{[\mf{ik}]})|} \leq 7^{2D}(12D+l) N^{(12D+l)\epsilon'},
	\end{equation}
	for any small $\epsilon' >0$ and sufficiently large $N$. Choosing $\epsilon' = \epsilon/2(12D+l)$, we get \eqref{eq:truncation}.

\end{proof}

\begin{corollary}\label{coro:bound error} Let $\E\qb{\Omega_l(I)}$ be as in \eqref{eq:cumulantexp1}. Then for any small $\epsilon <0$, we have 
	\begin{equation}\label{eq:stein error}
	\E \qb{ \Omega_l(I) } \leq N^\epsilon \left(\frac{1}{q_t} \right)^l,
	\end{equation}
	uniformly in $z \in \mathcal{D}_\ell$ for sufficiently large $N$. In particular, if $l \geq 8D$, then $\E \qb{ \Omega_l(I) }$ is negligible.
\end{corollary}
\begin{proof}
	Fix the pair of indices $\mf{(k,i)}, \mf{k}\neq \mf{i}$ and denote $\E_\mf{ik}$ the partial expectation with respect to $H_\mf{ik}$. Then from Lemma \ref{lemma:Stein} with $Q=q_t^{-1/2}$, we obtain
	\begin{align}
	|\E_\mf{ik}\Omega_l(H_\mf{ik}F_\mf{ki})| \leq C_l &\E_\mf{ik}[|H_\mf{ik}|^{l+2}]\sup_{|x|\leq q_t^{-1/2}}|\partial_\mf{ik}^{l+1} F_\mf{ki}(H^{(\mf{ik})} + x E^\mf{[ik]})| \notag \\
	& + C_l\E_\mf{ik}[|H_\mf{ik}|^{l+2} \mathds{1}(|H_\mf{ik}| > q_t^{-1/2}) ]\sup_{x\in\R} |\partial_\mf{ik}^{l+1} F_\mf{ki}(H^{(\mf{ik})} + x E^\mf{[ik]})|,
	\end{align}
	with $C_l \leq (Cl)^l/l!$, for some constant $C$. With the moment assumption \eqref{eq:moments} and previous Lemma, for any $\epsilon >0 $, we have
	\begin{equation}
	C_l\E_\mf{ik}[|H_\mf{ik}|^{l+2}]\sup_{|x|\leq q_t^{-1/2}}|\partial_\mf{ik}^{l+1} F_\mf{ki}(H^{(\mf{ik})} + x E^\mf{[ik]})| \leq C_l \frac{(C(l+2))^{c(l+2)}}{Nq_t^l} N^\epsilon \leq \frac{N^{2\epsilon}}{Nq^l_t},
	\end{equation}
	for sufficiently large $N$. To estimate the second line, with trivial bound $|G(z)|\leq \eta^{-1}$, we obtain
	\begin{equation}
	\sup_{x\in\R} |\partial_\mf{ik}^{l+1} F_\mf{ki}(H^{(\mf{ik})} + x E^\mf{[ik]})|\leq 7^{2D}(12D+l)\pa{\frac{C}{\eta}}^{12D+l}. \qquad \qquad (z \in \C^+)
	\end{equation}
	On the other hand, from H\"older's inequality we have for any $D' \in \N$,
	\begin{equation}
	C_l\E_\mf{ik}[|H_\mf{ik}|^{l+2} \mathds{1}(|H_\mf{ik}| > q_t^{-1/2}) ]\sup_{x\in\R} |\partial_\mf{ik}^{l+1} F_\mf{ki}(H^{(\mf{ik})} + x E^\mf{[ik]})| \leq \pa{\frac{C}{q}}^{D'},
	\end{equation}
	uniformly on $z \in \C^+$, sufficiently large $N$. Then summing over $\mf{i}$, $\mf{k}$ and choosing $D' \geq l$ sufficiently large, for any $\epsilon >0$, we get 
	\begin{equation}
	\absB{\E \qB{\Omega_l\pa{(1+zm)}QP^{D-1}\ol{P^D} }} = 	
	\absB{\E \qB{\Omega_l\pB{ \frac{1}{N}\sum_{\mf{i}\neq\mf{k}} H_\mf{ik}F_\mf{ki} }}} \leq \frac{N^\epsilon}{q_t^l},
	\end{equation} 
	uniformly on $z \in \mathcal{D}_\ell$, for $N$ sufficiently large. This concludes \eqref{eq:stein error}.
	
\end{proof}
\subsection{Truncated cumulant expansion}

We can prove Lemma \ref{lemma:recursive} directly from the following result.
\begin{lemma}\label{lemma:truncated}
	Fix $D\geq 2$ and $\ell \geq 8D$. Then we have, for any (small) $\epsilon >0$, we have 
	\begin{equation}\label{eq:estimate I_1}
	\begin{split}
	&w_{I_{1,0}}\E \qb{I_{1,0}}=-\E\qbb{ \pB{(1-\zeta)m^2Q+3\zeta^3m^4+\zeta^2 \frac{\xi^{(4)}}{q_t^{2}} m^6 -\zeta m(z+m+3\zeta m)(\frac{\xi^{(4)}}{q_t^{2}}m^4+1) }P^{D-1}\ol{P^D}}+O(\Phi_\epsilon),\\
	&w_{I_{2,0}}\E \qb{I_{2,0}}= O(\Phi_\epsilon),\\
	&w_{I_{3,0}}\E \qb{I_{3,0}}= -\E \qB{q_t^{-2}\xi^{(4)}Qm^4 P^{D-1}\ol{P^D}}+O(\Phi_\epsilon),\\
	&w_{I_{r,0}}\E \qb{I_{r,0}}=O(\Phi_\epsilon), \qquad (4\leq r\leq \ell), 
	\end{split}
	\end{equation}
	uniformly in $z \in \mathcal{D}_\ell$, for $N$ sufficiently large. Moreover, we have, for any small $\epsilon >0$,
	\begin{equation}\label{eq:estimate I_2}
	w_{I_{r,s}}\absa{\E \qb{I_{r,s}}} \leq \Phi_\epsilon, \qquad (1\leq s\leq r\leq \ell), 
	\end{equation}
	uniformly in $z \in \mathcal{D}_\ell$, for $N$ sufficiently large.
\end{lemma}
\begin{proof}[Proof of Lemma \ref{lemma:recursive} ]
	Recall that $Q:=(z+m+\zeta m)^2+2\zeta^2m^2-\zeta$. Using this, we write $|P|^{2D}$ as 
	\begin{align}\label{eq:expand P^2D}
	\E |P|^{2D} &= \E \qbb{(1 + zm + m^2 +q_t^{-2}\xi^{(4)}m^4)\pb{(z+m+\zeta m)^2-\zeta(1+mz+m^2)}P^{D-1}\ol{P^D} } \notag \\
	&= \E \qbb{(1 + zm + m^2 +q_t^{-2}\xi^{(4)}m^4)\pb{Q-2\zeta^2m^2-\zeta mz-\zeta m^2}P^{D-1}\ol{P^D} } \notag \\
	&= \E \qbb{(1+zm)QP^{D-1}\ol{P^D}} +\E \qbb{q_t^{-2}\xi^{(4)}m^4QP^{D-1}\ol{P^D}} \notag \\
	&\qquad -\E\qbb{ \pa{-m^2Q+\zeta m(1+zm+m^2+q_t^{-2}\xi^{(4)}m^4)(m+z+2\zeta m)}P^{D-1}\ol{P^D}}.
	\end{align}
	By simple calculation, it can be shown that
	\begin{equation*}
	\begin{split}
	(1-\zeta )m^2Q+3\zeta ^3m^4+\zeta ^2 \frac{\xi^{(4)}}{q_t^{2}} m^6 &-\zeta m(z+m+3\zeta m)(\frac{\xi^{(4)}}{q_t^{2}}m^4+1) \\ 
	&= -\pa{-m^2Q+\zeta m(1+zm+m^2+q_t^{-2}\xi^{(4)}m^4)(m+z+2\zeta m)}.
	\end{split}
	\end{equation*}

	Therefore, $w_{I_{1,0}}\E \qb{I_{1,0}}$ cancels the last term and $w_{I_{3,0}}\E \qb{I_{3,0}}$ cancels the middle term. Thus the whole right side of \eqref{eq:expand P^2D} is negligible. This proves Lemma \ref{lemma:recursive}.
\end{proof}

We now choose an initial (small) $\epsilon >0$. Again, we use the factor $N^\epsilon$ and allow $\epsilon$ to increase by a tiny amount from line to line. We often drop $z$ from the notation; it always understood that $z \in \mathcal{D}_\ell$ and all estimates uniform on $\mathcal{D}_\ell$ and also for sufficiently large $N$. The proof of Lemma \ref{lemma:truncated} is done in remaining Subsections \ref{sec:subsec1}-\ref{sec:subsec5} where $\E I_{r,s}$ controlled.
For the estimates in Lemma \ref{lemma:truncated}, we use the following power counting argument frequently.
\begin{lemma}{(Lemma 6.5 of \cite{LS18})}\label{lemma:powercounting0} For any $\mf{i}$ and $\mf{k}$,
	\begin{equation}
	\frac{1}{N}\sum_{\mf{j}=1}^{N}\absa{G_\mf{ij}(z)G_\mf{jk}(z)} \prec \frac{\im m(z)}{N\eta}, \qquad \frac{1}{N}\sum_{\mf{j}=1}^{N}\absa{G_\mf{ij}(z)} \prec \pa{\frac{\im m(z)}{N\eta}}^{1/2}, \qquad (z\in\C^+).
	\end{equation}
	Moreover, for fixed $n\in \N$,
	\begin{equation}
	\frac{1}{N^n} \sum_{\mf{j}_1, \mf{j}_2, \cdots,\mf{j}_n =1}^{N} \absa{G_{\mf{ij}_1}(z) G_{\mf{j}_1\mf{j}_2}(z)G_{\mf{j}_2\mf{j}_3}(z) \cdots G_{\mf{j}_n\mf{k}}(z)  } \prec \pa{\frac{\im m(z)}{N\eta}}^{n/2}, \qquad (z\in\C^+).
	\end{equation}
\end{lemma}

\subsection{Estimate on \texorpdfstring{$I_{2,0}$}{I2,0}}\label{sec:subsec1}
Recall the definition of $I_{r,s}$. We have
\[
I_{2,0}:= \frac{1}{N} \sum_{i\neq k} \kappa_{ij}^{(3)}{ \pB{\partial_{ij}^{2} G_{ij}} \pB{QP^{D-1}\ol{P^D}}}
\] 
We note that $I_{2,0}$ contains terms with one or three off-diagonal Green function entries $G_{ij}$.

\begin{remark}{[Power counting \MakeUppercase{\romannumeral 1}]}\label{rmk:power1}
	Consider the terms $I_{r,0}$, $r\geq 1$. For $n \geq 1$, we split
	\begin{equation}\label{eq:powercounting1}
	w_{I_{2n,0}}I_{2n,0}=\sum_{l=0}^{n}w_{I_{2n,0}^{(2l+1)}}I_{2n,0}^{(2l+1)}, \qquad w_{I_{2n-1,0}}I_{2n-1,0}=\sum_{l=0}^{n}w_{I_{2n-1,0}^{(2l)}}I_{2n-1,0}^{(2l)},
	\end{equation}
	according to the parity of $r$. For example, for $r=2$, $w_{I_{2,0}}I_{2,0} = w_{I_{2,0}^{(1)}}I_{2,0}^{(1)} +w_{I_{2,0}^{(3)}} I_{2,0}^{(3)}$ with
	\begin{align*}
	\E I_{2,0}^{(1)}=N\E \qa{\sum_{i\neq j} \frac{\kappa_{ij}^{(3)}}{N^2}G_{ij}G_{ii}G_{jj}QP^{D-1}\ol{P^D}    }, \qquad
	\E I_{2,0}^{(3)}=N\E \qa{\sum_{i\neq j} \frac{\kappa_{ij}^{(3)}}{N^2}(G_{ij})^3QP^{D-1}\ol{P^D}   }.
	\end{align*}
	Now we bound the summands in \eqref{eq:powercounting1} as follows. First, we note that each term in $I_{r,0}$ contains a factor of $q_t^{(r-1)+}$. Second, with $n\geq 1$, for $\E I_{2n,0}^{(2l+1)}$ and $\E I_{2n-1,0}^{(2l)}$ we can apply Lemma \ref{lemma:powercounting0} to extract one factor of $\frac{\im m}{N\eta}$ and other Green function entries can be bounded by $|G_{ij}| \prec 1$. Moreover $Q$ is bounded by some constant on $\mathcal{E}$. Therefore, for $n \geq 1, l \geq 1$,	
	\begin{equation}\label{eq:powercount1}
	|\E I_{2n,0}^{(2l+1)} | \leq \frac{N^\epsilon}{q_t^{2n-1}}\E \qa{\frac{\im m}{N\eta}|P|^{2D-1}}, 
	\qquad  	|\E I_{2n-1,0}^{(2l)} | \leq \frac{N^\epsilon}{q_t^{2n-2}}\E \qa{\frac{\im m}{N\eta}|P|^{2D-1}},
	\end{equation}
	for sufficiently large $N$. Hence we conclude that all these terms are negligible.
\end{remark}

Thus we only need to consider $\E I_{2,0}^{(1)}$ which is not covered by \eqref{eq:powercount1}. With Lemma \ref{lemma:powercounting0}, we have for sufficiently large $N$ 
\begin{equation}
|\E I_{2,0}^{(1)}| \leq \frac{N^\epsilon}{q_t}\E \qB{\frac{1}{N^2}\sum_{i\neq j} |G_{ij}|Q|||P|^{(2D-1)}  } \leq \frac{N^\epsilon}{q_t} \E \qB{ \pB{\frac{\im m }{N\eta} }^{1/2}|P|^{2D-1} }.
\end{equation}
Since this bound is not negligible, we need to gain an additional factor $q_t^{-1}$.  We have the following result.

\begin{lemma}\label{lemma:EI_{2,0}^1}
	For any small $\epsilon >0$ and for all $z \in \mathcal{D}_\ell$, we have
	\begin{equation}\label{eq:EI_{2,0}^1}
	|\E I_{2,0}^{(1)}| \leq \frac{N^\epsilon}{q_t^2} \E \qB{ \pB{\frac{\im m }{N\eta} }^{1/2}|P|^{2D-1} } +\Phi_\epsilon \leq N^\epsilon \E \qB{\pB{q_t^{-4} +\frac{\im m }{N\eta} }|P(m)|^{2D-1}  } + \Phi_\epsilon,
	\end{equation}
	for $N$ sufficiently large. In particular, $\E I_{2,0}$ is negligible.
\end{lemma}

\begin{proof}
	Fix a small $\epsilon >0$. Recall the definition of $I_{2,0}^{(1)}$, we have
	\begin{equation}\label{eq:def I_{2,0}}
	\E I_{2,0}^{(1)}=N\E \qa{\sum_{i\neq j} \frac{\kappa_{ij}^{(3)}}{N^2}G_{ij}G_{ii}G_{jj}QP^{D-1}\ol{P^D}    }
	\end{equation}
	Using the resolvent formula we expand in the index $j$ to get 
	\begin{equation}\label{eq:zEI_{2,0}^{(1)}}
	z\E I_{2,0}^{(1)} = N\E \qa{\sum_{i\neq j \neq k} \frac{\kappa_{ij}^{(3)}}{N^2}H_{jk}G_{ki}G_{ii}G_{jj}QP^{D-1}\ol{P^D}    }.
	\end{equation}
	Applying the cumulatn expansion to the right side of \eqref{eq:zEI_{2,0}^{(1)}}, we will show that the leading term of \eqref{eq:zEI_{2,0}^{(1)}} is $-\E[mI_{2,0}^{(1)}]$. Then changing $m(z)$ by the deterministic quantity $\wt m (z)$ and showing all other terms in the cumulant expansion are negligible. Then we will get 
	\begin{equation}
	|z+\wt m(z)||\E I_{2,0}^{(1)}|\leq  \frac{N^\epsilon}{q_t^2} \E \qB{ \pB{\frac{\im m }{N\eta} }^{1/2}|P|^{2D-1} } +\Phi_\epsilon \leq N^\epsilon \E \qB{\pB{q_t^{-4} +\frac{\im m }{N\eta} }|P(m)|^{2D-1}  } + \Phi_\epsilon,
	\end{equation}
	for sufficiently large $N$. Since $|z+\wt m(z)| > 1/6 $ uniformly on $\mathcal{E}$, the lemma follows directly.                                                                                                                                                                                                                                                                                                                                                                                                                                                                                                                                                                                                                                                                                                                                                                                                                                                                                                                                                                                                                                                                                                                                                                                                                                                                                                                                                                                                                                                                                                                                                                                                                                                                                                                                                                                                                                                                                                                                                                                                                                                                                                                                                                                                                                                                                                                                                                                                                                                                                                                                                                                                                                                                                                                                                                                                                                                                                                                                                                                                                                                                                                                                                                                         
	
	For simplicity we abbreviate $\hat{I}\equiv I_{2,0}^{(1)}$. Using Lemma \ref{lemma:Stein}, for arbitrary $l' \in \N$ we have the cumulant expansion
	\begin{equation}\label{eq:zEhat I}
	z\E \hat{I} = \sum_{r'=1}^{l'} \sum_{s'=0}^{r'} w_{\hat{I}_{r',s'}} \E \hat{I}_{r',s'} + \E \Omega_{l'}(\hat{I})
	\end{equation}
	with
	\begin{equation}
	\hat{I}_{r',s'}=\frac{1}{N} \sum_{i\neq j\neq k} \kappa_{ij}^{(3)} \kappa_{jk}^{(r'+1)}{ \pB{\partial_{jk}^{r'-s'} G_{ki}G_{ii}G_{jj}} \pB{\partial_{jk}^{s'}\pB{QP^{D-1}\ol{P^D}}}}
	\end{equation}
	with $ w_{\hat{I}_{r',s'}}= \frac{1}{(r'-s')!s'!}$. By Corollary \ref{coro:bound error}, the error term  $\E\Omega_{l'}(\hat{I})$ is negligible with choosing $l' \geq 8D$.
	
	We first estimate $\hat{I}_{r',0}$. For $r'=1$, we compute
	\begin{align}\label{eq:hatI1 split}
	\E \hat{I}_{1,0} =-&\E \qB{\frac{1}{N} \sum_{i\neq j\neq k} \kappa_{ij}^{(3)} \kappa_{jk}^{(2)}{ G_{ji}G_{kk}G_{jj}G_{ii} \pB{QP^{D-1}\ol{P^D}}}  } \notag \\ 
	&-3\E \qB{\frac{1}{N} \sum_{i\neq j\neq k} \kappa_{ij}^{(3)} \kappa_{jk}^{(2)}{ G_{jk}G_{ki}G_{jj}G_{jj} \pB{QP^{D-1}\ol{P^D}}}  } \notag \\ 
	&-2\E \qB{\frac{1}{N} \sum_{i\neq j\neq k} \kappa_{ij}^{(3)} \kappa_{jk}^{(2)}{ G_{ij}G_{jk}G_{ki}G_{jj} \pB{QP^{D-1}\ol{P^D}}}  } \notag \\ 
	=: \E &\hat{I}_{1,0}^{(1)} + 3\E \hat{I}_{1,0}^{(2)} + 2\E \hat{I}_{1,0}^{(3)}
	\end{align}
	where we organize the terms according to the off-diagonal Green function entries. By Lemma \ref{lemma:powercounting0},
	\begin{equation}\label{eq:hatI^2,3}
	|\E\hat{I}_{1,0}^{(2)}| \leq \frac{N^\epsilon}{q_t}\E\qB{\frac{\im m }{N\eta}|P|^{2D-1}}\leq \Phi_\epsilon, \qquad |\E\hat{I}_{1,0}^{(3)}| \leq \frac{N\epsilon}{q_t}\qB{\pB{\frac{\im m }{N\eta}}^{3/2}|P|^{2D-1} } \leq \Phi_\epsilon.
	\end{equation}
	Recall $\wt m \equiv \wt m_t(z)$ defined in Proposition \ref{prop:locallaw}. We rewrite $\hat{I}_{1,0}^{(1)}$ with $\wt m$ as 
	\begin{align}\label{eq:hatI^1}
	\E \hat{I}_{1,0}^{(1)} &= -\E \qB{\frac{1}{N} \sum_{i\neq j\neq k} \kappa_{ij}^{(3)} \kappa_{jk}^{(2)}{ G_{ji}G_{kk}G_{jj}G_{ii} \pB{QP^{D-1}\ol{P^D}}}  } \notag \\
	&= -\E \qB{\frac{1}{N} \sum_{i\neq j} \kappa_{ij}^{(3)} { mG_{ji}G_{jj}G_{ii} \pB{QP^{D-1}\ol{P^D}}}  } +O(\Phi_\epsilon) \notag \\
	&= -\E \qB{\frac{1}{N} \sum_{i\neq j} \kappa_{ij}^{(3)} { \wt mG_{ji}G_{jj}G_{ii} \pB{QP^{D-1}\ol{P^D}}}  } \notag \\
	&\qquad -\E \qB{\frac{1}{N} \sum_{i\neq j} \kappa_{ij}^{(3)} { (m-\wt m)G_{ji}G_{jj}G_{ii} \pB{QP^{D-1}\ol{P^D}}}  }  +O(\Phi_\epsilon)
	\end{align}
	By Schwarz inequality and the high probability bounds $G_{ii} \prec 1$, for sufficiently large $N$, the second term in \eqref{eq:hatI^1} bounded as 
	\begin{align}\label{eq:hatI^1 2term}
	\absbb{\E \qB{\frac{1}{N} &\sum_{i\neq j} \kappa_{ij}^{(3)} { (m-\wt m)G_{ji}G_{jj}G_{ii} \pB{QP^{D-1}\ol{P^D}}}  } } \leq \frac{N^{\epsilon/4}}{q_t}\E \qB{\frac{1}{N^2} \sum_{i\neq j} |m-\wt m||G_{ij}||Q||P|^{2D-1}} \notag \\
	& \leq \frac{N^{-\epsilon/4}}{q_t} \E\qB{\frac{1}{N^2}\sum_{i\neq j} |m-\wt m|^2|P|^{2D-1}} + \frac{N^{3\epsilon/4}}{q_t}\qB{\E \frac{1}{N^2} \sum_{i\neq j}|G_{ij}|^2|P|^{2D-1}} \notag \\
	& \leq \frac{N^{-\epsilon/4}}{q_t}\E\qB{|m-\wt m|^2|P|^{2D-1} } +\frac{N^{3\epsilon/4}}{q_t}\E\qB{\frac{\im m}{N\eta}|P|^{2D-1}}.
	\end{align}
	Thus, we get from \eqref{eq:hatI1 split}, \eqref{eq:hatI^2,3}, \eqref{eq:hatI^1} and \eqref{eq:hatI^1 2term} that
	\begin{equation}\label{eq:EhatI_{1,0}}
	\E \hat{I}_{1,0} = -\wt m \E \qB{\frac{1}{N} \sum_{i\neq j} \kappa_{ij}^{(3)} { G_{ji}G_{jj}G_{ii} \pB{QP^{D-1}\ol{P^D}}}  } + O(\Phi_\epsilon) = -\E \wt m I_{2,0}^{(1)} + O(\Phi_\epsilon).
	\end{equation}
	We remark that in the expansion of $\E \hat{I} = \E I_{2,0}^{(1)} $ the only term with one off-diagonal entry is $\E\hat{I}_{2,0}^{(1)}$. All the other terms contain at least two off-diagonal entries.
	
	\begin{remark}{[Power counting \MakeUppercase{\romannumeral 2}]} We have $\hat{I}_{r',s'} = (I_{2,0}^{(1)})_{r',s'}$. Now consider the terms with $s'=0$. As in \eqref{eq:powercounting1} we organize the terms according to the number of off-diagonal Green function entries. For $r'\geq 2$, 
		\begin{equation}
		w_{\hat{I}_{r',0}}\hat{I}_{r',0}=\sum_{l=0}^{n}w_{\hat{I}_{r',0}^{(l+1)}}\hat{I}_{r',0}^{(l+1)}=\sum_{l=0}^{n}w_{\hat{I}_{r',0}^{(l+1)}}{(I_{2,0}^{(1)})}_{r',0}^{(l+1)}.
		\end{equation}
		A simple power counting as in Remark \ref{rmk:power1} then directly yields
		\begin{equation}\label{eq:powercount2}
		|\E\hat{I}_{r',0}^{(1)}| \leq \frac{N^\epsilon}{q_t^{r'}}\E \qbb{\pa{\frac{\im m}{N\eta}}^{1/2}|P|^{2D-1}}, \qquad |\E\hat{I}_{r',0}^{(l+1)}| \leq \frac{N^\epsilon}{q_t^{r'}}\E \qbb{\frac{\im m}{N\eta}|P|^{2D-1}}, \qquad (l \geq 1),
		\end{equation}
		for $N$ sufficiently large. We used that each terms contains a factor $\kappa_{(\b{\cdot})}^{(3)}\kappa_{(\b{\cdot})}^{(r'+1)}\leq CN^{-2}q_t^{-r'}$. Hence with $r' \geq 2$, we conclude that all terms in \eqref{eq:powercount2} are negligible, yet we remark that $|\E \hat{I}_{1,0}^{(1)}|$ is the leading error term in $\E I_{2,0}^{(1)}$, which is explicitly listed on the right side of \eqref{eq:EI_{2,0}^1}.
	\end{remark}
	
	\begin{remark}{[Power counting \MakeUppercase{\romannumeral 3}]}\label{rmk:power3} We consider the terms $\hat{I}_{r',s'}$ with $1\leq s'\leq r'$.
		Recall that 
		\begin{equation}\label{eq:hat I}
		\hat{I}_{r',s'}=\frac{1}{N} \sum_{i\neq j\neq k} \kappa_{ij}^{(3)} \kappa_{jk}^{(r'+1)}{ \pB{\partial_{jk}^{r'-s'} G_{ki}G_{ii}G_{jj}} \pB{\partial_{jk}^{s'}\pB{QP^{D-1}\ol{P^D}}}}.
		\end{equation}
		We claim that it's enough to show that
		\begin{equation}\label{eq:wt I}
		\wt	{I}_{r',s'}:=\frac{1}{N} \sum_{i\neq j\neq k} \kappa_{ij}^{(3)} \kappa_{jk}^{(r'+1)}{ \pB{\partial_{jk}^{r'-s'} G_{ki}G_{ii}G_{jj}} \pB{\partial_{jk}^{s'}\pB{P^{D-1}\ol{P^D}}}},
		\end{equation}
		are negligible for $1\leq s'\leq r'$.
		Assume that \eqref{eq:wt I} holds. $\partial_{jk}$ can act on $Q, P^{D-1}$ or $\ol{P^D}$. If $\partial_{jk}$ does not act on $Q$, with $|Q| \prec 1 $, we can show that those terms from \eqref{eq:hat I} are negligible by \eqref{eq:wt I}. Now consider when $\partial_{jk}$ act only on $Q$. Note that $Q$ is second order polynomial in $m$.
		
		Using $|Q'|\prec 1$, $|Q''|\prec 1 $ and Lemma \ref{lemma:powercounting0} we have 
		\begin{align}\label{eq:partial Q}
		\abs{\partial_{jk} Q} &= \absB{\pB{\frac{1}{N}\sum_{u=1}^{N}G_{uj}G_{ku}} Q'} \prec \frac{\im m }{N\eta}, \notag \\
		\abs{\partial_{jk}^2 Q} &\leq \absB{\pB{\frac{1}{N}\sum_{u=1}^{N}G_{uj}G_{ku}}^2 Q''} + \absB{\frac{1}{N}\sum_{u=1}^{N}\partial_{ik}(G_{uj}G_{ku}) Q' } \prec {\frac{\im m }{N\eta}},
		\end{align}
		where the summation index $u$ is generated from $\partial_{ik}Q$. More generally, it can be easily shown that $\partial_{jk}^{n}Q$ contains at least two off-diagonal Green function entries for $n\geq 1$. 
		Thus we conclude that if $\partial_{jk}$ acts $n$ times on $Q$ and $m$ times on $P^{D-1}\ol{P^D}$ , then that term could be bounded by the term that $\partial_{jk}$ acts $m$ times only on $P^{D-1}\ol{P^D}$ where $m\geq 1$.
		Hence we only need to show that 
		\begin{align}
		\frac{1}{N} \sum_{i\neq j\neq k} \kappa_{ij}^{(3)} \kappa_{jk}^{(r'+1)}{ \pB{\partial_{jk}^{(r'-n)} G_{ki}G_{ii}G_{jj}} \pB{\pB{\partial_{jk}^n Q}P^{D-1}\ol{P^D}}},
		\end{align}
		is negligible. In particular, we have 
		\begin{equation}
		\absa{\frac{1}{N} \sum_{i\neq j\neq k} \kappa_{ij}^{(3)} \kappa_{jk}^{(r'+1)}{ \pB{\partial_{jk}^{(r'-n)} G_{ki}G_{ii}G_{jj}} \pB{\pB{\partial_{jk}^{n}Q}P^{D-1}\ol{P^D}}}}\leq \frac{N^\epsilon}{q_t^{r'}}\E \qbb{ \pB{\frac{\im m}{N\eta}}^{3/2}|P|^{2D-1}} \leq \Phi_\epsilon,
		\end{equation}
		where we used that $\partial_{jk}^{r}( G_{ki}G_{ii}G_{jj})$, $r\geq 0$, contains at least one off-diagonal Green function entry.
		
		Now we prove $\wt	{I}_{r',s'}$, $1\leq s'\leq r'$ are negligible. For $s'=1$, since $\partial_{jk}^{r}( G_{ki}G_{ii}G_{jj})$ contains at least one off-diagonal entries and $\partial_{ik}(P^{D-1}\ol{P^D})$ contains two off-diagonal Green function entries. Explicitly,
		\begin{align}
		\wt{I}_{r',1} =& -2(D-1)\frac{1}{N} \sum_{i\neq j\neq k} \kappa_{ij}^{(3)} \kappa_{jk}^{(r'+1)}{ \pB{\partial_{jk}^{r'-1} G_{ki}G_{ii}G_{jj}} \pB{\frac{1}{N}\sum_{u=1}^{N}G_{uj}G_{ku} }}P'P^{D-2}\ol{P^D} \notag \\
		& -2D\frac{1}{N} \sum_{i\neq j\neq k} \kappa_{ij}^{(3)} \kappa_{jk}^{(r'+1)}{ \pB{\partial_{jk}^{r'-1} G_{ki}G_{ii}G_{jj}} \pB{\frac{1}{N}\sum_{u=1}^{N}G_{uj}G_{ku} }}\ol{P}'P^{D-1}\ol{P^{D-1}}.
		\end{align}
		Using Lemma \ref{lemma:powercounting0}, for $r'\geq1$, we have 
		\begin{equation}\label{eq:wt I_{r',1}}
		|\E \wt{I}_{r',1}|\leq \frac{N^\epsilon}{q_t^{r'}}\E \qbb{ \pB{\frac{\im m}{N\eta}}^{3/2}|P'||P|^{2D-2}+\pB{\frac{\im m}{N\eta}}^{3/2}|P'||P|^{2D-2} }\leq 2\Phi_\epsilon,
		\end{equation}
		for sufficiently large $N$.
		For $2\leq s'\leq r'$, we first note that, for sufficiently large $N$,
		\begin{align}\label{eq:wt I_{r,s}}
		|\E \wt{I}_{r',s'}|&\leq \frac{N^\epsilon}{q_t^{r'}} \absbb{\E \qbb{\frac{1}{N^3} \sum_{i\neq j\neq k} { \pB{\partial_{jk}^{r'-s'} G_{ki}G_{ii}G_{jj}} \pB{\partial_{jk}^{s'}\pB{P^{D-1}\ol{P^D}}}}}} \notag \\
		&\leq \frac{N^\epsilon}{q_t^{r'}}\E \qbb{ \pB{\frac{\im m}{N\eta}}^{1/2}\frac{1}{N^2}\sum_{j\neq k} \absB{\partial_{jk}^{s'}\pB{P^{D-1}\ol{P^D}}} }.
		\end{align}
		Next, since $s' \geq 2$, the partial derivative $\partial_{jk}^{s'}\pB{P^{D-1}\ol{P^D}}$ acts on $P$ and $\ol{P}$ (and on their derivatives) more than once. Generally, we consider a resulting term containing
		\begin{equation}\label{eq:fresh index}
		P^{D-s_1'}\ol{P^{D-s_2'}}(P')^{s_3'}(\ol{P'})^{s_4'}(P'')^{s_5'}\cdots (P^{(5)})^{s_{11}'}(\ol{P^{(5)}})^{s_{12}'},
		\end{equation}
		with $1\leq s_1' \leq D, 0\leq s_2' \leq D$ and $\sum_{n=1}^{12}s_n'\leq s'$. Since $P^{(6)}$ is constant we did not list it. We see that such a term above was generated from $P^{D-1}\ol{P^D}$ by letting the partial derivative $\partial_{jk}$ act $s_1' -1$-times on $P$ and $s_2' -1$-times on $\ol{P}$, which implies that $s_1' -1 \geq s_3'$ and $s_2' \geq s_4'$. If  $s_1' -1 > s_3'$, then $\partial_{jk}$ acted on the derivative of $P$ directly $(s_1'-1-s_3')$-times, and a similar argument holds for $\ol{P'}$. Whenever $\partial_{jk}$ acted on $P, \ol{P}$ and their derivatives, it generated a term $2N^{-1}\sum_{a_l}G_{ja_l}G_{a_lk}$, with $a_l, l\geq 1$, a fresh summation index. For each fresh summation index we apply Lemma \ref{lemma:powercounting0} to gain a factor $\frac{\im m}{N\eta}$. The total number of fresh summation indices in a term corresponding to \eqref{eq:fresh index} is 
		\[
		(s_1'-1)+s_2'+(s_1'-1-s_3')+(s_2'-s_4')=2s_1'+2s_2'-s_3'-s_4'-2=2\wt s_0 -\wt s -2,
		\]
		with $\wt s_0:=s_1'+s_2'$ and $\wt s:= s_3'+s_4'$ we note that this number does not decrease when $\partial_{jk}$ acts on off-diagonal Green function entries later. Thus, from \eqref{eq:wt I_{r,s}} we conclude, for $2\leq s' \leq r'$,
		\begin{align}\label{eq:wt I_{r,s}2}
		|\E \wt{I}_{r',s'}|	&\leq \frac{N^\epsilon}{q_t^{r'}}\E \qbb{ \pB{\frac{\im m}{N\eta}}^{1/2}\frac{1}{N^2}\sum_{j\neq k} \absB{\partial_{jk}^{s'}\pB{P^{D-1}\ol{P^D}}} } \notag \\
		&\leq \frac{N^{2\epsilon}}{q_t^{r'}}\sum_{\wt s_0=2}^{2D}\sum_{\wt s=1}^{\wt s_0 -2}\E \qbb{ \pB{\frac{\im m}{N\eta}}^{1/2+2\wt s_0 - \wt s -2}|P'|^{\wt s}|P|^{2D-\wt s_0} } \\
		& \qquad +\frac{N^{2\epsilon}}{q_t^{r'}}\sum_{\wt s_0=2}^{2D}\E \qbb{ \pB{\frac{\im m}{N\eta}}^{1/2+\wt s_0 - 1}|P'|^{\wt s_0-1}|P|^{2D-\wt s_0} },\notag
		\end{align}
		for $N$ sufficiently large. Here the last term on the right corresponds to $\wt s = \wt s_0 -1$. Thus, we conclude from \eqref{eq:wt I_{r,s}2} that $\E[\wt I_{r',s'}]]$, $2\leq s'\leq r'$, is negligible. To sum up, we have established that all terms $\E[\wt I_{r',s'}]$ with $1 \leq s'\leq r'$ are negligible and therefore, $\hat{I}_{r',s'}$ with $1\leq s'\leq r'$ are also negligible.
	\end{remark}
	From \eqref{eq:def I_{2,0}}, \eqref{eq:zEhat I}, \eqref{eq:EhatI_{1,0}}, \eqref{eq:powercount2}, \eqref{eq:wt I_{r',1}} and \eqref{eq:wt I_{r,s}2} we find that 
	\begin{equation}
	|z+\wt m||\E I_{2,0}^{(1)}| \leq \frac{N^\epsilon}{q_t^2}\E \qB{\pB{\frac{\im m}{N\eta}}^{1/2}|P|^{2D-1}} + \Phi_\epsilon,
	\end{equation}
	for sufficiently large $N$. Since $|z+\wt m |> 1/6$, we obtain $|\E I_{2,0}^{(1)}|\leq \Phi_\epsilon$. This concludes the proof of \eqref{eq:EI_{2,0}^1}.	
\end{proof}	
Summarizing, we showed in \eqref{eq:powercount1}and \eqref{eq:EI_{2,0}^1} that 
\begin{equation}
|EI_{2,0}|\leq \Phi_\epsilon,
\end{equation}
for $N$ sufficiently large and the second estimate in \eqref{eq:estimate I_1} is proved.

\subsection{Estimate on \texorpdfstring{$I_{3,0}$}{I3,0}} Note that $I_{3,0}$ contains terms with zero, two or four off-diagonal Green function entries. We split accordingly
\[
w_{I_{3,0}}I_{3,0} = w_{I_{3,0}^{(0)}}I_{3,0}^{(0)} + w_{I_{3,0}^{(2)}}I_{3,0}^{(2)} +w_{I_{3,0}^{(4)}}I_{3,0}^{(4)}. 
\]
When there are two off-diagonal entries, from Lemma \ref{lemma:powercounting0}, we obtain
\begin{equation}
|\E I_{3,0}^{(2)} | \leq \absbb{ N\max_{i,j}\kappa_{ij}^{(4)} \E \qbb{\frac{1}{N^2} \sum_{i\neq j} G_{ii}G_{jj}(G_{ij})^2 QP\ol{P^D} }   } \leq \frac{N^\epsilon}{q^2_t}\E \qbb{\frac{\im m}{N\eta} |Q||P|^{2D-1}} \leq \Phi_\epsilon,
\end{equation}
for sufficiently large $N$ and similar argument holds for $\E I_{3,0}^{(4)}$. Thus the only non-negligible term is $I_{3,0}^{(0)}$.

\begin{align}
&w_{I_{3,0}^{(0)}}\E I_{3,0}^{(0)} = -\frac{1}{N}\E \qbb{ \sum_{i\neq j} \kappa_{ij}^{(4)}G_{ii}^2 G_{jj}^2QP^{D-1}\ol{P^D}} \nonumber \\
&= -\frac{1}{N}\E \qbb{ \sum_{i, j} \kappa_{d}^{(4)}G_{ii}^2 G_{jj}^2QP^{D-1}\ol{P^D}} -\frac{1}{N}\E \qbb{ \sum_{i}\sum_{j\sim i} (\kappa_{s}^{(4)}-\kappa_{d}^{(4)})G_{ii}^2 G_{jj}^2QP^{D-1}\ol{P^D}} 	\nonumber \\ 
&=-\frac{1}{N}\kappa_{d}^{(4)}\E \qbb{\sum_{i, j}  G_{ii}^2 G_{jj}^2 QP^{D-1}\ol{P^D}}-\frac{1}{N}(\kappa_{s}^{(4)}-\kappa_{d}^{(4)})\E \qbb{ \sum_{i}\sum_{j\sim i} G_{ii}^2 G_{jj}^2QP^{D-1}\ol{P^D}} 	\nonumber\\
&=-\frac{1}{N}\kappa_{d}^{(4)}\E \qbb{\sum_{i, j}  G_{ii}^2 G_{jj}^2QP^{D-1}\ol{P^D}}-\frac{1}{N}(\kappa_{s}^{(4)}-\kappa_{d}^{(4)})\E \qbb{ \sum_{i}\sum_{j\sim i} G_{ii}^2 G_{jj}^2QP^{D-1}\ol{P^D}} .
\end{align}

We have
\begin{align}
G_{ii}^2 G_{jj}^2 &= (G_{ii}^2-m^2) (G_{jj}^2-m^2) + m^2 G_{ii}^2 + m^2 G_{jj}^2 -m^4 \notag \\
&= O(\psi^2) + m^2 \pB{(G_{ii}-m)^2 + 2G_{ii}m -m^2   } + m^2 \pB{(G_{jj}-m)^2 + 2G_{jj}m -m^2   } -m^4 \notag \\
&=O(\psi^2) + 2m^3(G_{ii}+G_{jj}) -3m^4,
\end{align}
where $|G_{ii}-m|\prec \psi$ by weak local semicircle law.
Therefore, for the first term, we can conclude that 
\begin{align}
\frac{1}{N}\kappa_{d}^{(4)}\E \qbb{\sum_{i, j}  G_{ii}^2 G_{jj}^2QP^{D-1}\ol{P^D}} &= 	\frac{1}{N}\kappa_{d}^{(4)}\E \qbb{\sum_{i, j}  \pa{2m^3(G_{ii}+G_{jj}) -3m^4}QP^{D-1}\ol{P^D}} +O(\frac{N^\epsilon}{q_t^2}\psi^2) \notag \\
&=N\kappa_{d}^{(4)}\E \qbb{m^4 QP^{D-1}\ol{P^D} } + O(\frac{N^\epsilon}{q_t^2}\psi^2).
\end{align}

Similarly we can estimate the second term by 	
\begin{align}
\frac{1}{N}(\kappa_{s}^{(4)}-\kappa_{d}^{(4)})&\E \qbb{ \sum_{i}\sum_{j\sim i} G_{ii}^2 G_{jj}^2QP^{D-1}\ol{P^D}} \notag \\
&= 	\frac{1}{N}(\kappa_{s}^{(4)}-\kappa_{d}^{(4)}) \E \qbb{ \sum_{i}\sum_{j\sim i}\pa{ 2m^3(G_{ii}+G_{jj}) -3m^4 }QP^{D-1}\ol{P^D}} + O(\frac{N^\epsilon}{q_t^2}\psi^2) \notag \\
&=\frac{N}{k}(\kappa_{s}^{(4)}-\kappa_{d}^{(4)})\E \qbb{m^4QP^{D-1}\ol{P^D}}+O(\Phi_\epsilon).
\end{align}

Therefore we obtain
\begin{align}
&w_{I_{3,0}^{(0)}}\E I_{3,0}^{(0)} = -\frac{1}{N}\E \qbb{ \sum_{i\neq j} \kappa_{ij}^{(4)}G_{ii}^2 G_{jj}^2QP^{D-1}\ol{P^D}} \nonumber \\
&=-N\kappa_{d}^{(4)}\E \qbb{m^4QP^{D-1}\ol{P^D}}- \frac{N}{k}(\kappa_{s}^{(4)}-\kappa_{d}^{(4)})\E \qbb{m^4QP^{D-1}\ol{P^D}}+O(\Phi_\epsilon) \nonumber \\
& -\E \qB{q_t^{-2}\xi^{(4)}Qm^4 P^{D-1}\ol{P^D}}+O(\Phi_\epsilon).
\end{align}

\subsection{Estimate on \texorpdfstring{$I_{r,0}$ for $r\geq 4$}{Ir,0}} For $r \geq 5$ we use the bound $|G_{ii}| \prec 1 $ to obtain
\begin{align}
|\E I_{r,0}| &\leq \absB{N \E \qB{\frac{1}{N^2} \kappa_{ij}^{(r+1)} \sum_{i\neq j}  (\partial_{ij}^r G_{ij})P^{D-1}\ol{P^D} } }\notag \\
&\leq \frac{N^\epsilon}{q_t^4} \E\qbb{\frac{1}{N^2}\sum_{i\neq j}|P|^{2D-1}} \leq \frac{N^\epsilon}{q_t^4} \E \qb{|P|^{2D-1}} \leq \Phi_\epsilon,
\end{align}
for sufficiently large $N$. For $r=4$, $\partial_{ij}^r G_{ij}$ contains at least one off-diagonal term. Hence 
\begin{align}
\absB{N \E \qB{\frac{1}{N^2} \kappa_{ij}^{(r+1)} \sum_{i\neq j}  (\partial_{ij}^r G_{ij})P^{D-1}\ol{P^D} } } &\leq \frac{N^\epsilon}{q_t^3}  \E \qB{\frac{1}{N^2} \sum_{i\neq j} |G_{ij}||P|^{2D-1}}   \notag \\
&=\frac{N^\epsilon}{q_t^3} \E \qB{ \pa{\frac{\im m}{N\eta} }^{1/2}|P|^{2D-1}} \leq \Phi_\epsilon,
\end{align}
for $N$ sufficiently large. Thus we can conclude that all $I_{r,0}, r\geq 4$ are negligible. 

\subsection{Estimate on \texorpdfstring{$I_{r,s}$ for $r\geq 2, s\geq 1$}{Ir,s}} Similar to Remark \ref{rmk:power3}, if $\partial_{jk}$ act only on $Q$ and does not act on $P^{D-1}$ or $\ol{P^D}$ then it can be easily shown that those terms are negligible. We leave details for the reader. Hence we only need to prove that the terms
$$\dt{I}_{r,s}  :=\frac{1}{N} \sum_{i\neq j} \kappa_{ij}^{(r+1)}{ \pB{\partial_{ij}^{r-s} G_{ij}} \pB{\partial_{ij}^{s}\pB{P^{D-1}\ol{P^D}}}},\qquad\quad (r\geq 2,  s\geq 1),$$
are negligible for $N$ sufficiently large. 

For $r\geq 2$ and $s=1$, we have 
\begin{equation*}
\E \dt{I}_{r,1} =\E \qbb{\frac{1}{N} \sum_{i\neq j} \kappa_{ij}^{(r+1)}{ \pB{\partial_{ij}^{r-1} G_{ij}} \pB{\partial_{ij}\pB{P^{D-1}\ol{P^D}}}}}.
\end{equation*}
Note that each term in 	$\E \dt{I}_{r,1}$, $r\geq2$, contains at least two off-diagonal Green function entries. For the terms with at least three off-diagonal Green function entries, we use the bound $|G_{ij}|\prec 1$ and Lemma \ref{lemma:powercounting0} to get 
\begin{align}
\E \qbb{\frac{1}{N^2} &\sum_{i,j,k} \kappa_{ij}^{(r+1)} |G_{ij}G_{jk}G_{ki}||P'||P|^{2D-2}} \leq N^\epsilon\frac{1}{q_t}\E\qbb{\pB{\frac{\im m}{N\eta}}^{3/2}|P'||P|^{2D-2}}\notag \\
&\leq N^\epsilon\E \qbb{\sqrt{\im m}\frac{\im m}{N\eta}\pB{\frac{1}{N\eta} +q_t^{-2}} |P'||P|^{2D-2}  }
\end{align}
for $N$ sufficiently large. Since $\im m \prec 1$, the right hand side is negligible.
Denoting the terms with two off-diagonal Green function entries in $\E\dt{I}_{r,1}$ by   $\E\dt{I}_{r,1}^{(2)}$, we have 
\begin{equation}\label{eq:dtI_{r,1}}
\begin{split}
\E\dt{I}_{r,1}^{(2)} = N\E\qbb{ \frac{2(D-1)}{N^2} \sum_{i\neq j} \kappa_{ij}^{(r+1)}G_{ii}^{r/2}G_{jj}^{r/2}\pB{ \frac{1}{N}\sum_{k=1}^{N}G_{ik}G_{kj} }P'P^{D-2}\ol{P^D}   }\\
+N\E\qbb{ \frac{2D}{N^2} \sum_{i\neq j} \kappa_{ij}^{(r+1)}G_{ii}^{r/2}G_{jj}^{r/2}\pB{ \frac{1}{N}\sum_{k=1}^{N}\ol{G_{ik}G_{kj}} }\ol{P}'P^{D-1}\ol{P^{D-1}}   },
\end{split}
\end{equation}
where we noted that $r$ is necessarily even in this case. Then from Lemma \ref{lemma:powercounting0} we have the upper bound
\begin{equation}
|\E\dt{I}_{r,1}^{(2)}|\leq \frac{N^\epsilon}{q_t^{r-1}}\E\qB{\frac{\im m}{N\eta}|P|'|P|^{2D-2}  }, \qquad \qquad (r >2),
\end{equation}
for sufficiently large $N$, which is negligible. However for $r=2$, we need an additional factor $q_t^{-1}$. This can be done as in the proof of Lemma \ref{lemma:EI_{2,0}^1} by considering two off-diagonal terms $G_{jk}G_{ki}$, generated from $\partial_{ij}P(m)$.
\begin{lemma}
	For any small $\epsilon>0$, we have
	\begin{equation}
	|\E\dt{I}_{2,1}^{(2)}|\leq \frac{N^\epsilon}{q_t^{2}}\E\qB{\frac{\im m}{N\eta}|P|'|P|^{2D-2}} +\Phi_\epsilon,
	\end{equation}
	for $N$ sufficiently large, uniformly on $\mathcal{D}_\ell$. In particular, $\E\dt{I}_{2,1}$ is negligible.
\end{lemma}
\begin{proof}
	First, we consider the first term of the \eqref{eq:dtI_{r,1}}. We can write
	\begin{equation*}
	\begin{split}
	z N\E&\qbb{ \frac{1}{N^3} \sum_{i\neq j\neq k} \kappa_{ij}^{(3)}G_{jk}G_{ii}G_{jj}G_{ik} P'P^{D-2}\ol{P^D} }\\
	&=N\E\qbb{ \frac{1}{N^3} \sum_{i\neq j\neq k\neq u} \kappa_{ij}^{(3)}H_{ju}G_{uk}G_{ii}G_{jj}G_{ik} P'P^{D-2}\ol{P^D} }.
	\end{split}
	\end{equation*}
	As in the proof of Lemma \ref{lemma:EI_{2,0}^1}, we apply the cumulant expansion to the right side and get the leading term
	\begin{equation}
	N\E\qbb{ \frac{1}{N^3} \sum_{i\neq j\neq k} \kappa_{ij}^{(3)}mG_{jk}G_{ii}G_{jj}G_{ik} P'P^{D-2}\ol{P^D} }.
	\end{equation}
	
	Thanks to the additional factor of $q_t^{-1}$ from the $\kappa_{ij}^{(3)}$, all other terms in the expansion are negligible as can be checked by power counting as in the proof of Lemma \ref{lemma:EI_{2,0}^1}. Replacing in $m$ by $\wt m$ in the leading term, we obtain
	\begin{equation}
	|z+\wt m|N\absbb{\E\qbb{ \frac{1}{N^3} \sum_{i\neq j\neq k} \kappa_{ij}^{(3)}G_{jk}G_{ii}G_{jj}G_{ik} P'P^{D-2}\ol{P^D} }} \leq O(\Phi_\epsilon),
	\end{equation}
	for $N$ sufficiently large. Since $|z+\wt m |> 1/6$, we conclude that the first term of the right side of \eqref{eq:dtI_{r,1}} negligible. In the same way one can shows that the second term is negligible. We omit the details.	
\end{proof}
Hence we conclude that $\E \dt{I}_{r,1}$, $r\geq 2$ are negligible.

Consider next the terms
$$\E\dt{I}_{r,s}  =\E\qbb{\frac{1}{N} \sum_{i\neq j} \kappa_{ij}^{(r+1)}{ \pB{\partial_{ij}^{r-s} G_{ij}} \pB{\partial_{ij}^{s}\pB{P^{D-1}\ol{P^D}}}}}, \qquad\qquad (2\leq s\leq r).$$
We proceed in a similar way as in Remark \ref{rmk:power3}. We note that each term in $\partial_{ij}^{r-s} G_{ij}$ contains sat least one off-diagonal function entry when $r-s$ is even, yet when $r-s$ is odd there is a term with no off-diagonal Green function entries. Since $s\geq2$, the partial derivative $\partial_{ij}^{s}$ acts on $P$ or $\ol{P}$ (or their derivatives) more than once in total. As in Remark \ref{rmk:power3}, consider such a term with 
\[
P^{D-s_1}\ol{P^{D-s_2}}(P')^{s_3}(\ol{P'})^{s_4},
\]
for $1\leq s_1\leq D$ and $0\leq s_2 \leq D$. We do not include other higher derivative terms since $P^{(n)} \prec 1$ for $n=2,3,\cdots,6$ and $P^{(7)}=0$. We see that such a term was generated from $P^{D-1}\ol{P^D}$ by letting the partial derivative $\partial_{ij}$ act $(s_1-1)$-times on $P$ and $s_2$ times on $\ol{P}$, which implies that $s_3\leq s_1 -1$ and $s_4 \leq s_2$. If  $s_1 -1 > s_3$, then $\partial_{ij}$ acted on the derivative of $P$ directly $(s_1-1-s_3)$-times, and a similar argument holds for $\ol{P'}$. Whenever $\partial_{ij}$ acted on $P, \ol{P}$ and their derivatives, it generated a term $2N^{-1}\sum_{a_l}G_{ja_l}G_{a_lk}$, with $a_l, l\geq 1$, a fresh summation index. For each fresh summation index we apply Lemma \ref{lemma:powercounting0} to gain a factor $\frac{\im m}{N\eta}$. The total number of fresh summation indices in this case is 
\[
(s_1-1)+s_2+(s_1-1-s_3)+(s_2-s_4)=2s_1+2s_2-s_3-s_4-2.
\]
Assume first that $r=s$ so that $\partial_{ij}^{r-s} G_{ij} = G_{ij}$. Then applying Lemma \ref{lemma:powercounting0} $(2s_1+2s_2-s_3-s_4-2)$-times and letting $s_0=s_1+s_2$ and $s'=s_3+s_4$, we get an upper bound 
\begin{equation}
|\E \dt{I}_{r,r}|\leq \frac{N^\epsilon}{q_t^{r-1}}\sum_{s_0=2}^{2D}\sum_{s'=1}^{s_0-1}\E\qbb{ \pB{\frac{\im m}{N\eta} }^{1/2}\pB{\frac{\im m}{N\eta} }^{2s_0-s'-2}|P'|^{s'}|P|^{2D-s_0}  }\leq \Phi_\epsilon,
\end{equation}
for sufficiently large $N$. In other words, $\E \dt{I}_{r,r}$ is negligible for $r\geq 2$.

Second, Assume that $2\leq s <r$. Then applying Lemma \ref{lemma:powercounting0}  $(2s_1+2s_2-s_3-s_4-2)$-times, we get 
\begin{align}\label{eq:EI_{r,s}}
|\E \dt{I}_{r,s}|	&\leq \frac{N^{\epsilon}}{q_t^{r-1}}\sum_{s_0=2}^{2D}\sum_{s'=1}^{ s_0 -2}\E \qbb{ \pB{\frac{\im m}{N\eta}}^{2s_0 - s' -2}|P'|^{s'}|P|^{2D-s_0} } \\
&\qquad \qquad +\frac{N^{\epsilon}}{q_t^{r-1}}\sum_{s_0=2}^{2D}\E \qbb{ \pB{\frac{\im m}{N\eta}}^{s_0 - 1}|P'|^{ s_0-1}|P|^{2D- s_0} }, \qquad (2\leq s<r),
\end{align}
for $N$ sufficiently large. In \eqref{eq:EI_{r,s}} the second term bounds the terms corresponding to $s_0-1=s'$ obtained by acting on $\partial_{ij}$ exactly $(s_1-1)$-times on $P$ and $s_2$-times on $\ol{P}$ but never on their derivatives.

To sum up, we showed that, for $1\leq s\leq r$, $\E \dt{I}_{r,s}$ is negligible and therefore $\E I_{r,s}, 1\leq s \leq r$ is negligible. This proves \eqref{eq:estimate I_2}.

\subsection{Estimate on \texorpdfstring{$I_{1,0}$}{I1,0}}\label{sec:subsec5}
Finally we only need to estimate $\E I_{1,0}$. We have 
\begin{align}
\E I_{1,0} =& \frac{1}{N} \sum_{i\neq j} \kappa_{ij}^{(2)}\E \qbb{ \pB{\partial_{ij} G_{ij}} QP^{D-1}\ol{P^D}}\nonumber \\
=&- \frac{1}{N} \E \qbb{ \sum_{i,j} \kappa_{ij}^{(2)} G_{ii}G_{jj}QP^{D-1}\ol{P^D}}\\
&+\frac{1}{N} \E \qbb{ \sum_{i} \kappa_{ii}^{(2)} G_{ii}^2QP^{D-1}\ol{P^D}}-\frac{1}{N} \E \qbb{ \sum_{i\neq j} \kappa_{ij}^{(2)} G_{ij}^2QP^{D-1}\ol{P^D}}\nonumber\\
=:& -\E I_{1,0}^{(2)} + \E I_{1,0}^{(1)} - \E I_{1,0}^{(0)}
\end{align}

The second and third term can be bounded by	$\Phi_\epsilon$ since
\begin{equation}\label{eq:EI_{1,0}^1}
|\E I_{1,0}^{(1)}|= \absBB{\frac{1}{N} \E \qbb{ \sum_{i} \kappa_{s}^{(2)} G_{ii}^2QP^{D-1}\ol{P^D}}} \leq \frac{N^\epsilon}{N} \E \qbb{P^{D-1}\ol{P^D}}\leq \Phi_\epsilon,
\end{equation}
\begin{equation}\label{eq:EI_{1,0}^0}
|\E I_{1,0}^{(0)}|=\absBB{\frac{1}{N} \E \qbb{ \sum_{i\neq j} \kappa_{ij}^{(2)} G_{ij}^2QP^{D-1}\ol{P^D}}} \leq N^\epsilon \E \qbb{\frac{\im m}{N\eta}\abs{P}^{2D-1} }\leq \Phi_\epsilon.
\end{equation}

Hence we only need to estimate 
\begin{align}\label{eq:EI_{1,0}^2}
&\E I_{1,0}^{(2)}=\frac{1}{N} \E \qbb{ \sum_{i,j} \kappa_{ij}^{(2)} G_{ii}G_{jj}QP^{D-1}\ol{P^D}} \nonumber\\
&=\frac{1}{N}\E \qbb{ \sum_{i,j} \kappa_d^{(2)}G_{ii}G_{jj}QP^{D-1}\ol{P^D}}+\frac{1}{N}\E\qbb{\sum_{i}\sum_{j\sim i} (\kappa_s^{(2)}-\kappa_d^{(2)})G_{ii}G_{jj}QP^{D-1}\ol{P^D}} \nonumber \\
&= \E \qbb{ N\kappa_d^{(2)}m^2 QP^{D-1}\ol{P^D} } + \frac{1}{N}\E\qbb{\sum_{i}\sum_{j\sim i} (\kappa_s^{(2)}-\kappa_d^{(2)})G_{ii}G_{jj}QP^{D-1}\ol{P^D}} \nonumber \\
&= \E \qbb{(1-\zeta )m^2 QP^{D-1}\ol{P^D} } + \frac{\zeta k}{N^2}\E\qbb{\sum_{i}\sum_{j\sim i} G_{ii}G_{jj}QP^{D-1}\ol{P^D}} \notag \\
&= \E \qbb{(1-\zeta )m^2 QP^{D-1}\ol{P^D} } +\frac{\zeta k}{N^2}\E\qbb{\sum_{i}\sum_{j\sim i} G_{ii}G_{jj}z\p{z+m+3\zeta m}P^{D-1}\ol{P^D}} \notag \\ &+\frac{\zeta k}{N^2}\E\qbb{\sum_{i}\sum_{j\sim i} G_{ii}G_{jj}\pb{(1-\zeta )m(z+m+3\zeta m)+\zeta (6\zeta m^2-1)}P^{D-1}\ol{P^D}}, 
\end{align}
where we recall that  
\begin{align*}
Q&=(z+m+\zeta m)^2+2\zeta^2m^2-\zeta \\
&=(z+m+3\zeta m)z+(1-\zeta )m(z+m+3\zeta m)+\zeta (6\zeta m^2-1).
\end{align*}

Using Lemma \ref{lemma:Stein}, we expand the second term of \eqref{eq:EI_{1,0}^2} as
\begin{align}\label{eq:expand J_{r,s}}
&\frac{\zeta k}{N^2}\E\qbb{\sum_{i}\sum_{j\sim i} zG_{ii}G_{jj}(z+m+3\zeta m)P^{D-1}\ol{P^D}} \nonumber \\
&= \frac{\zeta k}{N^2}\E\qbb{\sum_{i}\sum_{j\sim i} \pB{\sum_{k}H_{ik}G_{ki} -1 } G_{jj}(z+m+3\zeta m)P^{D-1}\ol{P^D}}\nonumber \\ 
&= \frac{\zeta k}{N^2}\E\qbb{\sum_{i}\sum_{j\sim i}\sum_{k\neq i}H_{ik}G_{ki} G_{jj}(z+m+3\zeta m)P^{D-1}\ol{P^D}} - \zeta \E\qbb{m(z+m+3\zeta m)P^{D-1}\ol{P^D} }\nonumber \\
&=  \sum_{r=1}^{l} \sum_{s=0}^{r} w_{J_{r,s}} \E J_{r,s} - \zeta \E\qbb{m(z+m+3\zeta m)P^{D-1}\ol{P^D} } + O(\frac{N^\epsilon}{q_t^l}),
\end{align}
where 
\begin{align}
& w_{J_{r,s}}= \frac{1}{(r-s)!s!}\\
&J_{r,s}=\frac{\zeta k}{N^2}\sum_{i}\sum_{j\sim i}\sum_{k\neq i}\kappa_{ik}^{(r+1)}\E \qbb{ \pB{\partial_{ik}^{r-s} G_{ik}G_{jj}} \pB{\partial_{ik}^{s}\pB{RP^{D-1}\ol{P^D}}}}.
\end{align}
Here, we abbreviate $R:=z+m+3\zeta m$.
Similar as estimating $I_{r,s}$, it can be shown that all terms of $J_{r.s}$ and the error term are negligible except $J_{1,0}$ and $J_{3,0}$ by using Lemma \ref{lemma:powercounting0}. We omit the details.

\subsubsection{Estimate on \texorpdfstring{$J_{1,0}$}{J1,0}}

\begin{align}
\E J_{1,0}&=\frac{\zeta k}{N^2}\sum_{i}\sum_{j\sim i}\sum_{k\neq i}\kappa_{ik}^{(2)}\E \qbb{ \pB{\partial_{ik} G_{ik}G_{jj}} {RP^{D-1}\ol{P^D}}} \nonumber \\
&=-\frac{\zeta k}{N^2}\sum_{i}\sum_{j\sim i}\sum_{k}\kappa_{d}^{(2)}\E \qbb{ \pB{G_{ii} G_{jj}G_{kk}+G_{ik}^2G_{kk}} {RP^{D-1}\ol{P^D}}} \nonumber\\
&-\frac{\zeta k}{N^2}\sum_{i}\sum_{j\sim i}\sum_{k\sim i}(\kappa_{s}^{(2)}-\kappa_{d}^{(2)})\E \qbb{ \pB{G_{ii} G_{jj}G_{kk}+G_{ik}^2G_{kk}} {RP^{D-1}\ol{P^D}}} \nonumber \\
& -2\frac{\zeta k}{N^2}\sum_{i}\sum_{j\sim i}\sum_{k\neq i}\kappa_{ik}^{(2)}\E \qbb{ \pB{G_{ij} G_{jk}G_{ki}} {RP^{D-1}\ol{P^D}}}.
\end{align}
Then we can show that 
\begin{align}\label{eq:EJ_{1,0}}
\E J_{1,0}=&-\frac{\zeta k}{N^2}\sum_{i}\sum_{j\sim i}\sum_{k}\kappa_{d}^{(2)}\E \qbb{ \pB{G_{ii} G_{jj}G_{kk}} {RP^{D-1}\ol{P^D}}} \nonumber\\
&-\frac{\zeta k}{N^2}\sum_{i}\sum_{j\sim i}\sum_{k\sim i}(\kappa_{s}^{(2)}-\kappa_{d}^{(2)})\E \qbb{ \pB{G_{ii} G_{jj}G_{kk}} {RP^{D-1}\ol{P^D}}} +O(\Phi_\epsilon),
\end{align}
since other terms are all negligible, similar as proving \eqref{eq:EI_{1,0}^0} and \eqref{eq:EI_{1,0}^1}.
The first term can be computed by
\begin{align}
-\frac{\zeta k}{N^2}\sum_{i}\sum_{j\sim i}\sum_{k\neq i}\kappa_{d}^{(2)}\E \qbb{ \pB{G_{ii} G_{jj}G_{kk}} {RP^{D-1}\ol{P^D}}}&=	-\frac{\zeta k}{N}\sum_{i}\sum_{j\sim i}\kappa_{d}^{(2)}\E \qbb{ mG_{ii} G_{jj} {RP^{D-1}\ol{P^D}}}\notag \\
&=-\frac{\zeta (1-\zeta )k}{N^2}\sum_{i}\sum_{j\sim i}\E \qbb{ mG_{ii} G_{jj} {RP^{D-1}\ol{P^D}}}
\end{align}
To estimate the second term of \eqref{eq:EJ_{1,0}}, we abbreviate 
\[
\hat{J}:=\frac{\zeta k}{N^2}\sum_{i}\sum_{j\sim i}\sum_{k\sim i}(\kappa_{s}^{(2)}-\kappa_{d}^{(2)})\E \qbb{ \pB{G_{ii} G_{jj}G_{kk}} {P^{D-1}\ol{P^D}}}.
\]
Using Lemma \ref{lemma:Stein}, for arbitrary $\ell' \in \N$ we have the cumulant expansion
\begin{align}\label{eq:zEhat J}
z\E \hat{J}&:=\frac{\zeta k}{N^2}\sum_{i}\sum_{j\sim i}\sum_{k\sim i}(\kappa_{s}^{(2)}-\kappa_{d}^{(2)})\E \qbb{ \pB{zG_{ii} G_{jj}G_{kk}} {P^{D-1}\ol{P^D}}} \notag \\
&=\frac{\zeta ^2K^2}{N^3}\sum_{i}\sum_{j\sim i}\sum_{k\sim i}\E \qbb{ \pB{zG_{ii} G_{jj}G_{kk}} {P^{D-1}\ol{P^D}}}\notag\\
&=\frac{\zeta ^2K^2}{N^3}\E\qbb{\sum_{i}\sum_{j\sim i}\sum_{k\sim i} \pB{\sum_{u}H_{iu}G_{ui} -1 } G_{jj}G_{kk}P^{D-1}\ol{P^D}}\nonumber \\ 
&= \frac{\zeta ^2K^2}{N^3}\E\qbb{\sum_{i}\sum_{j\sim i}\sum_{k\sim i}\sum_{u\neq i}H_{iu}G_{ui} G_{jj}G_{kk}P^{D-1}\ol{P^D}} - \frac{\zeta ^2K}{N^2}\E\qbb{ \sum_{i}\sum_{j\sim i} G_{ii}G_{jj}P^{D-1}\ol{P^D} }\nonumber \\
&=  \sum_{r=1}^{l} \sum_{s=0}^{r} w_{\hat{J}_{r,s}} \E \hat{J}_{r,s} - \frac{\zeta ^2K}{N^2}\E\qbb{ \sum_{i}\sum_{j\sim i} G_{ii}G_{jj}P^{D-1}\ol{P^D} } + \E \Omega_{l}(\hat{J}),
\end{align}

where 
\begin{align}
& w_{\hat{J}_{r,s}}= \frac{1}{(r-s)!s!}\\
&\hat{J}_{r,s}=\frac{\zeta ^2K^2}{N^3}\sum_{i}\sum_{j\sim i}\sum_{k\sim i}\sum_{u\neq i}\kappa_{iu}^{(r+1)}\E \qbb{ \pB{\partial_{iu}^{r-s} G_{iu}G_{jj}G_{kk}} \pB{\partial_{ik}^{s}\pB{P^{D-1}\ol{P^D}}}},
\end{align}
By Corollary \ref{coro:bound error}, the error term  $\E\Omega_{l}(\hat{J})$ is negligible with choosing $l\geq 8D$.

By a similar argument as estimating $I_{r,s}$, we can easily check that the non-negligible terms in the cumulant expansion of $z\E \hat{J}$ are
\begin{align}
&\wt J_{1,0}:= -\frac{\zeta ^2K^2}{N^3}\sum_{i}\sum_{j\sim i}\sum_{k\sim i}\sum_{u\neq i}\kappa_{iu}^{(2)}\E \qB{G_{ii}G_{jj}G_{kk}G_{uu}P^{D-1}\ol{P^D}},\\
&\wt J_{3,0} :=-\frac{\zeta ^2K^2}{N^3}\sum_{i}\sum_{j\sim i}\sum_{k\sim i}\sum_{u\neq i}\kappa_{iu}^{(4)}\E \qB{G_{ii}^2G_{jj}G_{kk}G_{uu}^2P^{D-1}\ol{P^D}},
\end{align}
where $\wt J_{1,0}$ and $\wt J_{3,0}$ come from $\hat{J}_{1,0}$ ad $\hat{J}_{3,0}$, respectively. 

We rewrite $\wt J_{1,0}$ as
\begin{align}\label{eq:wt J_{1,0}}
\wt J_{1,0} = &-\frac{\zeta ^2K^2}{N^3}\sum_{i}\sum_{j\sim i}\sum_{k\sim i}\sum_{u}\kappa_{d}^{(2)}\E \qB{G_{ii}G_{jj}G_{kk}G_{uu}P^{D-1}\ol{P^D}}\notag \\
&\quad-  \frac{\zeta ^2K^2}{N^3}\sum_{i}\sum_{j\sim i}\sum_{k\sim i}\sum_{u\sim i}(\kappa_{s}^{(2)}-\kappa_{d}^{(2)})\E \qB{G_{ii}G_{jj}G_{kk}G_{uu}P^{D-1}\ol{P^D}}\notag\\
=& -\frac{\zeta ^2(1-\zeta )K^2}{N^3}\sum_{i}\sum_{j\sim i}\sum_{k\sim i}\E \qB{mG_{ii}G_{jj}G_{kk}P^{D-1}\ol{P^D}}\notag\\
&\quad- \frac{\zeta ^2K^2}{N^3} \frac{K\zeta }{N}\sum_{i}\sum_{j\sim i}\sum_{k\sim i}\sum_{u\sim i}\E \qB{\pb{\psi^4 + 4m G_{ii}G_{jj}G_{kk} -6m^2  G_{ii}G_{jj}+4m^3 G_{ii}   } P^{D-1}\ol{P^D}} \notag \\
=& -3\zeta ^3\E \qB{ m^4 P^{D-1}\ol{P^D}} + 6\frac{\zeta ^3K}{N^2}\sum_{i}\sum_{j\sim i}\E\qB{m^2G_{ii}G_{jj}P^{D-1}\ol{P^D}} \notag \\
&\quad-\pa{\frac{3\zeta ^3K^2}{N^3}+\frac{\zeta ^2K^2}{N^3}}\E \qB{\sum_{i}\sum_{j\sim i}\sum_{k\sim i} mG_{ii}G_{jj}G_{kk} P^{D-1}\ol{P^D} } + O(\Phi_\epsilon),
\end{align}
where we used the expansion of $(G_{ii}-m)(G_{kk}-m)(G_{kk}-m)(G_{uu}-m)$ and $i\sim j\sim k\sim u$ as in \eqref{eq:G_ii^2G_jjG_kk^2}.

Similarly, using the expansion of  $(G_{ii}^2-m^2)(G_{kk}-m)(G_{kk}-m)(G_{uu}^2-m^2)$, we obtain
\begin{align}\label{eq:wt J_{3,0}}
\wt J_{3,0}&=-\frac{\zeta ^2K^2}{N^3}\sum_{i}\sum_{j\sim i}\sum_{k\sim i}\sum_{u\neq i}\kappa_{iu}^{(4)}\E \qB{G_{ii}^2G_{jj}G_{kk}G_{uu}^2P^{D-1}\ol{P^D}}\notag \\
&= -\frac{\zeta ^2K^2}{N^3} \sum_{i\sim j \sim k}\sum_{u}\kappa_{d}^{(4)}\E \qbb{G_{ii}^2 G_{jj}G_{kk}G_{uu}^2P^{D-1}\ol{P^D}}\notag\\
&\qquad -\frac{\zeta ^2K^2}{N^3} \sum_{i\sim j \sim k \sim u}(\kappa_{s}^{(4)}-\kappa_{d}^{(4)})\E \qbb{G_{ii}^2 G_{jj}G_{kk}G_{uu}^2P^{D-1}\ol{P^D}} \nonumber\\
&=-\zeta ^2N\kappa_{d}^{(4)}\E \qbb{m^6P^{D-1}\ol{P^D}}-\frac{\zeta ^2K^2}{N^3}(\kappa_{s}^{(4)}-\kappa_{d}^{(4)})\E \qbb{ \frac{N^4}{K^3}m^6P^{D-1}\ol{P^D}} + O(N^\epsilon q^{-2}\psi^2)\nonumber\\
&= -\zeta ^2\pa{\frac{N}{k} (\kappa_s^{(4)} - \kappa_d^{(4)} ) + N \kappa_d^{(4)}}\E\qbb{m^6P^{D-1}\ol{P^D}}+ O(\Phi_\epsilon) \nonumber\\
&= -\zeta ^2\frac{\xi^{(4)}}{q_t^2}\E\qbb{m^6P^{D-1}\ol{P^D}}+ O(\Phi_\epsilon).
\end{align}

To sum up, from \eqref{eq:zEhat J}, \eqref{eq:wt J_{1,0}} and \eqref{eq:wt J_{3,0}} we get
\begin{align}
\frac{\zeta ^2 K^2}{N^3}\sum_{i}&\sum_{j\sim i}\sum_{k\sim i} \E \qB{zG_{ii}G_{jj}G_{kk}P^{D-1}\ol{P^D}} \notag\\
&= -3\zeta ^3\E \qB{ m^4 P^{D-1}\ol{P^D}} + 6\frac{\zeta ^3K}{N^2}\sum_{i}\sum_{j\sim i}\E\qB{m^2G_{ii}G_{jj}P^{D-1}\ol{P^D}} \notag \\
&-\pa{\frac{3\zeta ^3K^2}{N^3}+\frac{\zeta ^2K^2}{N^3}}\E \qB{\sum_{i}\sum_{j\sim i}\sum_{k\sim i} mG_{ii}G_{jj}G_{kk} P^{D-1}\ol{P^D} } \\
& -\zeta ^2\frac{\xi^{(4)}}{q_t^2}\E\qbb{m^6P^{D-1}\ol{P^D}} - \frac{\zeta ^2K}{N^2}\E\qbb{ \sum_{i}\sum_{j\sim i} G_{ii}G_{jj}P^{D-1}\ol{P^D} } + O(\Phi_\epsilon). \notag 
\end{align}

Collecting the terms belonging to the same summation, we finally get
\begin{align}
\E J_{1,0}&=-\frac{\zeta k}{N^2}\sum_{i}\sum_{j\sim i}\sum_{k}\kappa_{d}^{(2)}\E \qbb{ \pB{G_{ii} G_{jj}G_{kk}} {(z+m+3\zeta m)P^{D-1}\ol{P^D}}} \nonumber\\
&\qquad -\frac{\zeta k}{N^2}\sum_{i}\sum_{j\sim i}\sum_{k\sim i}(\kappa_{s}^{(2)}-\kappa_{d}^{(2)})\E \qbb{ \pB{G_{ii} G_{jj}G_{kk}} {(z+m+3\zeta m)P^{D-1}\ol{P^D}}} +O(\Phi_\epsilon) \notag \\
&=-\frac{\zeta (1-\zeta )k}{N^2}\sum_{i}\sum_{j\sim i}\E \qbb{ mG_{ii} G_{jj} {(z+m+3\zeta m)P^{D-1}\ol{P^D}}} \notag \\
&\qquad-\frac{\zeta ^2K}{N^2}\sum_{i\sim j}\E\qB{(6\zeta m^2-1)G_{ii}G_{jj}P^{D-1}\ol{P^D}}\notag \\
&\qquad+3\zeta ^3\E \qB{ m^4 P^{D-1}\ol{P^D}} +\zeta ^2\frac{\xi^{(4)}}{q_t^2}\E\qbb{m^6P^{D-1}\ol{P^D}}+O(\Phi_\epsilon).
\end{align}

\subsubsection{Estimate on \texorpdfstring{$J_{3,0}$}{J3,0}}
Recall that 
\[ 
w_{J_{3,0}}J_{3,0}=\frac{\zeta k}{N^2}\sum_{i}\sum_{j\sim i}\sum_{k\neq i}\kappa_{ik}^{(4)}\E \qbb{ \pB{\partial_{ik}^{3} G_{ik}G_{jj}} \pB{RP^{D-1}\ol{P^D}}}.
\]
Note that the terms contained in $J_{3,0}$ with more than two off-diagonal Green function entries are negligible by using Lemma \ref{lemma:powercounting0}.
The only non-negligible terms are $J_{3,0}^{(0)}$ and $J_{3,0}^{(1)}$ which contain no and one off-diagonal Green function entry respectively.
By simple calculation, we get
\begin{align}
J_{3,0}^{(0)}= -\frac{\zeta k}{N^2} \sum_{i}\sum_{j\sim i}\sum_{k}\kappa_{ik}^{(4)}G_{ii}^2 G_{jj}G_{kk}^2RP^{D-1}\ol{P^D},\\
J_{3,0}^{(1)}= -\frac{\zeta k}{N^2} \sum_{i}\sum_{j\sim i}\sum_{k}\kappa_{ik}^{(4)}G_{ii} G_{jj}G_{kk}^2G_{ij}RP^{D-1}\ol{P^D},
\end{align}
To estimate 	$J_{3,0}^{(0)}$ we expand 
\begin{align}
(G_{ii}^2-m^2)(G_{jj}-m)(G_{kk}^2-m^2)=-m^5 &+G_{jj}m^4 +(G_{ii}^2+G_{kk}^2)m^3\nonumber \\
&-G_{jj}(G_{ii}^2+G_{kk}^2)m^2 - G_{ii}^2G_{kk}^2m + G_{ii}^2G_{jj}G_{kk}^2,
\end{align}
and
\begin{align*}
G_{jj}G_{ii}^2m^2 = O(N^\epsilon\psi^2) + m^3 G_{ii}^2+m^4G_{jj}-m^5, \\ 	
G_{jj}G_{kk}^2m^2 = O(N^\epsilon\psi^2) + m^3 G_{kk}^2+m^4G_{jj}-m^5,\\ 
G_{ii}^2G_{kk}^2m = O(N^\epsilon\psi^2) + m^3 G_{ii}^2+m^3G_{kk}^2-m^5.
\end{align*}
Therefore we obtain
\begin{align}\label{eq:G_ii^2G_jjG_kk^2}
G_{ii}^2 G_{jj}G_{kk}^2 &= -2m^5 +m^4 G_{jj} +(G_{ii}^2+G_{kk}^2)m^3 +O(N^\epsilon\psi^2) \notag \\
&= -2m^5 +m^4 G_{jj} + (2G_{ii}m+2G_{kk}m-2m^2)m^3+O(N^\epsilon\psi^2) \notag\\
&= -4m^5 +m^4 (G_{jj} + 2G_{ii}+2G_{kk})+O(N^\epsilon\psi^2).
\end{align}

Hence we get 
\begin{align}
\E J_{3,0}^{(0)}&=\frac{\zeta k}{N^2} \sum_{i}\sum_{j\sim i}\sum_{k}\kappa_{ik}^{(4)}\E \qbb{G_{ii}^2 G_{jj}G_{kk}^2RP^{D-1}\ol{P^D}} \nonumber\\
&= -\frac{\zeta k}{N^2} \sum_{i}\sum_{j\sim i}\sum_{k}\kappa_{d}^{(4)}\E \qbb{G_{ii}^2 G_{jj}G_{kk}^2RP^{D-1}\ol{P^D}}-\frac{\zeta k}{N^2} \sum_{i}\sum_{j\sim i}\sum_{k\sim i}(\kappa_{s}^{(4)}-\kappa_{d}^{(4)})\E \qbb{G_{ii}^2 G_{jj}G_{kk}^2RP^{D-1}\ol{P^D}} \nonumber\\
&=-\zeta N\kappa_{d}^{(4)}\E \qbb{m^5RP^{D-1}\ol{P^D}}-\frac{\zeta k}{N^2}(\kappa_{s}^{(4)}-\kappa_{d}^{(4)})\E \qbb{ \frac{N^3}{k^2}m^5RP^{D-1}\ol{P^D}} + O(N^\epsilon q^{-2}\psi^2)\nonumber\\
&= -\zeta \pa{\frac{N}{k} (\kappa_s^{(4)} - \kappa_d^{(4)} ) + N \kappa_d^{(4)}}\E\qbb{Rm^5P^{D-1}\ol{P^D}}+ O(\Phi_\epsilon) \nonumber\\
&= -\zeta q_t^{(-2)}\xi^{(4)}\E\qbb{m^5(z+m+3\zeta m)P^{D-1}\ol{P^D}}+ O(\Phi_\epsilon).
\end{align}

Now we show that $\E J_{3,0}^{(1)}$ is also negligible. Using $|G_{ii}|, |G_{jj}|, |G_{kk}| \prec 1 $ and Lemma \ref{lemma:powercounting0}, we get 
\begin{equation}
|\E J_{3,0}^{(1)}| \leq \frac{N^\epsilon}{q_t^2}\E \qB{\pB{\frac{\im m}{N\eta}}^{1/2}|R||P|^{2D-1}  } \leq N^\epsilon \E \qB{\pB{\frac{1}{q_t^4} + \frac{\im m}{N\eta}   }|P|^{2D-1}  } \leq \Phi_\epsilon.
\end{equation}
To sum up, we conclude that 
\begin{equation}
w_{J_{3,0}}\E J_{3,0} = -\zeta q_t^{(-2)}\xi^{(4)}\E\qbb{m^5(z+m+3\zeta m)P^{D-1}\ol{P^D}}+ O(\Phi_\epsilon).
\end{equation}

Now we go back to \eqref{eq:expand J_{r,s}}. In conclusion, 
\begin{align}
&\frac{\zeta K}{N^2}\E\qbb{\sum_{i}\sum_{j\sim i} zG_{ii}G_{jj}(z+m+3\zeta m)P^{D-1}\ol{P^D}} \\
&=	w_{J_{1,0}} \E J_{1,0}+w_{J_{3,0}} \E J_{3,0} - \zeta \E\qbb{m(z+m+3\zeta m)P^{D-1}\ol{P^D} } +O(\Phi_\epsilon) \notag\\
&=-\frac{\zeta K}{N^2}\sum_{i}\sum_{j\sim i}\E \qbb{ \pb{m(1-\zeta )(z+m+3\zeta m) + \zeta (6\zeta m^2-1)}G_{ii} G_{jj}P^{D-1}\ol{P^D}}\notag \\
&\qquad+3\zeta ^3\E \qB{ m^4 P^{D-1}\ol{P^D}} +\zeta ^2\frac{\xi^{(4)}}{q_t^2}\E\qbb{m^6P^{D-1}\ol{P^D}}\notag\\
&\qquad-\zeta q_t^{-2}\xi^{(4)}\E\qbb{m^5(z+m+3\zeta m)P^{D-1}\ol{P^D}}- \zeta \E\qbb{m(z+m+3\zeta m)P^{D-1}\ol{P^D} }+O(\Phi_\epsilon).
\end{align}
Thus, we have 
\begin{align}
\E I_{1,0} &= - \E I_{1,0}^{(2)} + O(\Phi_\epsilon) \notag\\
&= -\E \qbb{(1-\zeta )m^2 QP^{D-1}\ol{P^D} } -\frac{\zeta k}{N^2}\E\qbb{\sum_{i}\sum_{j\sim i} G_{ii}G_{jj}z\p{z+m+3\zeta m}P^{D-1}\ol{P^D}} \notag \\ &-\frac{\zeta k}{N^2}\E\qbb{\sum_{i}\sum_{j\sim i} G_{ii}G_{jj}\pb{(1-\zeta )m(z+m+3\zeta m)+\zeta (6\zeta m^2-1)}P^{D-1}\ol{P^D}} +O(\Phi_\epsilon)\\
&=-\E\qbb{ \pB{(1-\zeta )m^2Q+3\zeta ^3m^4+\zeta ^2 \frac{\xi^{(4)}}{q_t^{2}} m^6 -\zeta m(z+m+3\zeta m)(\frac{\xi^{(4)}}{q_t^{2}}m^4+1) }P^{D-1}\ol{P^D}}+O(\Phi_\epsilon)\notag,
\end{align}
and this concludes the proof of Lemma \ref{lemma:truncated}.
\section{Proof of Lemma \ref{lem:Hnorm}}\label{app:Hnorm}
In this appendix we provide the proof of Lemma \ref{lem:Hnorm}. First we consider the upper bound on the largest eigenvalue $\lambda_1^{H_t}$ of $H_t$.	
\begin{lemma}\label{lem:lambda_1^H_t}
	Let $H_0$ satisfy Assumption \ref{assumption} with $\phi >0$. Let $L_t$ be deterministic number defined in Lemma \ref{lemma:rho_t}. Then,
	\begin{equation}\label{eq:lambda_1^H_t}
	\lambda_1^{H_t}-L_t \prec \frac{1}{q_t^4} +\frac{1}{N^{2/3}},
	\end{equation}
	uniformly in $t \in [0,6\log N]$.
\end{lemma}
\begin{proof}
	Fix $t \in [0,6\log N]$. Recall the $z$-dependent deterministic parameters
	\begin{align}
	\alpha_1(z) := \im \widetilde{m}_t(z), \quad \alpha_2(z):=P'(\widetilde{m}_t(z)), \quad \beta:= \frac{1}{N\eta} + \frac{1}{q_t^2}.
	\end{align}
	For brevity, we mostly omit the $z$-dependence. We further introduce the $z$-dependent quantity
	\begin{equation}
	\wt \beta := \pB{ \frac{1}{q_t^4}+\frac{1}{N^{2/3}}}^{1/2}.
	\end{equation}
	
	Fix a small $\epsilon > 0$ and define 
	\begin{equation}
	\mathcal{D_\epsilon} := \hbb{z=E+\ii \eta : N^{4\epsilon} \wt \beta^2 \leq \kappa_t \leq q_t^{-1/3}, \eta = \frac{N^\epsilon}{N\sqrt{\kappa_t}}   },
	\end{equation}
	where $\kappa_t=\kappa_t(E)=E-L_t$. Note that for any sufficiently small $\epsilon>0$, on $\mathcal{D_\epsilon}$, we have
	\[
	N^{-1+\epsilon} \ll \eta \leq \frac{N^{-\epsilon}}{N\wt \beta}, \qquad \kappa \geq N^{5\epsilon}\eta.
	\]
	From the last inequality, we get $N^\epsilon \wt \beta \leq (N\eta)^{-1}$, thus $N^\epsilon q_t^{-2} \leq C(N\eta)^{-1}$ so that $q_t^{-2}$ is negligible when compared to $(N\eta)^{-1}$ and $\beta$ on $\mathcal{D_\epsilon}$. Furthermore, note that 
	\begin{equation}\label{eq:alpha2 alpha1}
	\begin{split}
	&\abs{\alpha_2} = \abs{P' (\wt m_t)} \sim \sqrt{\kappa_t + \eta} \sim \sqrt{\kappa_t} = \frac{N^\epsilon}{N\eta} \sim N^\epsilon\beta,\\
	&\alpha_1 = \im \wt m_t \sim \frac{\eta}{\sqrt{\kappa_t + \eta}}\sim  \frac{\eta}{\sqrt{\kappa_t }} \leq N^{-5\epsilon}\sqrt{\kappa_t} \sim N^{-5\epsilon}|\alpha_2| \sim N^{-4\epsilon}\beta.
	\end{split}
	\end{equation}
	In particular we have $\alpha_1 \ll \alpha_2$ on $\mathcal{D_\epsilon}$.
	
	We next claim that 
	\[
	\Lambda_{t} := |m_t - \wt m_t| \ll \frac{1}{N\eta}
	\]
	with high probability on the domain $\mathcal{D_\epsilon}$.
	Since $\mathcal{D_\epsilon} \subset \mathcal{E}$, from Proposition \ref{prop:locallaw}, we find that $\Lambda_{t}\leq N^{\epsilon'}$ for any $\epsilon' >0$ with high probability.
	Fix $0<\epsilon' < \epsilon/7$. We obtain from \eqref{eq:recursive1} that, 
	\begin{align*}
	\E [|P(m_t)|^{2D}] &\leq CN^{(4D-1)\epsilon'}\mathbb{E}[\beta^{2D}(\alpha_1 +\Lambda_t)^D(|\alpha_2|+C_1\Lambda_t)^D]+\frac{N^{(2D+1)\epsilon'}}{D}q_t^{-8D} +\frac{N^{-(D/4-1)\epsilon'}}{D}q_t^{-2D}\mathbb{E}[\Lambda_t^{4D}]\\
	&\leq C^{2D}N^{6D\epsilon'}\beta^{4D} +\frac{N^{(2D+1)\epsilon'}}{D}q_t^{-8D} + \frac{N^{4D\epsilon'}}{D}q_t^{-2}\beta^{4D}\\
	&\leq  C^{2D}N^{6D\epsilon'}\beta^{4D},
	\end{align*}
	for $N$ sufficiently large, where we used the fact that $\Lambda_{t} \leq N^{\epsilon'\beta}\ll N^\epsilon\beta$ with high probability and $\alpha_1 \ll \alpha_2$, $|\alpha_2| \leq CN^\epsilon\beta$ on $\mathcal{D_\epsilon}$  by \eqref{eq:alpha2 alpha1}. Using $(2D)$-th order Markov inequality and a lattice argument with a union bound, we obtain
	\[
	|P(m_t)| \leq CN^{3\epsilon'}\beta^2,
	\]
	with high probability and uniformly on $\mathcal{D_\epsilon}$. Then from the Taylor expansion of $P(m_t)$ around $\wt m_t$ in \eqref{eq:3rdTaylor}, we obtain that 
	\begin{equation}\label{eq:alpha2 Lamda_t}
	|\alpha_2|\Lambda_{t} \leq C'\Lambda_{t}^2 + CN^{3\epsilon'}\beta^2,
	\end{equation}
	with high probability, uniformly on $\mathcal{D_\epsilon}$  with some constant $C'$. Here we also used that $\Lambda_{t} \ll 1$ on $\mathcal{D_\epsilon}$ with high probability.
	
	Since  $\Lambda_{t} \leq N^{\epsilon'\beta}\leq CN^{\epsilon'-\epsilon} |\alpha_2|$ on $\mathcal{D_\epsilon}$ with high probability, we have $|\alpha_2|\Lambda_{t} \geq CN^{\epsilon'-\epsilon}\Lambda_{t}^2 \gg C'\Lambda_{t}^2 $. Thus the first term on the right side of \eqref{eq:alpha2 Lamda_t} can be absorbed into the left side and we conclude that 
	\[
	\Lambda_{t} \leq CN^{3\epsilon'}\frac{\beta}{|\alpha_2|}\beta \leq CN^{3\epsilon'-\epsilon}\beta,
	\]
	hold with high probability, uniformly on $\mathcal{D_\epsilon}$. Hence with $0<\epsilon' < \epsilon/7$, we get that 
	\[
	\Lambda_{t} \leq N^{-\epsilon/2}\beta\leq 2\frac{N^{-\epsilon/2}}{N\eta},
	\]
	uniformly on $\mathcal{D_\epsilon}$ with high probability. Thus we can claim that $\Lambda_{t} \ll (N\eta)^{-1}$ on $\mathcal{D_\epsilon}$ with high probability. Combining with \eqref{eq:alpha2 alpha1}, this also shows that
	\begin{equation}
	\im m_t \leq \im \wt m_t + \Lambda_{t} = \alpha_1 + \Lambda_{t} \ll \frac{1}{N\eta}.
	\end{equation}
	on $\mathcal{D_\epsilon}$ with high probability.
	
	Now we prove \eqref{eq:lambda_1^H_t}. If $\lambda_1^{H_t}\in [E-\eta, E+\eta]$ for some $E \in [L_t - N^\epsilon (q_t^{-4} +N^{-2/3}), L_t + q^{-1/3} ]$ with $z \in \mathcal{D_\epsilon}$,
	\begin{equation}
	\im m_t(z) \geq \frac{1}{N} \im \frac{1}{\lambda_1^{H_t}-E-\ii\eta}=\frac{1}{N}\frac{\eta}{(\lambda_1^{H_t}-E)^2+\eta^2}\geq \frac{1}{5N\eta},
	\end{equation}
	which contradicts the high probability bound $\im m_t \ll (N\eta)^{-1}$. Since the size of each interval $[E-\eta,E+\eta]$ is at least $N^{-1+\epsilon}q_t^{1/6}$, we can conclude by considering $O(N)$ such intervals that $\lambda_1\notin [L_t - N^\epsilon (q_t^{-4} +N^{-2/3}), L_t + q^{-1/3} ]$ with high probability. From Lemma \ref{thm:weak matrix norm} we find that $\lambda_1^{H_t} -L_t \prec q_t^{-1/3}$ with high probability. Therefore we conclude that \eqref{eq:lambda_1^H_t} holds for fixed $t \in [0,6\log N]$. We then get \eqref{eq:lambda_1^H_t} uniformly in $t \in [0,6\log N]$ using a lattice argument and the continuity of the Dyson matrix flow.
\end{proof}
\begin{proof}[Proof of Lemma \ref{lem:Hnorm} and Theorem {\ref{thm:matrix norm}}]
	Fix $t \in [0, 6\log N]$. Consider the largest eigenvalue $\lambda_1^{H_t}$. In Lemma \ref{lem:lambda_1^H_t}, we already showed that $(L_t - \lambda_1^{H_t})_-\prec q_t^{-4}+N^{-2/3}$. Therefore it is enough to consider $(L_t - \lambda_1^{H_t})_+$. By Lemma \ref{lemma:rho_t} there is $c>0$ such that  $c(L_t - \lambda_1^{H_t})_+ \leq n_{\wt \rho_t }(\lambda_1^{H_t},L_t)$. Hence by obvious generalization of Corollary \ref{coro:local density} to $H_t$, we obtain the estimate 
	\[
	(L_t - \lambda_1^{H_t})_+ \prec \frac{(L_t - \lambda_1^{H_t})_+}{q_t^2} +\frac{1}{N},
	\] 
	so that $(L_t - \lambda_1^{H_t})_+ \prec q_t^{-4}+N^{-2/3}$. Thus $\abs{\lambda_1^{H_t}-L_t} \prec q_t^{-4}+N^{-2/3}$. In same way, we can show the estimate $\abs{\lambda_N^{H_t}-L_t} \prec q_t^{-4}+N^{-2/3}$ for the smallest eigenvalue $\lambda_N^{H_t}$. Using the continuity of the Dyson matrix flow, we obtain \eqref{eq:norm H_t} uniformly in $t \in [0,6\log N]$.
\end{proof}

\section{Proof of Lemma \ref{lemma:A_r}}\label{app:TW}
In this appendix we prove Lemma \ref{lemma:A_r}. We begin by considering the case $r\geq 5$. In this case, we can see that $A_r = O(N^{\frac{2}{3}-\epsilon'})$, since it contains at least two off-diagonal entries in $\partial_{jk}^{r}(F'(X)G_{ij}G_{k i})$ and $|A_r|$ is bounded~by
\begin{align*}
N^3 N^{-1} q_t^{-4} N^{-2/3+2\epsilon} \ll N^{2/3-\epsilon'} 
\end{align*}
which can be checked by a simple power counting. Thus we only need to consider the cases $r=2,3,4$.

\subsection{Proof of Lemma \ref{lemma:A_r} for \texorpdfstring{$r=2$}{r=2}} Note that
\begin{align}
|A_{2}|&=\Big|\frac{e^{-t}q_t^{-1}}{N}\sum_i\sum_{j\neq k} \mathbb{E}[s_{jk}^{(3)}\partial_{jk}^2(F'(X)G_{ij}G_{k i}) ]\Big|\\ \notag 
&=\Big|\frac{e^{-t}q_t^{-1}}{N}\sum_{i,j,k} \mathbb{E}[s_{d}^{(3)}\partial_{jk}^2(F'(X)G_{ij}G_{k i})]+\frac{e^{-t}q_t^{-1}}{N}\sum_i\sum_{j\sim k} \mathbb{E}[(s_{s}^{(3)}-s_{d}^{(3)})\partial_{jk}^2(F'(X)G_{ij}G_{k i})] \Big|\\ \notag 
&\leq\frac{e^{-t}q_t^{-1}}{N}\sum_{i,j,k}\frac{Cq_t^{-1}}{N}\mathbb{E}[|\partial_{jk}^2(F'(X)G_{ij}G_{k i}) |],
\end{align}
for some constant $C$. Also, observe that
\begin{align}
\partial_{jk}^2(F'(X)G_{ij}G_{k i}) = F'(X)\partial_{jk}^2(G_{ij}G_{ki}) + 2\partial_{jk}F'(X)\partial_{jk}(G_{ij}G_{k i}) + (\partial^2_{jk}F'(X))G_{ij}G_{k i}.
\end{align}

We first consider the expansion of $\partial_{jk}^2(G_{ij}G_{ki})$. We can estimate the terms with four off diagonal Green function entries. For example,

\begin{align}\label{eq:4offdiag1}
\sum_{i,j,k}\mathbb{E}[|F'(X)G_{ij}G_{kj}G_{k j}G_{k i}|]\leq N^{C\epsilon}\sum_{i,j,k}|G_{ij}G_{k j}G_{k j}G_{k i}|\leq N^{C\epsilon} \Big(\frac{\im m}{N\eta_0}\Big)^2 \leq N^{-4/3+C\epsilon}, 
\end{align}
where we used Lemma \ref{lemma:powercounting0}. Thus, for sufficiently small $\epsilon$ and $\epsilon'$, we get
\begin{align} \label{eq:4offdiag2}
\frac{\mathrm{e}^{-t}q_t^{-1}}{N}\sum_{i,j,k}\mathbb{E}[|F'(X)G_{ij}G_{k j}G_{k j}G_{ki}|]\ll N^{2/3-\epsilon'}.
\end{align}

For the terms with three off-diagonal Green function entries, the bound we get from Lemma \ref{lemma:powercounting0} is
\begin{align*}
q_t^{-1}N^{-1} N^3 N^{C\epsilon} \Big(\frac{\im m}{N\eta_0 }\Big)^{3/2}\sim q_t^{-1}N^{1+C\epsilon},
\end{align*} 	
which is not sufficient. To gain an additional factor of $q_t^{-1}$, which makes the above bound $q_t^{-2}N^{1+C\epsilon}\ll N^{2/3-\epsilon'}$, we use Lemma \ref{lemma:Stein} to expand in an unmatched index. For example, such a term is of the form
\begin{align*}
G_{ij}G_{kj}G_{kk}G_{ji}
\end{align*}
and we focus on the unmatched index $\alpha$ in $G_{k j}$. We get
\begin{align}
\frac{q_t^{-1}}{N}\sum_{i,j,k}\mathbb{E}&[F'(X)G_{ij}G_{kj}G_{kk} G_{ji}] = \frac{q_t^{-1}}{N}\sum_{i,j,k,n} \mathbb{E}[F'(X)G_{ij}H_{kn}G_{nj}G_{kk}G_{ji}] \notag \\
&=\frac{q_t^{-1}}{N}\sum_{r'=1}^{\ell}\frac{1}{r'!}\sum_{i,j,k,n}\mathbb{E}[\kappa_{kn}^{(r'+1)}\partial_{kn}^{r'}(F'(X)G_{ij}G_{nj}G_{kk}G_{ji})]+O(N^{2/3-\epsilon'}),
\end{align}
for $\ell=10$.

For $r'=1$, we need to consider $\partial_{kn}(F'(X)G_{ij}G_{nj}G_{kk}G_{ji})$. When $\partial_{kn}$ acts on $F'(X)$ it creates a fresh summation index $n$, and we get a term
\begin{align}
\frac{q_t^{-1}}{N^2}&\sum_{i,j,k,n}\mathbb{E}[(\partial_{kn}(F'(X))G_{ij}G_{nj}G_{kk}G_{ji})] \notag \\
&=-\frac{2q_t^{-1}}{N^2}\int_{E_1}^{E_2}\sum_{i,j,k,n,m}\mathbb{E}[G_{ij}G_{nj}G_{kk}G_{ji}F''(X)\im(G_{mn}(y+L+\ii\eta_0)G_{km}(y+L+\ii\eta_0))] \textrm{d}y \notag \\
&=-\frac{2q_t^{-1}}{N^2}\int_{E_1}^{E_2}\sum_{i,j,k,n,m}\mathbb{E}[G_{ij}G_{nj}G_{kk}G_{ji}F''(X)\im(\widetilde{G}_{mn}\widetilde{G}_{km})] \textrm{d}y,
\end{align}
where we abbreviate $\widetilde{G}\equiv G(y+L+\ii\eta_0)$. Applying Lemma \ref{lemma:powercounting0} to the index $n$ and $\widetilde{G}$, we get
\begin{align*}
\frac{1}{N}\sum_{n=1}^{N}|\widetilde{G}_{mn}\widetilde{G}_{kn}| \prec N^{-2/3+2\epsilon},
\end{align*}	
which also shows that
\begin{align}\label{eq:partialF'}
|\partial_{kn}F'(X)|\prec N^{-1/3+C\epsilon}.
\end{align}
Applying Lemma \ref{lemma:powercounting0} to the remaining off-diagonal Green function entries, we obtain that
\begin{align}\label{eq:remainoffdiag}
\frac{q_t^{-1}}{N^2}\sum_{i,j,k,n}|\mathbb{E}[(\partial_{kn}(F'(X))G_{ij}G_{nj}G_{kk}G_{ji})]|\leq q_t^{-1}N^{-2} N^{-1/3+C\epsilon} N^{4}N^{-1+3\epsilon} = q_t^{-1}N^{2/3+C\epsilon}.
\end{align}

If $\partial_{kn}$ acts on $G_{ij}G_{nj}G_{kk}G_{ji}$, then we always get four or more off-diagonal Green function entries with the only exception being
\begin{align*}
-G_{ij}G_{nn}G_{kj}G_{kk}G_{ji}.
\end{align*}
To the terms with four or more off-diagonal Green function entries, we apply Lemma \ref{lemma:powercounting0} and obtain a bound similar to \eqref{eq:remainoffdiag} by power counting. For the term of the exception, we rewrite it as
\begin{align}
&-\frac{q_t^{-1}}{N^2}\sum_{i,j,k,n}\mathbb{E}[F'(X)G_{ij}G_{nn}G_{k j}G_{kk}G_{ji}] = -\frac{q_t^{-1}}{N}\sum_{i,j,k}\mathbb{E}[mF'(X)G_{ij}G_{k j}G_{kk}G_{ji}] \notag \\
&=-\widetilde{m}\frac{q_t^{-1}}{N}\sum_{i,j,k}\mathbb{E}[F'(X)G_{ij}G_{k j}G_{kk}G_{ji}]+\frac{q_t^{-1}}{N}\sum_{i,j,k}\mathbb{E}[(\widetilde{m}-m)F'(X)G_{ij}G_{k j}G_{kk}G_{ji}]
\end{align}
Here, the last term is bounded by $q_t^{-1}N^{2/3+C\epsilon}$ as we can easily check with Proposition \ref{prop:locallaw} and Lemma \ref{lemma:powercounting0}. We thus arrive at
\begin{align}
\frac{q_t^{-1}}{N}(z+\widetilde{m}&)\sum_{i,j,\alpha}\mathbb{E}[F'(X)G_{ij}G_{k j}G_{kk}G_{ji}] \notag \\
&=\frac{q_t^{-1}}{N}\sum_{r'=2}^{\ell}\frac{1}{r'!}\sum_{i,j,k,n}\mathbb{E}[\kappa_{kn}^{(r'+1)}\partial_{kn}^{r'}(F'(X)G_{ij}G_{nj}G_{kk}G_{ji})]+O(N^{2/3-\epsilon'}).
\end{align}
On the right side, the summation is from $r'=2$, hence we have gained a factor $N^{-1}q_t^{-1}$ from $\kappa_{(kn)}^{(r'+1)}$ and added a fresh summation index $k$, so the net gain is $q_t^{-1}$. Since $|z+\widetilde{m}|\sim 1$, this shows that
\begin{align}
\frac{q_t^{-1}}{N}\sum_{i,j,k}\mathbb{E}[F'(X)G_{ij}G_{kj}G_{kk}G_{ji}]=O(N^{2/3-\epsilon'}).
\end{align}
Together with \eqref{eq:4offdiag2}, this takes care of the first term on the right side of \eqref{eq:4offdiag1}.
For the second term on the right side of \eqref{eq:4offdiag1}, we focus on
\begin{align}
\partial_{jk}F'(X) = -\int_{E_1}^{E_2} \sum_{a=1}^N [F''(X)\im (\widetilde{G}_{ja}\widetilde{G}_{ak})]\textrm{d}y
\end{align}
and apply the same argument to the unmatched index $k$ in $\widetilde{G}_{ak}$. For the third term, we focus on $G_{ij}G_{k i}$ and again apply the same argument with the index $k$ in $G_{ki}$.

\subsection{Proof of Lemma \ref{lemma:A_r} for \texorpdfstring{$r=3$}{r=3}}

If $\partial_{jk}$ acts on $F'(X)$ at least once, then that term is bounded by
\begin{align*}
N^{\epsilon}N^{-1}q_t^{-2}N^{3}N^{-1/3+C\epsilon}N^{-2/3+2\epsilon}=q_t^{-2}N^{1+C\epsilon}\ll N^{2/3-\epsilon'},
\end{align*}
where we used \eqref{eq:partialF'} and the fact that $G_{ij}G_{ki}$ or $\partial_{jk}(G_{ij}G_{ki})$ contains at least two off-diagonal entries. Moreover, in the expansion $\partial_{jk}^3(G_{ij}G_{ki})$, the terms with three or more off-diagonal Green function entries can be bounded by
\begin{align*}
N^{\epsilon}N^{-1}q_t^{-2}N^{3}N^{C\epsilon}N^{-1+3\epsilon}=q_t^{-2}N^{1+C\epsilon}\ll N^{2/3-\epsilon'}.
\end{align*}
Thus,
\begin{align}\label{eq:B2main}
&\frac{\mathrm{e}^{-t}q_t^{-2}}{3!N}\sum_{i,j,k}\mathbb{E}[s_{jk}^{(4)}\partial_{jk}^3 F'(X)G_{ij}G_{k i}] \notag \\ \notag
&= \frac{\mathrm{e}^{-t}q_t^{-2}}{3!N}\sum_{i,j,k}\mathbb{E}[s_{d}^{(4)}\partial_{jk}^3 F'(X)G_{ij}G_{k i}] +\frac{\mathrm{e}^{-t}q_t^{-2}}{3!N}\sum_{i,j}\sum_{k\sim j}\mathbb{E}[(s_{s}^{(4)}-s_{d}^{(4)})\partial_{jk}^3 F'(X)G_{ij}G_{k i}] \\ \notag
&=-\frac{4!}{2}\frac{\mathrm{e}^{-t}s_d^{(4)}q_t^{-2}}{3!N}\sum_{i,j,k}\mathbb{E}[F'(X)G_{ij}G_{jj}G_{ji}G_{kk}^2] \notag \\
&\qquad-\frac{4!}{2}\frac{\mathrm{e}^{-t}(s_s^{(4)}-s_d^{(4)})q_t^{-2}}{3!N}\sum_{i,j}\sum_{k\sim j}\mathbb{E}[F'(X)G_{ij}G_{jj}G_{ji}G_{kk}^2]+O(N^{2/3-\epsilon'}),
\end{align}
where $(4!/2)$ is the combinatorial factor. If we split $G_{kk}^2$ into $G_{kk}^2= (G_{kk}- m)^2 +2G_{kk}m - m^2$, then we obtain
\begin{align}
&\frac{\mathrm{e}^{-t}q_t^{-2}}{3!N}\sum_{i,j,k}\mathbb{E}[s_{jk}^{(4)}\partial_{jk}^3 F'(X)G_{ij}G_{k i}] \notag \\
&= -2\mathrm{e}^{-t}s_d^{(4)}q_t^{-2}N^{-1}\sum_{i,j,k}\mathbb{E}[F'(X)G_{ij}G_{jj}G_{ji}((G_{kk}- m)^2 +2G_{kk}m - m^2)]  \notag \\
&\qquad -2\mathrm{e}^{-t}(s_s^{(4)}-s_d^{(4)} )q_t^{-2}N^{-1}\sum_{i,j}\sum_{k\sim j}\mathbb{E}[F'(X)G_{ij}G_{jj}G_{ji}((G_{kk}- m)^2 +2G_{kk}m - m^2)] +O(N^{2/3-\epsilon'}) \notag \\
&= -2\mathrm{e}^{-t}s_d^{(4)}q_t^{-2}\sum_{i,j}\mathbb{E}[F'(X)G_{ij}G_{jj}G_{ji}(O(\psi^2 N^{\epsilon}) + 2m^2 -m^2)]  \notag \\
&\qquad -2\mathrm{e}^{-t}(s_s^{(4)}-s_d^{(4)} )q_t^{-2}K^{-1}\sum_{i,j}\mathbb{E}[F'(X)G_{ij}G_{jj}G_{ji}(O(\psi^2 N^{\epsilon}) + 2m^2 -m^2)] +O(N^{2/3-\epsilon'}) \notag \\
&= -2\mathrm{e}^{-t}s_d^{(4)}q_t^{-2}\sum_{i,j}\mathbb{E}[F'(X)G_{ij}G_{jj}G_{ji}m^2]  \notag \\
&\qquad -2\mathrm{e}^{-t}(s_s^{(4)}-s_d^{(4)} )q_t^{-2}K^{-1}\sum_{i,j}\mathbb{E}[F'(X)G_{ij}G_{jj}G_{ji}m^2] +O(N^{2/3-\epsilon'}) \notag \\
&= -2\mathrm{e}^{-t}\xi^{(4)}q_t^{-2}\sum_{i,j}\mathbb{E}[F'(X)G_{ij}G_{jj}G_{ji}m^2] +O(N^{2/3-\epsilon'}).
\end{align}
where we used \eqref{thm:weak law 2} and Lemma \ref{lemma:powercounting0} for the third equality. Since $m=-1+O(N^{-1/3+\epsilon})$ with high probability, we finally have	 
\begin{align}\label{eq:B2exp1}
\frac{\mathrm{e}^{-t}q_t^{-2}}{3!N}&\sum_{i,j,k}\mathbb{E}[s_{jk}^{(4)}\partial_{jk}^3 F'(X)G_{ij}G_{k i}] \notag \\
&= -2\mathrm{e}^{-t}\xi^{(4)}q_t^{-2}\sum_{i,j}\mathbb{E}[F'(X)G_{ij}G_{jj}G_{ji}m^2] +O(N^{2/3-\epsilon'}) \notag\\
&= -2\mathrm{e}^{-t}\xi^{(4)}q_t^{-2}\sum_{i,j}\mathbb{E}[F'(X)G_{ij}G_{jj}G_{ji}] +O(N^{2/3-\epsilon'}).
\end{align}

Next we consider
\begin{align}\label{eq:B2exp2}
q_t^{-2}\sum_{i,j}\mathbb{E}[zF'(X)G_{ij}G_{jj}G_{ji}]=2 q_t^{-2}\sum_{i,j}\mathbb{E}[F'(X)G_{ij}G_{jj}G_{ji}] +O(N^{2/3-\epsilon'}) .
\end{align}
Expanding the left hand side using the resolvent expansion, we obtain
\begin{align*}
q_t^{-2}\sum_{i,j}\mathbb{E}[zF'(X)G_{ij}G_{jj}G_{ji}]=-q_t^{-2}\sum_{i,j}\mathbb{E}[F'(X)G_{ij}G_{ji}]+q_t^{-2}\sum_{i,j,k}\mathbb{E}[F'(X)H_{jk}G_{ij}G_{kj}G_{ji}].
\end{align*}
Applying Lemma \ref{lemma:Stein} to the second term on the right side, most of the terms are $O(N^{2/3-\epsilon'})$ either due to three (or more) off-diagonal entries, the partial derivative $\partial_{jk}$ acting on $F'(X)$, or higher cumulants. Thus we find that
\begin{align*}
-q_t^{-2}\sum_{i,j,k} \mathbb{E}[\kappa_{jk}^{(2)}F'(X)G_{ij}G_{kk}G_{jj}G_{ji}] = -q_t^{-2}\sum_{i,j} \mathbb{E}[F'(X)G_{ij}G_{jj}G_{ji}m]
\end{align*}
is the only non-negligible term, which is generated when $\partial_{jk}$ acts on $G_{kj}$. From this argument we obtain
\begin{align*}
q_t^{-2}\sum_{i,j}\mathbb{E}[zF'(X)G_{ij}G_{jj}G_{ji}]=-&q_t^{-2}\sum_{i,j}\mathbb{E}[F'(X)G_{ij}G_{ji}] \\&-q_t^{-2}\sum_{i,j}\mathbb{E}[mF'(X)G_{ij}G_{jj}G_{ji}] +O(N^{2/3-\epsilon'}).
\end{align*}
Combining with \eqref{eq:B2exp2} and the fact that $m(z)=-1+O(N^{-1/3+\epsilon})$ with high probability, we get
\begin{align}\label{eq:B2term1}
q_t^{-2}\sum_{i,j}\mathbb{E}[F'(X)G_{ij}G_{jj}G_{ji}]=-q_t^{-2}\sum_{i,j}\mathbb{E}[F'(X)G_{ij}G_{ji}]+O(N^{2/3-\epsilon'}).
\end{align}
By combining \eqref{eq:B2exp1} and \eqref{eq:B2term1}, we conclude that
\begin{align}
\frac{\mathrm{e}^{-t}q_t^{-2}}{3!N}\sum_{i,j,k}\mathbb{E}[s^{(4)}_{jk}\partial_{jk}^3 F'(Y)G_{ij}G_{ki}] &=2\mathrm{e}^{-t}\xi^{(4)}q_t^{-2}\mathbb{E}[ F'(X)G_{ij}G_{j i}] + O(N^{2/3-\epsilon'}).
\end{align}

\subsection{Proof of Lemma \ref{lemma:A_r} for \texorpdfstring{$r=4$}{r=4}} We estimate the term by using similar argunment as in the case $r=2$ and one can get
\begin{align}
\frac{q_t^{-3}}{N}\sum_{i,j,k} \big|\mathbb{E}\big[\partial_{jk}^4 (F'(Y)G_{ij}G_{ki})\big]\big| =O(N^{2/3-\epsilon'}). 
\end{align}	
We leave the details to the interested reader.

\section*{Acknowledgements}
This work was partially supported by the Samsung Science and Technology Foundation project number SSTF-BA1402-04.

\end{document}